\documentclass[a4paper]{article}
\usepackage{lmodern}
\usepackage{stmaryrd,amsmath,amssymb,ifthen,mathrsfs,hyperref}
\usepackage{amsthm}
\newtheorem{prop}{Proposition}
\newtheorem{thm}[prop]{Theorem}
\newtheorem{cor}[prop]{Corollary}
\newtheorem{lem}[prop]{Lemma}

\theoremstyle{definition}

\newtheorem{example}[prop]{Example}

\theoremstyle{remark}

\newcommand{\Q}{\mathbb{Q}}

\newcommand{\Z}{\mathbb{Z}}

\newcommand{\F}{\mathbb{F}}

\newcommand{\GL}{\operatorname{GL}}
\newcommand{\AGL}{\operatorname{AGL}}
\newcommand{\PGL}{\operatorname{PGL}}
\newcommand{\SL}{\operatorname{SL}}
\newcommand{\PSL}{\operatorname{PSL}}
\newcommand{\Sp}{\operatorname{Sp}}
\newcommand{\PSp}{\operatorname{PSp}}
\newcommand{\PGU}{\operatorname{PGU}}

\newcommand{\PSU}{\operatorname{PSU}}

\newcommand{\SO}{\operatorname{SO}}
\newcommand{\POmega}{\operatorname{P\Omega}}

\newcommand{\Alt}{\operatorname{Alt}}
\newcommand{\Aut}{\operatorname{Aut}}
\newcommand{\Out}{\operatorname{Out}}
\newcommand{\Sym}{\operatorname{Sym}}
\newcommand{\Hom}{\operatorname{Hom}}
\newcommand{\Ext}{\operatorname{Ext}}
\newcommand{\soc}{\operatorname{soc}}

\newcommand{\gen}[1]{\langle #1\rangle}
\numberwithin{prop}{section}
\numberwithin{equation}{section}
\numberwithin{table}{section}

\newcommand{\mb}[1]{\mathbf{#1}}
\newcommand{\mbG}{\mathbf{G}}
\title{On the Maximal Subgroups of $E_7(q)$ and Related Almost Simple Groups}
\author{David A.\ Craven}
\date{\today}

\begin{document}
\maketitle

\setcounter{tocdepth}{3}
\tableofcontents

\section{Introduction}

The Classification of Finite Simple Groups (CFSG) is one of the crowning achievements of twentieth-century mathematics, allowing us to reduce many problems to (groups related to) simple groups. But on its own, CFSG just gives us a list of simple groups: to apply it to problems, we usually need a significant amount of structural information about simple groups. One of the major open questions in this area is to compute the maximal subgroups of the finite simple groups (and their almost simple versions). For classical groups one may expect lists in small dimensions and general results for arbitrary dimensions; for exceptional groups of Lie type, complete tables can in theory be produced. This was done decades ago for the small-rank cases of Suzuki and Ree groups, $G_2(q)$ and ${}^3\!D_4(q)$ (see, for example, \cite{wilsonrob}). The cases of $F_4(q)$, $E_6(q)$ and ${}^2\!E_6(q)$ were recently completed by the author \cite{craven2020un}. This leaves $E_7(q)$ and $E_8(q)$.

This paper is a contribution to tackling the case of $E_7(q)$, leaving only a small handful of ambiguities, at most two isomorphism types of subgroups (up to taking normalizers) for any given group (Table \ref{t:e7stillleft}). As a result of a number of papers by various authors (see Section \ref{sec:othermaximals} for full details, or later in this introduction for a shorter summary), all maximal subgroups $M$ of an almost simple group with socle $E_7(q)$ are known except for $M$ almost simple, in some sense the most difficult case.

\begin{thm}\label{thm:mainthm} Let $\bar G$ be an almost simple group with socle $G=E_7(q)$, for $q$ a power of a prime $p$. Let $M$ be a maximal subgroup of $\bar G$, and assume that $M$ does not contain $G$. One of the following holds:
\begin{enumerate}
\item $M=N_{\bar G}(M_0)$ is the normalizer of the intersection $M_0$ of $G$ with a positive-dimensional subgroup of an adjoint algebraic group of type $E_7$, and is given in Table \ref{t:othermaximals};
\item $M$ is an almost simple group and appears in Table \ref{t:themaximals};
\item $M=N_{\bar G}(M_0)$ is the normalizer of one of the groups $M_0=\PSL_2(r)$ for $r=7,8,9$, with possibilities for $p$ given in Table \ref{t:e7stillleft}.
\end{enumerate}
In cases (i) and (ii), all entries in the tables exist, and numbers of classes are given. In case (iii), no maximal example is known, but one cannot be ruled out.
\end{thm}

Table \ref{t:e7stillleft} has no entry for $p=3$, yielding the following corollary to Theorem \ref{thm:mainthm}.

\begin{cor} If $q$ is a power of $3$ then all maximal subgroups of almost simple groups with socle $E_7(q)$ are known, and appear in Tables \ref{t:themaximals} and \ref{t:othermaximals}.
\end{cor}

For $p=2$, the only entry in Table \ref{t:e7stillleft} is $\PSL_2(8)$, and for this subgroup we are able to eliminate it for $E_7(2)$ (already done in \cite{bbr2015}, but our proofs here are independent of their paper), and for $E_7(4)$ (which is new). See Proposition \ref{prop:psl28q=4} for the case of $\PSL_2(8)$ and $q=2,4$.

\begin{prop} The maximal subgroups of $E_7(4)$ are known, and appear in Tables \ref{t:themaximals} and \ref{t:othermaximals}.
\end{prop}

%
%\begin{table}
%\begin{center}
%\begin{tabular}{cc}
%\hline Prime & Group
%\\\hline All primes & $\PSL_2(q)$, $q=\b{19},27,29,\b{37}$, $\PSL_3(4)$, $\PSU_3(3)$
%\\ $p\nmid |H|$ & $\Alt(6)$, $\PSL_2(q)$, $q=7,11,13$
%\\ $p=3$ & $\PSL_2(q)$, $q=11,13$
%\\ $p=5$ & $\Alt(6)$, $\Alt(7)$, ${}^2\!B_2(8)$
%\\ $p=7$ & $\PSL_2(13)$
%\\ \hline
%\end{tabular}
%\end{center}
%\caption{Simple groups of $E_7$ whose status is still to be determined. Bold indicates embeddings of $\PSL_2(h+1)$ and $\PSL_2(2h+1)$ (i.e., Serre and Kostant embeddings).}
%\label{t:e7after}
%\end{table}

\begin{table}
\begin{center}
\begin{tabular}{ccccc}
\hline Group & $p$ & $q$ & No. classes & Stabilizer
\\ \hline %$\Alt(7)$ & $5$ & $25$ & $2$ & $\gen\phi$
%\\
$M_{12}$ & $5$ & $5$ & $1$ & $\gen\delta$
\\ $HS$ & $5$ & $5$ & $2$ & $1$
\\ $Ru$ & $5$ & $5$ & $2$ & $1$
\\ $\PSU_3(3)$ & $p\equiv \pm 1\bmod 8$ & $p$ & $2$ & $1$
\\ $\PSU_3(3)$ & $p\neq 3$, $p\equiv \pm 3\bmod 8$ & $p^2$ & $2$ & $\gen\phi$
\\ $\PSU_3(8).6$ & $p\equiv \pm 1\bmod 8$ & $p$ & $2$ & $1$
\\ $\PSU_3(8).3_1$ & $p\equiv \pm 3\bmod 8$ & $p$ & $1$ & $\gen\delta$
\\ $\PSL_2(37)$ & $p\neq 37$ a square modulo $37$ & $p$ & $2$ & $1$
\\ $\PSL_2(37)$ & $p$ a non-square modulo $37$ & $p^2$ & $d$ & $\gen\phi$
\\ $\PSL_2(29)$ & $p\equiv \pm 1\bmod 5$, $p\neq 29$ a square modulo $29$ & $p$ & $4$ & $1$
\\ $\PSL_2(29)$ & $p\equiv \pm 1\bmod 5$, $p$ a non-square modulo $29$ & $p^2$ & $4$ & $\gen\phi$
\\ $\PSL_2(29)$ & $p\equiv \pm 2\bmod 5$ & $p^2$ & $2d$ & $1$
\\ $\PSL_2(27).3$ & $13$ & $13$ & $2$ & $1$
\\ $\PSL_2(27)$ & $p\equiv \pm 1\bmod 13$ & $p$ & $4$ & $1$
\\ $\PSL_2(27)$ & $p\equiv \pm 5\bmod 13$ & $p^2$ & $4$ & $1$
\\ $\PSL_2(27)$ & $p\equiv \pm 3,\pm 4\bmod 13$ & $p^3$ & $4$ & $\gen\phi$
\\ $\PSL_2(27)$ & $p\equiv \pm2,\pm6\bmod 13$ & $p^6$ & $2d$ & $\gen{\phi^2}$
\\ $\PSL_2(19)$ & $5$ & $5$ & $2$ & $1$
\\ $\PSL_2(19)$ & $p\neq 19$, $p\equiv \pm 1\bmod 5$, $p_{19}\equiv p_3\bmod 2$ & $p$ & $4$ & $1$
\\ $\PSL_2(19)$ & $p\equiv \pm 1\bmod 5$, $p_{19}\not\equiv p_3\bmod 2$ & $p^2$ & $4$ & $\gen\phi$
\\ $\PSL_2(19)$ & $p\neq 3$, $p\equiv \pm 2\bmod 5$ & $p^2$ & $2d$ & $1$
\\ $\PGL_2(19)$ & $5$ & $5$ & $2$ & $1$
\\ $\PGL_2(19)$ & $p\equiv \pm 1,\pm 9\bmod 40$ & $p$ & $4$ & $1$
\\ $\PSL_2(19)$ & $p\neq 19$, $p\equiv \pm 11,\pm 19\bmod 40$ & $p$ & $2$ & $\gen\delta$
\\ $\PGL_2(19)$ & $p\neq 2$, $p\equiv \pm 2\bmod 5$ & $p^2$ & $4$ & $1$
\\ $\PGL_2(13)$ & $p\equiv \pm 1\bmod 8$ & $p$ & $2$ & $1$
\\ $\PSL_2(13)$ & $p\neq 13$, $p\equiv \pm 5, \pm 11, \pm 19, \pm 21, \pm 37, \pm 45, \pm 59\bmod 104$ & $p$ & $1$ & $\gen\delta$

\\\hline $M_{22}$ & $5$ & $5$ & $1$ & $\gen\delta$
\\ $\PSL_2(13)$ & $p\neq 13$, $p\equiv \pm 3, \pm 27, \pm29, \pm 35, \pm 43, \pm 51, \pm 53\bmod 104$ & $p$ & $1$ & $\gen\delta$
\\ \hline
\end{tabular}
\end{center}
\caption{The known maximal subgroups of the simple group $G=E_7(q)$, for $q$ a power of $p$, other than those that are fixed points of a positive-dimensional subgroup of the algebraic group. $M_{22}.2$ and $\PSL_2(13).2$ at the bottom of the table are novelty maximal subgroups that occur in the group $E_7(p).2$. Here, $d=\gcd(2,p-1)$, $p_n$ is the order of $p$ modulo $n$, $\phi$ is a generator for the group of field automorphisms and $\delta$ is a generator for the group of outer diagonal automorphisms of $G$.}
\label{t:themaximals}
\end{table}

\begin{table}
\begin{center}
\begin{tabular}{cc}
\hline Prime & Group $M_0$
\\\hline $p=2$ & $\PSL_2(8)$
\\ $p\geq 5$ & $\PSL_2(7),\PSL_2(9)$
\\ \hline
\end{tabular}
\end{center}
\caption{Potential simple subgroups $M_0$ of $G=E_7(q)$ ($q$ a power of $p$) with $N_{\bar G}(M_0)$ maximal that have not been eliminated in this document.}
\label{t:e7stillleft}
\end{table}

Making headway with the remaining open cases is possible, but for each group there are significant issues with using the methods from this paper.

\bigskip

We now give a brief overview of the results prior to this paper, to set this work in the appropriate context. For simplicity we just consider the simple group, but all results are proved for almost simple groups. Let $G=E_7(q)$ be simple, where $q$ is a power of $p$. For odd primes $p$ this is not the fixed points of an algebraic group $\mbG$ under a Frobenius morphism, but the extension $G.2$ is. If $M$ is a maximal subgroup of $G.2$ other than $G$ then either $M$ is the fixed points of a maximal subgroup of $\mbG$ or it is not.  In \cite{borovik1989} and independently in \cite{liebeckseitz1990} (see Theorem \ref{thm:borliesei} below), if $M$ is not the fixed points then $M$ could normalize an $r$-subgroup, where $r$ is a prime different from $p$, or $M$ could be a `subfield subgroup' $E_7(q_0)\leq E_7(q)$, or $M$ is almost simple.\footnote{Notice that the fixed points are only given for the extended group $G.2$, not for $G$. Indeed, the intersection of the fixed points with $G$ is not known for the $A_2$ subgroup of $E_7$, see Section \ref{sec:othermaximals}.}

Since all simple groups are known thanks to CFSG, we can run through the list, determining if a simple group $X$ embeds in $G$, and if so if it can be maximal. This is not easy. If $M$ is a Lie type group also in characteristic $p$ then restrictions on $M$ were proved by Liebeck and Seitz in \cite{liebeckseitz1990,liebeckseitz1998,liebeckseitz2004}, but there were a number of possibilities, particularly for $M\cong \PSL_2(p^a)$. In \cite{craven2017} the author eliminated all possibilities for $M\cong \PSL_2(p^a)$ except for $p^a=7,8,25$, and in \cite{craven2019un} the author eliminated all possibilities for $M\not\cong\PSL_2(p^a)$, so only these three cases remain. For $p^a=7,8$, these remain unresolved here, but $p^a=25$ is proved not to yield maximal subgroups in Section \ref{sec:psl225}.

If $M$ is not a Lie type group in characteristic $p$ then the complete list of options can be found in \cite{liebeckseitz1999}, and considerable work was done on them in \cite{litterickmemoir}. Individual papers had constructed some subgroups, for example \cite{griessryba2002}, but there had been no concerted effort to determine \emph{all} maximal subgroups. Most papers proved that the subgroups exist, and sometimes counted the number of classes in $\mbG$, but not whether they exist in $G$ and how many classes there are.

This paper performs this task, for the first time systematically working through the list from \cite{liebeckseitz1999} and determining if $M$ is maximal, and if so how many classes there are. The result is Table \ref{t:themaximals}, which includes several never-before-seen maximal subgroups. The remaining open questions, which are the groups in Table \ref{t:e7stillleft} and one ambiguity in Table \ref{t:othermaximals}

\bigskip

It should also be noted that in the preparation of this manuscript an error in one of the author's previous papers came to light. A result on when subgroups can be placed in connected positive-dimensional subgroups, rather than arbitrary positive-dimensional subgroups, was given in \cite[Lemma 1.12]{craven2017}. The version there incorrectly rules out one possibility because the structure of a particular subgroup of $E_7$ was wrongly stated. In Section \ref{sec:intrinsicimp} below we rework the proof, and in Corollary \ref{cor:intrinsicimp} we give a corrected, and improved, version of the result. Note that it is not known whether the result from \cite{craven2017} is false, just that the proof was false.

\bigskip

We start with a few sections of preliminary results, and then in Section \ref{sec:methods} discuss the two computational methods that will be implemented in determining most of the subgroups in Table \ref{t:themaximals}. The determination of the subgroups in Table \ref{t:othermaximals} is known from previous papers (described in Section \ref{sec:othermaximals}), except for the technical question of whether the outer diagonal automorphism fuses two classes of maximal subgroups or normalizes a single class. This is also known for some of the groups already, and for others we compute the answer here. For one set of groups, $\PGL_3(q)$ and $\PGU_3(q)$ for $q$ an odd power of a prime, we cannot determine here whether this has index $1$ or $2$ in its normalizer in $E_7(q)$, (i.e., whether the diagonal automorphism normalizes these classes) and we leave this case open. We do this analysis for the other groups in Section \ref{sec:othermaximals}.

The determination of the groups in $\mathcal S$ (except for $\PSL_2(r)$, $r=7,8,9$, which we do not complete here) takes place in Section \ref{sec:detofgroups}, starting from the list in Table \ref{t:e7tocheck}. Information about the four remaining groups is given in Section \ref{sec:leftovers}. Finally, Section \ref{sec:proofofmaximality} proves that the subgroups in Table \ref{t:themaximals} are genuinely maximal by checking there are no proper subgroups of $E_7(q)$ containing them.

Note that in the supplementary materials we explicitly construct all of the groups in Table \ref{t:themaximals}, so this offers an independent verification of the existence.

\bigskip

The author would like to thank Alex Ryba profusely, who while at the Isaac Newton Institute in January 2020 explained his ideas, from \cite{griessryba2002}, which form the subalgebra method from Section \ref{sec:subalgmethod}. He would also like to thank Friedrich Knop for his proof of Proposition \ref{prop:abeliansubalgebra}. Finally, he would like to thank Ben Grossmann for discussions on the StackExchange website that led to producing a short algorithm for determining ranks of linear combinations of two matrices, which is used in several of the computer programs for this article.

\section{Preliminary results}

Readers of previous papers in this series \cite{craven2015un,craven2019un,craven2020un} will be familiar with this notation and setup. Many of the preliminary results in this section will also have been come across before. There are some simplifications in this paper versus previous ones because we no longer have to consider Steinberg endomorphisms that are not Frobenius, and will not need to consider the action of graph automorphisms on maximal subgroups. Chief among the new actors in this paper is the alternating form on the $56$-dimensional minimal module for $E_7$, which has not been needed before.

\subsection{Basic notation and definitions}

Let $p$ be a prime or $0$ and let $k$ be an algebraically closed field of characteristic $p$. Let $\mb G$ be a simple, simply connected algebraic group defined over $k$, and let $\sigma$ be a Frobenius endomorphism of $\mb G$. Write $G_{\mathrm{sc}}$ for the fixed points $\mb G^\sigma$, and $G$ for $G_{\mathrm{sc}}/Z(G_{\mathrm{sc}})$. Write $F_p$ for a generator of the semigroup of `standard' Frobenius endomorphisms, i.e., with trivial action on the Weyl group. Write $E_7(q)$ for the simple group of this type when $\mbG$ has type $E_7$ and $\sigma=F_p^a$, where $q=p^a$. (Thus if $\sigma=F_p$ for $p$ odd then $G_{\mathrm{sc}}$ is the simply connected group $2\cdot E_7(p)$.) Let $\bar G$ be an almost simple group with socle $G$. Write $G_{\mathrm{ad}}$ for the fixed points under a Frobenius morphism of the adjoint version of $\mbG$ compatible with $\sigma$. Thus $G_{\mathrm{ad}}$ has $G$ as a normal subgroup and the quotient by $G$ is the group of outer diagonal automorphisms.

Let $M(\mb G)$ denote the non-trivial Weyl module of smallest dimension, and $L(\mbG)$ denote the adjoint module of $\mbG$. We will endow $L(\mbG)$ with a Lie bracket (this is crucial in this article) that turns it into a Lie algebra. If $L(\mbG)$ has a single trivial composition factor and a single non-trivial one (so the important one for us is $\mbG$ of type $E_7$ and $p=2$) then write $L(\mbG)^\circ$ for the non-trivial composition factor.

If $k$ has characteristic $p\neq 0$ and $u\in \GL_n(k)$ has $p$-power order then $u$ is a unipotent element, and acts (up to conjugacy) as a sum of Jordan blocks of various sizes. We write $n_1^{a_1},\dots,n_r^{a_r}$ to describe the sizes, for example $5^2,4,1$ for an element of $\GL_{15}(k)$, as in \cite{lawther1995}. We also use the labelling of unipotent classes from that paper.

Following \cite{liebeckmartinshalev2005} (and in \cite{craven2015un,craven2019un,craven2020un}), write $\Aut^+(\mbG)$ for the set of inner, graph, and $p$-power field automorphisms of $\mbG$. Since we do not consider $G_2$ for $p=3$ and $F_4$ for $p=2$, we do not need to consider the very twisted groups. For $\mb G$ of type $E_7$, $\Aut^+(\mbG)$ induces the whole of $\Aut(G)$ on $G$. A subgroup $H$ of $\mbG$ is \emph{Lie primitive} if there is no closed, proper, positive-dimensional subgroup of $\mbG$ containing $H$, and Lie imprimitive otherwise. While Lie primitive subgroups often yield maximal subgroups of the finite groups $G$, there may be others. A subgroup $H$ of $\mbG$ is \emph{strongly imprimitive} if it is contained in a closed, proper, positive-dimensional subgroup $\mb X$, and in addition $\mb X$ may be chosen to be $N_{\Aut^+(\mbG)}(H)$-stable. In this case, the normalizer of $\bar H=HZ(\mbG)/Z(\mbG)$ in $\bar G$ is always contained in the normalizer of $\mb X^\sigma\; Z(\mbG)/Z(\mbG)$, and so the only strongly imprimitive maximal subgroups are (normalizers of) fixed points of Frobenius endomorphisms of positive-dimensional subgroups of $\mbG$, which are known (see Section \ref{sec:othermaximals}). Write $\mathscr X$ for the set of proper, closed, positive-dimensional subgroups of $\mbG$. We will write $\bar H$ for the image of $H\leq G_{\mathrm{sc}}$ in $G$. However, if $Z(\mbG)\cap H=1$ then $H\cong \bar H$ and in this case, to reduce notation, we identify $H$ with its image in $\mbG/Z(\mbG)$.

If $M$ is a $k\mbG$-module for $\mbG$ of type $E_7$, and $H$ is a finite subgroup of $\mbG$, then $H$ is a \emph{blueprint} for $M$ if there is a closed, positive-dimensional subgroup of $\mbG$ stabilizing exactly the same subspaces of $M$ stabilized by $H$. If $H$ is a blueprint for a module $M$ then $H$ is either strongly imprimitive or stabilizes the same subspaces of $M$ as $\mbG$ itself \cite[Corollary 3.3]{craven2019un}.

\medskip

Our notation for modules and their structures is as in previous papers in this series \cite{craven2015un,craven2019un,craven2020un}. Let $H$ be a finite group. We will generally label simple $kH$-modules by their dimension, for example $10$ for a $10$-dimensional module. If there are two $10$-dimensional modules that are dual to one another, write $10$ and $10^*$ for the two modules, and if they are self-dual write $10_1$ and $10_2$. If $M$ is a $kH$-module, write $P(M)$ for its projective cover. We use `$/$' to distinguish between socle layers of a module, so that a module $1,5/10$ has socle $10$ and second socle $1\oplus 5$. If $L$ is a subgroup of $H$, write $M{\downarrow_L}$ for the restriction of the $kH$-module $M$ to $L$. Write $S^i(M)$ and $\Lambda^i(M)$ for the $i$th symmetric and exterior power of $M$ respectively.

For algebraic groups we use the `standard' notation for simple highest-weight modules $L(\lambda)$ for $\lambda$ a dominant weight. The labelling of such modules follows \cite{bourbakilie2}, and is the one used by the references in our paper, and those in previous papers in this series. We often abbreviate $L(\lambda)$ to just $\lambda$, particularly if we have a product of simple groups, so $(10,\lambda_1)$ is the label for the product of the natural modules of $A_2A_5$, for example. The composition factors of maximal reductive subgroups and Levi subgroups of exceptional algebraic groups on $M(\mbG)$ and $L(\mbG)$ can be found in \cite{thomas2016}, for instance.

\medskip

The possible sets of composition factors for a simple group $H$, not of Lie type in characteristic $p$, on the modules $M(E_7)$ and $L(E_7)$, are collated in \cite[Section 6.3]{litterickmemoir}. One of the main results of \cite{litterickmemoir} states that, unless a row of a table is labelled with `\textbf{P}' (the first is in Table 6.114), every subgroup with those composition factors must be strongly imprimitive. This yields a significant reduction in the number of cases that need to be considered in this paper. In effect, we must go through the tables of \cite{litterickmemoir} and determine, for each row labelled `\textbf{P}', whether such a subgroup can be maximal in $G_\mathrm{sc}$ and, if so, how many conjugacy classes there are, and understand the action of $\Out(G)$ on these classes.

If $M_0$ is a subgroup of $G$ such that $N_G(M_0)$ is not maximal then it could be that $M_1=N_{\bar G}(M_0)$ is maximal in $\bar G$; such maximal subgroups are called \emph{novelty maximal subgroups}. Denoting by $A$ the set of outer automorphisms induced by $\bar G$ on $G$, $M_1$ is a novelty maximal subgroup if $H$ is $A$-stable but no intermediate group $H<X<G$ is $A$-stable. A thorough description of such maximal subgroups is given in \cite[Section 3.5]{craven2020un} (largely taken from \cite[Section 1.3.1]{bhrd}), but we do not need this here as, apart from $M_{22}$, all subgroups are either strongly imprimitive or produce `standard' maximal subgroups.

\subsection{Results about algebraic groups}

This section contains various results about algebraic groups that will be of use in what follows. We start with two results: the first, by Larsen \cite[Theorem A.12]{griessryba1998}, establishes that the number of conjugacy classes of Lie primitive subgroups isomorphic to $H$ of an algebraic group $\mb G$ is independent of the characteristic $p$ of $\mb G$, at least so long as $p\nmid |H|$. (This result is often referred to as `Larsen's $(0,p)$-correspondence'.)

\begin{thm}\label{thm:larsencorr} Let $H$ be a finite group such that $p\nmid |H|$. If $\mbG$ is a semisimple split group scheme, and $k$ and $k'$ are algebraically closed fields of characteristic $p$ and $0$ respectively, the number of $\mbG(k)$- and $\mbG(k')$-conjugacy classes of subgroups isomorphic to $H$ are the same finite quantity.
\end{thm}

The second extends this in one direction to all $p$, including $p\mid |H|$. It is no longer true that the number of classes stays the same; there are examples of two Lie primitive classes becoming one (Proposition \ref{prop:psl227} below), and also of Lie primitive subgroups becoming Lie imprimitive ($\PSL_2(8)$ in $F_4$, see \cite[Proposition 4.9]{craven2020un}). But by results of Serre \cite[Section 5]{serre1996} at least there is always an embedding with the reduction modulo $p$ of the character from characteristic $0$. We require that $H$ is an almost simple subgroup purely to avoid complications with normal $p$-subgroups, and it is no loss for us as all of our subgroups are almost simple. The version here is proved in \cite[Theorem 3.10]{craven2020un}.

\begin{prop}\label{prop:p=0impallp} Let $\mb G_0$ be the algebraic group defined over $\mathbb C$ of the same type as $\mb G$. Let $H_0$ be a subgroup of $\mb G_0$ such that $O_p(H_0)=1$, embedding with character $\chi_1,\dots,\chi_n$ on $p$-restricted highest weight modules $L_1,\dots,L_n$ respectively. There exists a subgroup $H$ of $\mb G$ such that $H\cong H_0$, and with Brauer character the reduction modulo $p$ of $\chi_i$ for each corresponding $L_i$.
\end{prop}

The next result is a generalization of an old result of Mal'cev \cite{malcev1945} to all primes due to Friedrich Knop.\footnote{This was posted in response to a question asked by the author on the MathOverflow website.} We only need the case of $\mathfrak e_7$ here, but it is of independent and general interest; for example, the author will use the result for $\mathfrak e_8$ in a future paper.
\begin{table}
\begin{center}
\begin{tabular}{cccc}
\hline Lie algebra & Dimension & Lie algebra & Dimension
\\ \hline $\mathfrak a_n$ ($n\geq 1$) & $\lfloor(n+1)^2/4\rfloor+\epsilon_n$ & $\mathfrak g_2$ & $3$
\\ $\mathfrak b_n$ ($n\geq 4$) & $n(n-1)/2+1$ & $\mathfrak f_4$ & $9$
\\ $\mathfrak b_3$ & $5$ & $\mathfrak e_6$ & $16+\delta_{p,3}$
\\ $\mathfrak c_n$ ($n\geq 2$) & $n(n+1)/2$ & $\mathfrak e_7$ & $27+\delta_{p,2}$
\\ $\mathfrak d_n$ ($n\geq 4$) & $n(n-1)/2+2\delta_{p,2}$ & $\mathfrak e_8$ & $36$
\\ \hline
\end{tabular}
\end{center}
\caption{Maximal dimensions of abelian subalgebras of simple Lie algebras in characteristic $p$. Here $\epsilon_n=0$ if $p\nmid (n+1)$ and $\epsilon_n=1$ if $p\mid (n+1)$. Also, $p\neq 2$ for types $\mathfrak b_n$, $\mathfrak c_n$ and $\mathfrak f_4$, and $p\neq 2,3$ for $\mathfrak g_2$.}\label{tab:maxdimension}
\end{table}
%\begin{table}
%\begin{center}
%\begin{tabular}{cc}
%\hline Group & Dimension
%\\ \hline $A_n$ ($n\geq 1$) & $\floor{(n+1)^2/4}+\epsilon_n$
%\\ $B_n$ ($n\geq 4$) & $n(n-1)/2+1$
%\\ $B_3$ & $5$
%\\ $C_n$ ($n\geq 2$) & $n(n+1)/2$
%\\ $D_n$ ($n\geq 4$) & $n(n-1)/2+2\delta_{p,2}$
%\\ $G_2$ & $3$
%\\ $F_4$ & $9$
%\\ $E_6$ & $16+\delta_{p,3}$
%\\ $E_7$ & $27+\delta_{p,2}$
%\\ $E_8$ & $36$
%\\ \hline
%\end{tabular}
%\end{center}
%\caption{Maximal dimensions of abelian subalgebras of simple Lie algebras in characteristic $p$. Here $\epsilon_n=0$ if $p\nmid (n+1)$ and $\epsilon_n=1$ if $p\mid (n+1)$. Also, $p\neq 2$ for types $B_n$, $C_n$ and $F_4$, and $p\neq 2,3$ for $G_2$.}\label{tab:maxdimension}
%\end{table}
\begin{prop}\label{prop:abeliansubalgebra} Let $\mb G$ be a simple exceptional algebraic group in characteristic $p$, and let $\mathfrak g$ be its Lie algebra. Suppose that $p\neq 2$ for $\mb G$ of types $B_n$, $C_n$ and $F_4$, and $p\neq 2,3$ for $\mb G$ of type $G_2$. The maximal dimension of an abelian subalgebra is given in Table \ref{tab:maxdimension}.
\end{prop}
\begin{proof} In the Grassmannian of $\mathfrak g$, the condition of being abelian is a closed condition. Thus we have a closed subscheme $A_d\subset \mathrm{Gr}_d(\mathfrak g)$ classifying all abelian subalgebras of a given dimension $d$, and $\mb G$ acts on $A_d$ by conjugation. Let $d$ be such that $A_d$ is non-empty. As $A_d$ is a projective variety, any Borel subgroup $\mb B$ of $\mb G$ has a fixed point on $A_d$. Thus, if there exists an abelian subalgebra of $\mathfrak g$ of dimension $d$ then there exists a $\mb B$-stable such subalgebra $\mathfrak a$. If $\mb T$ is a maximal torus inside $\mb B$ then $\mathfrak a$ is also $\mb T$-stable. Thus $\mathfrak a$ has a root space decomposition
\[ \mathfrak a=\mathfrak a_0\oplus \bigoplus_{\alpha\in R} \mathfrak g_\alpha,\]
where $R$ is a set of roots and $\mathfrak a_0$ is a subalgebra of the corresponding maximal torus $\mathfrak t$.

Since $\mathfrak a_0$ is the set of semisimple elements in $\mathfrak a$, it is normalized by $\mb B$, so consists of $\mb B$-fixed, and thus $\mb G$-fixed, points. Thus $\mathfrak a_0$ is central, and is the contribution of $\delta_{p,2}$ and so on to the entries in Table \ref{tab:maxdimension}. Thus it suffices to determine the maximum size of $R$. We show that this is independent of $p$, and thus we may assume that $p=0$. Then these values coincide with those of Mal'cev in \cite{malcev1945}. (In the English summary of this paper---but not in the Russian original---it is incorrectly stated that the answer for $B_3$ is $6$, whereas it is $5$.)

Note that, for all $\alpha,\beta\in R$, we have that $[\mathfrak g_\alpha,\mathfrak g_\beta]=0$, since $\mathfrak a$ is abelian. If $\alpha+\beta$ is a root then, for Chevalley generators $e_\alpha,e_\beta$, we have that
\[ [e_\alpha,e_\beta]=\pm N_{\alpha\beta}e_{\alpha+\beta},\]
and this is non-zero because of the restrictions on $p$. Thus $\alpha+\beta$ cannot be a root for all $\alpha,\beta\in R$. If $\alpha+\beta=0$ then for $p>2$ we have that $[e_\alpha,e_{-\alpha}]=h_\alpha\neq 0$, another contradiction.  Thus for $p$ odd, $d$ is the sum of a maximal set of roots the sum of any two of which is neither a root nor zero, which is clearly independent of $p$.

If $p=2$ then our condition on $\mb G$ means that the Dynkin diagram is simply laced, and therefore the root spaces have dimension exactly $2$ (spanned by $\alpha$ and `$-\alpha$', which is of course equal to $\alpha$ in characteristic $2$). We cannot have that $g_\alpha$ has dimension $2$ though, since the $2$-dimensional root space is not abelian (it commutates into a Cartan subalgebra), and thus $g_\alpha$ is $1$-dimensional in this case as well, spanned by a root. Hence $|R|$ is bounded by the same value as for $p\neq 2$, and the proof is complete.
\end{proof}

The next result comes from \cite[Proposition 3.11]{litterickmemoir}. This version is good enough for our needs here (we only need it in Section \ref{sec:proofofmaximality} to reduce the number of potential overgroups of a putative maximal subgroup), but we give an improved version of it in Section \ref{sec:intrinsicimp}.

\begin{lem}\label{lem:inconnected} Let $\mbG$ be the group $E_7$, and let $H$ be a subgroup of $\mbG$ such that $HZ(\mbG)/Z(\mbG)$ is simple. Either $H$ is Lie primitive, $H$ is contained in a proper, closed, connected, positive-dimensional subgroup of $\mbG$, or $H$ is isomorphic to one of $\PSp_6(2)$, $\PSU_3(3)$, $\PSL_2(7)$ and $\PSL_2(8)$.
\end{lem}

The point of this lemma is to move from $H$ being contained in a proper subgroup to being contained in a proper \emph{connected} subgroup. If $H$ is isomorphic to a subgroup of the Weyl group then that is not obviously always the case.

\subsection{Stabilizing subspaces}

We begin with \cite[(1.3)]{seitz1991}.

\begin{lem}\label{lem:paramaxl} Let $H$ be a subgroup of $\mbG$ and suppose that $H$ stabilizes a line on $L(E_7)^\circ$. Then $H$ is contained in a parabolic or maximal-rank subgroup.
\end{lem}

This does not tell us exactly what we want to know though, because an almost simple group need not stabilize a line if the simple subgroup does. This lemma is \cite[Propositions 4.5 and 4.6]{craven2015un}, specialized to the case $\mb G=E_7$.

\begin{lem}\label{lem:fixedline} Suppose that $H$ is a finite subgroup of $\mbG=E_7$ possessing no subgroups of index $2$. If $H$ centralizes a line on either $M(E_7)$ or $L(E_7)^\circ$ then $H$ is strongly imprimitive.
\end{lem}

(Of course, if $H$ is perfect then $H$ has no subgroups of index $2$, and any line stabilized by $H$ is centralized by it.) The next result is a little complicated in full generality, because it has to include both graph automorphisms and Steinberg endomorphisms, not just Frobenius endomorphisms. For $E_7$ it is much simpler. This may be found in \cite[Proposition 1.12]{liebeckseitz1998}, and in various guises in \cite{craven2015un,craven2019un,craven2020un,litterickmemoir}.

\begin{prop}\label{prop:intersectionsubspace} Let $H$ be a subgroup of $\mbG=E_7$. Suppose that $W$ is a subspace of $M(E_7)$ or $L(E_7)^\circ$ stabilized by $H$. If the intersection of the stabilizers of all subspaces in the $N_{\Aut^+(\mb G)}(H)$-orbit of $W$ is positive dimensional then $H$ is strongly imprimitive.
\end{prop}

\begin{lem}[See, for example, {{\cite[Theorem 21.11]{malletesterman}}}]\label{lem:fixedpoints} Let $\mbG$ be a connected linear algebraic group with a Steinberg endomorphism $\sigma$, and write $G=\mbG^\sigma$. Let $X$ be a non-empty set such that $\mbG$ and $\sigma$ act on $X$, and such that $(xg\sigma)=(x\sigma)(g\sigma)$ for all $x\in X$ and $g\in \mbG$. Then $X^\sigma=\emptyset$, i.e., there are $\sigma$-fixed points in $X$.

Furthermore, if $\mbG$ is transitive on $X$ and for some $x\in X$ the stabilizer $\mbG_x$ of $x$ is closed, the number of $G$-orbits on the $\sigma$-fixed points of $X$ are in one-to-one correspondence with $\sigma$-conjugacy classes on $\mbG_x/\mbG_x^\circ$. In particular, if $\mbG_x$ is connected then all $\sigma$-fixed points of $X$ are $G$-conjugate.
\end{lem}

The next lemma can be found in, for example, \cite[Lemma 3.2.15]{stewart2013}.

\begin{lem}\label{lem:1cohom} Let $\mbG$ be a connected linear algebraic group, let $\mb P$ be a maximal parabolic subgroup of $\mbG$. Let $\mb U$ be the unipotent radical of $\mb P$ and let $\mb L \mb T$ be a Levi complement to $\mb U$ in $\mb P$, where $[\mb L,\mb L]=\mb L$ and $\mb T$ is toral. Suppose that $\mb U$ forms a $k\mb L$-module.

Let $H$ be a subgroup of $\mb L$. If
\[ \dim_k(H^1(H,\mb U))=1\]
then there are exactly two $\mb T$-classes of complements to $\mb U$ in $\mb U.H$, one of which lies in $\mb L$.
\end{lem}

Specializing to the case where $\mbG$ has type $C_n$ and $\mb L$ has type $A_{n-1}$ yields the following corollary.

\begin{cor}\label{cor:symp1cohom} Let $\mbG=\Sp_{2n}(k)$, and let $H$ be a perfect subgroup of $\mbG$. Suppose that $H$ stabilizes a subspace $W$ of $M(\mbG)$ of dimension $n$ but no complement to $W$. If the $kH$-module $S^2(W)$ has $1$-dimensional $1$-cohomology then any other subgroup $L$ of $\mbG$ that is isomorphic to $H$, and such that $M(\mbG){\downarrow_H}$ and $M(\mbG){\downarrow_L}$ are isomorphic, is $\mbG$-conjugate to $H$. In other words, $H$ is unique up to $\mbG$-conjugacy.
\end{cor}
\begin{proof} Since $H$ does not stabilize a complement to $W$, $W$ cannot be non-singular (as then $W^\perp$ intersects $W$ trivially and also has dimension $n$). Thus $H$ (and the group $L$) lies in an $A_{n-1}$-parabolic subgroup $\mb X$. As all such subspaces are in the same orbit, we may assume that $L$ also stabilizes $W$ (and acts in the same way on $W$).

To apply Lemma \ref{lem:1cohom}, we need that the $A_{n-1}$ acts on the unipotent radical of $\mb X$ as the symmetric square of $W$. Since $L(C_n)$ is the symmetric square of $M(C_n)$, clearly $L(C_n){\downarrow_{A_{n-1}}}$ has factors $S^2(M(A_{n-1}))$ and its dual, and the tensor product $M(A_{n-1})\otimes M(A_{n-1})^*$, the Lie algebra of $A_{n-1}T_1$. Thus the unipotent radical has $kH$-module structure either $S^2(M(A_{n-1}))$ or its dual.

To see which one, we can either perform a calculation, or can simply use the computation from Proposition \ref{prop:2alt7} below, which shows that it is the submodule we need to take $1$-cohomology of, rather than the quotient module. This completes the proof.
\end{proof}

(The reason we may take the computation from Proposition \ref{prop:2alt7} below is that, if we are given an example where there definitely are two classes of complements, then we can use that to deduce the module structure by simply asking which module for that group has non-zero $1$-cohomology.)

\begin{lem}[See, for example, {{\cite[Corollary 21.8]{malletesterman}}}]\label{lem:inneriscent} Let $\sigma$ be a Steinberg endomorphism of a connected group $\mbG$. If $g\in \mbG$ then $\sigma$ and $g\sigma$ are $\mbG$-conjugate. In other words, all outer automorphisms of $\mbG$ corresponding to the same element of $\Out(\mbG)$ as $\sigma$ are conjugate. Consequently, if $\sigma$ acts as an inner automorphism of the subgroup $H$ then there exists a $\mbG$-conjugate of $H$ centralized by $\sigma$.
\end{lem}

The following is well known, and may be found, for example, in either \cite[Lemma 4.3]{liebecksaxl1987} or \cite{cooperstein1981}.

\begin{prop}\label{prop:1spacestabs} Let $\mbG$ be simply connected $E_7$, and let $W$ be a $1$-space in $M(E_7)$. The stabilizer of $W$ is one of the following:
\begin{enumerate}
\item an $E_6T_1$-parabolic subgroup;
\item the normalizer of an $E_6$-Levi subgroup, $E_6T_1.2$;
\item a subgroup $\mb U_{26}\cdot F_4T_1$ where $\mb U_{26}$ is a unipotent subgroup of dimension $26$, contained in an $E_6T_1$-parabolic subgroup;
\item a subgroup $\mb U_{33}\cdot B_5T_1$ where $\mb U_{33}$ is a unipotent subgroup of dimension $33$, contained in a $D_6T_1$-parabolic subgroup.
\end{enumerate}
\end{prop}

\subsection{Action of outer automorphisms}

In this section we will give some general results on how to understand the action of outer automorphisms of the simple groups $G$ on subgroups $\bar H$. Because $\mbG$ has type $E_7$, we are interested in field automorphisms $F_q$ and diagonal automorphisms for $p$ odd. This subsection is a simplified version of \cite[Section 3.7]{craven2020un}, because here we need not consider graph automorphisms. Let $\sigma=F_p$ be the generator for the semigroup of field endomorphisms. Suppose that $H$ is a quasisimple subgroup of $\mbG$ with $Z(H)\leq Z(\mbG)$. First, $\sigma$ might stabilize the $\mbG$-conjugacy class containing $H$ or it might map it to another class.

In all of our cases, if $H_1$ has the same actions on $M(E_7)$ and $L(E_7)$ as $H$ then $H$ and $H_1$ are $\mbG$-conjugate. As this is the case, one may determine whether $\sigma$ stabilizes the class of $H$ simply by examining the character of $M(E_7){\downarrow_H}$ and $L(E_7){\downarrow_H}$. As $F_p$ raises each eigenvalue to the $p$th power, this determines the character of $H^\sigma$, and therefore the class to which it belongs. If there were more than one class with the same character, it would be much harder to determine whether $\sigma$ stabilizes the class of $H$.

By replacing $\sigma$ by a power if necessary, we may now assume that $\sigma$ stabilizes the $\mbG$-class containing $H$. Replacing $H$ by a $\mbG$-conjugate if necessary, we may assume that $\sigma$ normalizes $H$, via Lemma \ref{lem:fixedpoints}. We now must determine the action of $\sigma$ on $H$, and in particular if it acts as an element of $N_\mbG(H)$. If this is the case, then $\sigma$ acts trivially on some other conjugate of $N_\mbG(H)$ by Lemma \ref{lem:inneriscent}, and therefore a conjugate of $H$ lies in $G$.

Generally the outer automorphism groups of our subgroups $H$ are fairly small, since $H$ is simple and of fairly low order, so this is not too difficult. If the characters of $H$ on $M(E_7)$ and $L(E_7)$ are not stable under an automorphism $\phi$ of $H$, then in order for $\sigma$ to induce $\phi$ on $H$, raising eigenvalues to that power must have the effect of $\phi$ on the characters, just as in the analysis above on swapping $\mbG$-classes.

If $\sigma$ has no effect on the characters of $M(E_7)$ and $L(E_7)$, and neither does $\phi$, then we will need a different method to determine the action of $\sigma$ on $H$, and that will have to be ad hoc. This is particularly important for $\PSU_3(3)$ (see Section \ref{sec:psu33}).

\medskip

Diagonal automorphisms are generally much more difficult to understand. The next lemma gives us some control in many cases. The issue is when $p$ is odd and $N_\mbG(H)/HZ(\mbG)$ contains an involution, which could be the diagonal automorphism of $G$.

\begin{lem}\label{lem:diagswaps} Let $H$ be a subgroup of $\mbG$ such that $C_\mbG(H)=Z(\mbG)$, and with $\bar H=HZ(\mbG)/Z(\mbG)$ simple. Suppose that $H$ is centralized by $\sigma$.
\begin{enumerate}
\item We have that $N_\mbG(H)$ is contained in the preimage in $\mbG$ of $G_{\mathrm{ad}}$.
\item The number of $G$-classes in the $G_{\mathrm{ad}}$-class of $\bar H$ is $|G_{\mathrm{ad}}:G|/|N_\mbG(H):N_{G_\mathrm{sc}}(H)|$.
\end{enumerate}
In particular, if $N_\mbG(H)\leq G_{\mathrm{sc}}$ (for example, if $|N_\mbG(H):HZ(\mbG)|$ is prime to $|G_\mathrm{ad}:G|$) then a diagonal automorphism of order $d$ of $G$ fuses $d$ distinct $G$-classes of subgroups $\bar H$.
\end{lem}
\begin{proof} Since the image of $H$ in the adjoint group is simple, and by a result of Burnside the automorphism group of a simple group is complete\footnote{Let $S$ be finite non-abelian simple and let $G=\Aut(S)$. Let $\sigma\in\Aut(G)$. Then $\sigma=c_g$ on $S$ for some $g\in G$, so $\sigma g^{-1}$ centralizes $S$. Thus we may assume that $\sigma$ centralizes $S$. If $\phi\in\Aut(S)=G$ then $\phi^\sigma=\phi$ on $S$, by a simple calculation. So $\sigma$ centralizes $G=\Aut(S)$, but $\sigma\in\Aut(G)$, so $\sigma=1$.}, we see that any automorphism of $N_\mbG(H)/Z(\mbG)$ that centralizes the image of $H$ -- for example, $F_q$ -- must centralize all of $N_\mbG(H)/Z(\mbG)$. Thus in the adjoint group, $N_\mbG(H)$ is centralized by $F_q$, and so $N_\mbG(H)\leq G_{\mathrm{ad}}\; Z(\mbG)$, proving the first part.

The second part and the consequence are clear.
\end{proof}

In the remaining case, where $N_\mbG(H)/HZ(\mbG)$ contains an involution, $N_\mbG(H)$ might lie in $G\cdot Z(\mbG)$ or it might only lie in $G_\mathrm{ad}$. If the action of $N_\mbG(H)$ on $M(E_7)$ is not definable over $\F_q$ then $N_\mbG(H)$ cannot embed in $2\cdot E_7(q)$, of course, and this is often enough to conclude that the normalizer does not lie inside $G$ (if it does not). Other situations must be dealt with on a case-by-case basis (or in the case of $\PGL_3(q)$ in Section \ref{sec:othermaximals} not dealt with at all).

\subsection{Alternating forms on vector spaces}

Fix a non-degenerate, alternating bilinear form $f(-,-)$ on a vector space $V$, and let $\mb X=\Sp(V)$ with respect to this form. Let $H$ be a subgroup of $\mb X$. If $f'(-,-)$ is another $H$-invariant symplectic form on $V$ then there exists $g\in \GL(V)$ such that, if $\mb Y$ is the symmetry group of $f'$ then $\mb Y^g=\mb X$, i.e., $(f')^g=f$. Furthermore, since $\mb X$ is irreducible on $V$, if $g'$ is another such element then there exists $n\in \mb X\cdot Z(\GL(V))$ such that $g'=gn$. Hence, the subgroup $H^g$ of $\mb X$ is determined up to $\mb X$-conjugacy. Thus to every (non-degenerate) $H$-invariant symplectic form on $V$ we may associate a subgroup of $\Sp(V)$, unique up to conjugacy. We see that the number of $\mb X$-conjugacy classes in the image of this function is equal to the number of $\Sp(V)$-conjugacy classes of subgroups isomorphic to $H$ and that lie in the same $\GL(V)$-conjugacy class. In particular, $H$ is `unique up to $\Sp(V)$-conjugacy' (given an action on $V$) if, given $f'$, we may choose $g$ to lie in $C_{\GL(V)}(H)$.

In attempting to understand the $\Sp(V)$-conjugacy classes of a subgroup $H$, we start with a proposition that uses the above observations to split $V$ up into various $H$-invariant summands.

\begin{prop}\label{prop:ssnofactorsincommon} Let $H$ be a finite group and let $V\cong V_1\oplus V_2$ be a faithful $kH$-module. If $\phi_i$ is an alternating bilinear form on $V_i$ with radical zero then all subgroups $H$ of $\Sp(V)$ corresponding to the forms $a\phi_1+b\phi_2$ ($a,b\in k$) are conjugate.
\end{prop}
\begin{proof} The subgroup $H$ lies in the direct product of $\Sp(V_1)$ and $\Sp(V_2)$. The scalar matrices on $V_1$ and $V_2$ act transitively on the set of forms $a\phi_1+b\phi_2$, and these scalars centralize $H$. Thus the corresponding subgroups of $\Sp(V)$ are conjugate.
\end{proof}

Thus if an $H$-invariant alternating form on $V=V_1\oplus V_2$ (where each $V_i$ is $H$-invariant) can always be written as a sum of two forms, one on each $V_i$, then this proposition applies. For example, if the $V_i$ have no composition factors in common as $kH$-modules then this must hold.

This allows us to produce the following proposition, which is often enough for us.

\begin{prop}\label{prop:ssmultfree} Let $H$ be a finite group and let $V$ be a faithful, semisimple $kH$-module. Suppose that, for all composition factors $W$ of $V$,
\begin{enumerate}
\item if $W$ is not self-dual then both $W$ and $W^*$ have multiplicity exactly $1$ in $V$;
\item if $W$ is self-dual and supports a symmetric form then $W$ appears with multiplicity exactly $2$ in $V$; and
\item if $W$ is self-dual and supports an alternating form then $W$ appears with multiplicity $1$ in $V$.
\end{enumerate}
Then up to conjugacy we may assume that $H\leq \Sp(V)$, and if $H_1\leq \Sp(V)$ is $\GL(V)$-conjugate to $H$ then it is $\Sp(V)$-conjugate to $H$. In other words, there is a unique conjugacy class of $H$ in the symplectic group acting on the natural module as $V$.
\end{prop}
\begin{proof} First, we note that $H$ may be conjugated into $\Sp(V)$, because each homogeneous component (where we group a module and its dual into a component) supports a unique symplectic form with zero radical.

To prove uniqueness of $H$, by the previous paragraph the result holds if $V$ is homogeneous, so we may use induction and assume that $V=V_1\oplus V_2$ and each $V_i$ has a unique symplectic form $f_i$ with zero radical (up to $\Sp(V_i)$-conjugacy). Then we apply Proposition \ref{prop:ssnofactorsincommon}.\end{proof}

\section{Two methods for determining conjugacy classes}
\label{sec:methods}

The two main computational methods in this article both involve the Lie product on $L(\mb G)$, and one involves the minimal module as well. We start with the easiest.

\subsection{Lie product method}
\label{sec:lieprodmethod}
Let $H$ be a finite group, let $M$ be a $kH$-module of dimension $\dim(L(\mb G))$, and suppose that we wish to determine whether $H$ embeds in $\mb G$, acting on $L(\mb G)$ as the module $M$, and if so, how many $\mb G$-classes of subgroups have such an action.

The Lie product is a non-zero element of the space 
\[\Hom_{k\mb G}(\Lambda^2(L(\mb G)),L(\mb G)),\]
so we construct the space
\[\mathcal R^H(M)=\Hom_{kH}(\Lambda^2(M),M),\]
which is $n$-dimensional. We then take a basis $\phi_1,\dots,\phi_n$ for this space, and let
\[ \phi=\sum_{i=1}^n a_i\phi_i,\]
where the $a_i$ are indeterminates from $k$. The element $\phi$ is a generic alternating product on $M$.

We impose the requirement that $\phi$ is a Lie product, i.e., it satisfies the Jacobi identity. This requirement yields a collection of quadratic equations in the $a_i$,
\[ ((u\wedge v)\phi \wedge w)\phi+((v\wedge w)\phi \wedge u)\phi+((w\wedge u)\phi \wedge v)\phi=0,\]
for $u,v,w$ basis elements of $M$.

Solving these equations yields the collection $\mathscr L$ of all Lie products on $M$. There may well be many elements of $\mathscr L$ that are not simple Lie algebras, for example the zero product, or more generally any with a kernel, or whose image is not the whole of $M$. Removing these, and for example any that have an abelian subalgebra of too large a dimension (see Proposition \ref{prop:abeliansubalgebra}), we end up with a collection of candidates for being a simple Lie algebra of type $L(\mb G)$. Finally, we need to take orbits under the action of the normalizer in $\GL(M)$ of $H$. For each of these orbits, we simply test on Magma whether the Lie algebra is simple, and then that it is $L(\mb G)$.

\medskip

This method was used successfully in \cite{craven2020un} to determine the number of conjugacy classes of various Lie primitive subgroups for $E_6$, and we use it here, with much larger dimensions for the space of alternating products, to deal with some of the larger subgroups. It has the benefit that one does not need to worry about whether $H$ is a non-split extension by the centre of $\mbG$.

If $\mathcal{R}^H(M)$ has dimension more than about 100 then the number of variables becomes too large to practically determine the subset of Lie products, and so we use a different method.

\subsection{Subalgebra method}
\label{sec:subalgmethod}

This method is more powerful than the previous one, but requires putative actions $M$ and $L$ on both the minimal and adjoint modules. It also requires an alternating bilinear form on $M$ that yields an embedding of $H$ into $\Sp_{56}(k)$. For this reason, the previous method is preferred in cases where the space of alternating forms is small, because the setup is less complicated. This method, due initially to Alexander Ryba, appeared in \cite{griessryba2002}, where it was used to compute conjugacy classes of $\PSL_2(27)$ in $E_7(k)$ for $p\nmid |\PSL_2(27)|$. We use it with a number of modifications compared with \cite{griessryba2002} in order to increase its utility in cases where the prime divides the order of the subgroup.

For all primes $p$, the group $\mbG$ of type $E_7$ is contained in the algebraic group $\mb X$ of type $C_{28}$, i.e., $\Sp_{56}(k)$ contains $E_7(k)_{\mathrm{sc}}$. Thus $L(E_7)$ is a Lie subalgebra of $L(C_{28})$. Since $L(C_{28})$ is the symmetric square of $M(C_{28})$, we place a Lie algebra structure directly on the symmetric square of $M(C_{28})$. For practical purposes when computing we place the Lie algebra structure on $M(C_{28})\otimes M(C_{28})$, given by
\[ [u\otimes v, x\otimes y]=(u,y) x\otimes v-(x,v)u\otimes y\]
(where $u,v,x,y\in M(C_{28})$ and $(\,,\,)$ is the bilinear form) and then consider the symmetric part of the tensor square.

Let $H$ be a putative subgroup of $\mbG$, and assume that we are given $56$- and $133$-dimensional $kH$-modules $M$ and $L$. We will count the number of $\mbG$-conjugacy classes of subgroups $H$ embedding with actions $M$ and $L$ on $M(E_7)$ and $L(E_7)$ respectively.

Suppose first that $M$ supports a unique (up to $\Sp_{56}$-conjugacy) non-degenerate symplectic form. This yields a unique embedding of $H$ into $C_{28}$, and thus $H$ acts on the Lie algebra $L(C_{28})=S^2(M)$. Let $U$ denote the submodule of $S^2(M)$ that is the sum of the images of all maps in $\Hom_{kH}(L,S^2(M))$. If $H\leq \mbG\leq C_{28}$ then the Lie algebra of $\mbG$ lies in the subspace $U$ of $S^2(M)$.

Let $W$ be a simple submodule of $L$, let $\phi_1,\dots,\phi_n$ denote a basis for $\Hom_{kH}(W,U)$, and let $W_i$ denote the image of $\phi_i$. If $w\in W$, and $a_1,\dots,a_n$ are parameters with values in $k$, write
\[ \bar\;:W\to U,\quad \bar w=\sum_{i=1}^n a_i(w\phi_i)\in U.\]
If $\bar u$ and $\bar w$ lie in $L(\mbG)$ for $u,w\in W$ then, in particular, $[\bar u,\bar w]\in U$. By extending a basis for $U$ to a basis for $S^2(M)$, the coefficients of $[\bar u,\bar w]$ on basis elements outside of $U$, which must be $0$, yield quadratic equations in the $a_i$. Of course, choosing three points $u,v,w\in W$, one may obtain cubic equations via $[[\bar u,\bar v],\bar w]$, and so on. These restrictions yield a smaller set of potential subalgebras, and (hopefully) allow us to determine which copies of $W$ in $S^2(M)$ lie in $L(\mbG)$ precisely. We then take the subalgebra generated by each option for $W$. If $H$ does not preserve any proper subalgebra of $L(\mbG)$ then the subalgebra so generated must be $L(\mbG)$ itself, and we can test this via computer. Running through all possibilities for the image of $W$ up to $N_{C_{28}}(H)$-conjugacy yields all $\mbG$-conjugacy classes of subgroups $H$. (See Example \ref{ex:m12} below for an example of this in action.)

\medskip

This is the basic idea, but there are many modifications that can be made. First, $W$ might not be enough, and we can allow more than one simple submodule to be chosen simultaneously. If there is a $1$-parameter of symplectic forms then one may still compute $[\bar u,\bar w]$, but it will come with an extra parameter that must be solved for. (This is the case for $\PSL_2(19)$, $p=2$.) If $p=2$ then $S^2(M)$ is reducible, and we may consider $\Lambda^2(M)$ instead of $S^2(M)$, thought of as a submodule of $M\otimes M^*$, and hence inheriting its Lie product. If even one parameter for the symplectic form is not enough, we might perform the same analysis in the Lie algebra of $A_{55}$ instead, so work with $M\otimes M$ rather than $S^2(M)$. (We demonstrate this method in the supplementary materials for $\PSL_2(19)$ and $p=2$ in Section \ref{sec:psl219}, providing two proofs for that case.)

\medskip

Having applied the method, probably to multiple submodules, one obtains a number of different possibilities for submodules $\bar W$ of $L(\mbG)$, up to the action of $N_{C_{28}}(H)$. For each of these we take the subalgebra generated by $\bar W$. If it has dimension greater than $133$ then $\bar W$ cannot lie in $L(\mbG)$. If it generates an abelian subalgebra of dimension greater than $27$ ($28$ for $p=2$) then $\bar W$ cannot lie in $L(\mbG)$, by Proposition \ref{prop:abeliansubalgebra}. One also knows all of the actions of nilpotent classes of the Lie algebra of $\mbG$ on $M$ \cite[Tables 2 and 3]{stewart2016} (they differ slightly from those from \cite{lawther1995} of unipotent classes for $E_7$ and $p\leq 5$), and hence on $S^2(M)$, so if the subalgebra contains a nilpotent element of $L(C_{28})$ that has the wrong Jordan normal form on $S^2(M)$ then $\bar W$ cannot lie in $L(\mbG)$. (We use this for $\Alt(7)$, $p=5$, $\PSL_2(11)$, $p=3$, and $\PSL_2(13)$, $p=2$.) If the subalgebra generated is not a subalgebra of $\mathfrak{e}_7$ then $\bar W$ cannot lie in $L(\mbG)$.

For smaller examples $H$, the process becomes increasingly difficult. One useful tool is a linear matrix pencil problem. Given nilpotent matrices $A$ and $B$ such that $A+xB$ is nilpotent for all $x\in k$, we want to compute the Jordan normal form of $A+xB$, both generically and for the finitely many values of $x$ for which the normal form is different. Assuming these actions are not correct, this allows us to eliminate $1$-parameter possibilities for $\bar W$. We do this for, for example, $\PSL_2(11)$ (Section \ref{sec:psl211}).

These various tests are enough to eliminate almost all possibilities other than submodules $\bar W$ that generate a copy of $\mathfrak{e}_7$. In some cases, for example $\Alt(7)$ for $p=5$ (Section \ref{sec:alt7}), we always show that the submodule generates a $\mathfrak d_4$ subalgebra. At this stage we may apply a recent result of Premet and Stewart \cite{premetstewart2019}, that classifies maximal subalgebras of exceptional Lie algebras in good characteristic. In the few cases where we need to do this we will describe exactly what occurs in the relevant section.

We now give an example using this method, which in Section \ref{sec:m12} is solved using the Lie product method. It will offer another proof of the result, and be small enough to illustrate the technique.

\begin{example}\label{ex:m12} Let $H=2\cdot M_{12}$ and $p=5$. Examining \cite[Table 6.148]{litterickmemoir}, we see there is a possible embedding of $H$ into $\mbG$ with composition factors on $M(E_7)$ and $L(E_7)$ given by $32,12^2$ and $78,55_3$ respectively (the character $55_3$ is the one of the three with Brauer character value $7$ for a non-central involution). This is labelled `\textbf{P}' in that table, and would be Lie primitive if it exists.

The module $12$ has no self-extensions, but there is an extension between $12$ and $32$, allowing us to construct a self-dual module $12/32/12$. The actions of $u\in H$ of order $5$ on $32\oplus 12^{\oplus 2}$ and $12/32/12$ have Jordan blocks $5^{10},1^6$ and $5^{10},3^2$. Both of these actions are allowable according to \cite[Table 7]{lawther1995}, they of course correspond to different unipotent classes. The action on $L(E_7)$ is clearly fixed though as $78\oplus 55$, and $u$ acts on this module with Jordan blocks $5^{26},3$. Consulting \cite[Tables 7 and 8]{lawther1995}, we see that $u$ belongs to class $A_4+A_2$ and $H$ must act uniserially on $M(E_7)$. A computer check confirms that the module $12/32/12$ has a unique symplectic form up to scalar multiplication, and so we are in a good position to apply the subalgebra method. Write $M$ and $L$ for the $kH$-modules $12/32/12$ and $55_3\oplus 78$ respectively.

We first determine the normalizer of $H$ in $\Sp_{56}(k)$. In this case it is quite easy to check that $H$ is centralized by a $1$-dimensional unipotent subgroup $\mb X$ (and the centre of $\Sp_{56}(k)$). In addition, $H$ is normalized by an element of order $2$ acting as an outer automorphism on $H$. This will either swap two copies of $L(E_7)$ inside $S^2(M)$ or stabilize one, and this corresponds to $H.2$ not embedding in $\mbG$ and embedding in $\mbG$ respectively.

\medskip

We now consider $S=S^2(M)$. The space $\Hom_{kH}(55_3,S)$ is $4$-dimensional, and the space $\Hom_{kH}(78,S)$ is $2$-dimensional. This means that $U$, the sum of the images of all maps from $L$ to $S$, has dimension $55\cdot 4+78\cdot 2=376$, whereas $S$ has dimension $1596$. We can illustrate a number of our techniques here. First, $\mb X$ acts on the set of $78$-dimensional $kH$-submodules of $U$; there is a regular orbit, with representative $W_1$, and a fixed point, with representative $W_2$. Let $L_i$ denote the subalgebra generated by $W_i$. It turns out that both $L_i$ are isomorphic to $L$ as $kH$-modules; the subalgebra $W_1$ is of type $\mathfrak e_7$, but the second is a nilpotent Lie algebra with a $55$-dimensional abelian ideal.

We can simply test whether $L_1$ is an $\mathfrak e_7$-Lie algebra, so it remains to eliminate the other one. We can do this in a number of ways: we can prove that the $55_3$ forms an abelian subalgebra of $L_2$, thus $L_2$ is not $\mathfrak e_7$ by Proposition \ref{prop:abeliansubalgebra}. We could also check that the $55_3$ forms an ideal of the Lie algebra. Finally, we could find some nilpotent element of $S$ that lies in $W_2$ and has the wrong Jordan normal form. In the supplementary materials we demonstrate each of these proofs.

We also demonstrate an approach using $W=55_3$ and $W'=78$, using the action of $\mb X$ on the set of possibilities for $W$. Here we therefore have six variables, $a_1,\dots,a_4$ parametrizing the submodule $W$ and $a_5$ and $a_6$ parametrizing $W'$. The action of $\mb X$ shows that we may choose $a_1=a_2=a_3=0$, and then we are the $55_3$ inside $L_2$, or $a_3$ may be chosen arbitrarily, and we can choose it to be the $55_3$ inside $L_1$ (by choosing $a_3$ appropriately).

Since $W_1$ is unique up to the action of the centralizer $\mb X$, we see that $H$ must extend to $H.2$ in $\mbG$.
\end{example}

\section{Reducing to almost simple groups}
\label{sec:othermaximals}
In this section we determine all maximal subgroups that arise from fixed points of a positive-dimensional subgroup. Apart from these, there is a finite list of other possibilities, each of which will be considered in turn in the next section.

We start with a result of Borovik and Liebeck--Seitz, which reduces us to almost simple groups.

\begin{thm}[\cite{borovik1989,liebeckseitz1990}]\label{thm:borliesei} Let $M$ be a maximal subgroup of an almost simple group $\bar G$ of exceptional type. One of the following holds:
\begin{enumerate}
\item $M=N_{\bar G}(\mb X^\sigma)$ for $\mb X$ a member of $\mathscr X$, as listed in \cite[Theorem 2]{liebeckseitz1990};
\item $M$ is an exotic $r$-local subgroup, as listed in \cite{clss1992};
\item $\mbG$ is of type $E_8$, $p\geq 5$, and $M=(\Alt(5)\times \Alt(6))\rtimes 2^2$;
\item $M$ has the same type as $\mb G$, and is the normalizer of the fixed points of a field or field-graph automorphism of $G$;
\item $M$ is almost simple.
\end{enumerate}
\end{thm}

If $\mbG$ is of type $E_7$ then the third option does not apply, and the second option is a subset of the first, as it happens. The fourth is simply (the normalizer of) $E_7(q_0)$ for $q=q_0^r$ with $r$ a prime. Note that if $M$ is almost simple but does not appear in (i)--(iv) then $C_\mbG(M)=Z(\mbG)$. To see this we argue as in \cite[Section 2]{craven2020un}: notice that otherwise $M\leq C_\mbG(C_\mbG(M))$, and this latter group is proper in $\mbG$. It is certainly $\sigma$-stable whenever $M$ is, and is infinite if $M$ is simple by \cite[Lemma 1.2(i)]{liebeckseitz1990} (as the remaining option there is the Borovik subgroup (iii) above). In particular, we can always apply Lemmas \ref{lem:fixedline} and \ref{lem:diagswaps} above.

We will look at almost simple groups in the next section. Thus in this section we determine the list for the first option, fixed points of ($\sigma$-stable) maximal members of $\mathscr X$. So assume that $M$ lies in (i).

By \cite[Theorem 1]{liebeckseitz2004}, $M$ is either a parabolic or maximal-rank subgroup, or a short list of other (possibly disconnected) reductive groups. The maximal-rank subgroups are tabulated in \cite{liebecksaxlseitz1992}, and we reproduce those subgroups here. The parabolic subgroups are well-known. The result from \cite{liebeckseitz2004} lists all other subgroups.

Ignoring the maximal-rank, parabolic, and exotic $r$-local subgroups, this leaves only the maximal subgroups arising from reductive maximal subgroups, which for $E_7$ are (from \cite{liebeckseitz2004}):
\begin{enumerate}
\item[(a)] $A_1F_4$ (all $p$), $G_2C_3$ (all $p$), $A_1G_2$ ($p\geq 3$);
\item[(b)] $A_1A_1$ ($p\geq 5$);
\item[(c)] $A_2.2$ ($p\geq 5$);
\item[(d)] $A_1$ ($p\geq 17$);
\item[(e)] $A_1$ ($p\geq 19$);
\item[(f)] $(2^2\times D_4)\rtimes \Sym(3)$ ($p\geq 3$).
\end{enumerate}
(The final one is disconnected, and the connected component lies in $A_7$.)

The number of classes of each of these in $G_\mathrm{ad}$ is given in \cite{liebeckseitz1990,liebeckseitz2004}, and so our issue is whether the diagonal automorphism of $G$ normalizes the subgroup or fuses two classes of them. The three subgroups in (a) are easy to understand: they are all non-split extensions over the centre, and exactly one of the two factors has a single isogeny type. Thus they are direct products and of simply connected type. This yields subgroups $\PSL_2(q)\times F_4(q)$, $G_2(q)\times \PSp_6(q)$ and $\PSL_2(q)\times G_2(q)$ of $G$. Since they are not groups of Lie type themselves, they must be normalized by a diagonal automorphism, and since they are unique up to $\mbG$-conjugacy, they are stabilized by all Frobenius endomorphisms. A similar statement holds for the two classes of $A_1$ subgroups (which cannot be interchanged by $\sigma$ for many reasons, for example the unipotent classes intersecting the two subgroups are different).

(If $q=2,3$ then $\PSL_2(q)$ is soluble. In this case the groups above are $r$-local subgroups, where $r=3$ for $q=2$, and $r=2$ for $q=3$. We consult \cite{clss1992} to see when these are maximal.)

The subgroup $A_1A_1$ is of type $\SL_2\times \PGL_2$, by considering the composition factors on $M(E_7)$, and so we cannot have a Frobenius endomorphism that swaps the two factors. Thus the only option is $\PSL_2(q)\times \PGL_2(q)$ in $E_7(q)$, and as with the other cases, there must be a single class normalized by the diagonal automorphism.

For $\mb X=(2^2\times D_4)\rtimes \Sym(3)$, this can also be written as $D_4\rtimes \Sym(4)$. We work first in the adjoint version of $\mbG$, and then see whether the subgroup lies in the derived subgroup of $G_\mathrm{ad}$. Since $\Sym(4)$ has no outer automorphisms, if $\mb X$ is $\sigma$-stable then $\sigma$ induces an inner automorphism on $\Sym(4)$. There are five conjugacy classes for $\Sym(4)$, hence five $\sigma$-conjugacy classes. The classes given by $1$ and $(1,2)(3,4)$ yield fixed-point subgroups $D_4(q)\rtimes \Sym(4)$ and $D_4(q)\rtimes D_8$ respectively; the classes given by $(1,2)$ and $(1,2,3,4)$ both give subgroups ${}^2\!D_4(q)$, and $(1,2,3)$ yields a fixed-point subgroup ${}^3\!D_4(q).3$. Since each of $(1,2)(3,4)$, $(1,2)$ and $(1,2,3,4)$ yield a fixed point on the $2^2$ subgroup in the centre of the fixed points, in these three cases the $\sigma$-fixed points lie in the centralizer of an involution, so a subgroup $A_7$ or ${}^2\!A_7$. In particular, they cannot be maximal. Thus the only maximal subgroups are $(2^2\times D_4(q))\rtimes \Sym(3)$ and ${}^3\!D_4(q).3$, neither of which can lie in any other maximal subgroup.

Of course the subgroup ${}^3\!D_4(q).3$ must lie in the derived subgroup $G$, which has index $2$ in $G_\mathrm{ad}$, so it remains to consider $(2^2\times D_4(q))\rtimes \Sym(3)$. But this is the exotic $2$-local subgroup from \cite{clss1992}. It is not explicitly stated there whether there is a single class for the simple group, just for $G_\mathrm{ad}$. In \cite[Table 1]{andietrich2016} it is erroneously stated that $(2^2\times D_4(q)).3$ lies in the simple group, hence there is a single class. This was corrected in a recent paper of Korhonen \cite{korhonen2025}, and the answer depends on the congruence of $q$ modulo $8$. The proof in \cite{korhonen2025} explicitly constructs the subgroup in the algebraic group; in Appendix \ref{sec:altkorhonenproof} we sketch a proof using more general theory of groups of Lie type and some finite group theory, which hopefully elucidates a different aspect of the situation.

Finally, for the subgroups $A_2.2$, we see one class each of $\PGL_3(q).2$ and $\PGU_3(q).2$ in the adjoint group $G.2$ (see, for example, \cite[Corollary 7]{liebeckseitz1998} for this statement; the proof follows the same lines as the same subgroup of $E_6$ in \cite[Section 7]{craven2020un}). The question is whether these groups lie inside $G$ or intersect $G$ in $\PGL_3(q)$ or $\PGU_3(q)$. Notice that, by Lemma \ref{lem:fixedpoints}, the number of conjugacy classes of $\PGL_3(q)$ and $\PGU_3(q)$ put together is equal to the number of $\sigma$-conjugacy classes on the cyclic group of order $4$ (which is the normalizer of $A_2$ modulo $A_2$), and this number is either $2$ or $4$, according as $\sigma$ inverts or centralizes the $4$. Hence, either both $\PGL_3(q)$ and $\PGU_3(q)$ are self-normalizing in $G$ or neither is, so we may as well consider just $\PGL_3(q)$.

For $p=5$, we obtain a different answer to $p\geq 7$, because the action on $M(E_7)$ is completely different. For $p\geq 7$, the action is a sum of two dual $28$-dimensional modules, and the centralizer in $\Sp(M(E_7))$ is a torus. For $p=5$, the action is a tilting module, of the form $L(22)/L(10),L(01),L(20),L(02)/L(22)$. This has centralizer in $\Sp(M(E_7))$ a unipotent subgroup of dimension $1$.

Consequently, while the centralizer for $p\geq 7$ has order $q-1$, the centralizer for $p=5$ has order $2q$ (with the $2$ being the centre of $\Sp_{56}(k)$). This means that there is an odd number of $\mathfrak e_7$-subalgebras of the Lie algebra, and thus the outer $2$ on top of $\PSL_3(q)$ must stabilize at least one of them. Thus $\PSL_3(q).2$ must embed in $E_7(q)$, as stated in Table \ref{t:othermaximals}. (In the supplementary materials we prove this by computer for $q=5$ to confirm, using the argument for $p\geq 7$ below to extend to all powers of $5$.)

For $p\geq 7$, we cannot determine the answer when $q$ is an odd power of $p$, but we believe the answer is that there is a single class. We have checked this for $q=7,11,13,17$ using a computer (we include the case $q=17$ in the supplementary materials for the reader's benefit) but can see no way to prove it in general at the moment without using the full theory of algebraic groups and a direct construction following \cite{seitz1991} and \cite{liebeckseitz2004}, which we will pursue in another paper. We can say something though. Let $H=\PSL_3(q)$ and let $L=\PSL_3(p)$ be a subgroup of $H$. We have $H.2=\gen{H,L.2}$, and so $H.2$ lies in $G$ if and only if $L.2$ does. If $q$ is an even power of $p$ then $L.2\leq E_7(p).2\leq E_7(q)$, and so $H.2$ does lie inside $G$. If $q$ is an odd power of $p$ then the outer diagonal automorphism of $E_7(q)$ restricts to the outer diagonal automorphism of $E_7(p)$, so $L.2$ lies in $E_7(p)$ if and only if $H.2$ lies in $E_7(q)$.

Thus there are definitely two classes of $\PSL_3(q)$ for $q$ an even power of $p$, and the answer is always the same for $q$ any odd power of a fixed prime $p$.

\begin{table}
\begin{center}
\begin{tabular}{ccccc}
\hline Group & $p$ & $q$ & No. classes & Stabilizer
\\ \hline $[q^{33}]\cdot d\cdot \POmega_{12}^+(q)\cdot (q-1)$ & All $p$ & All $q$ & $1$ & $\gen{\delta,\phi}$
\\ $[q^{42}]\cdot \SL_7(q) \cdot(q-1)/d$ & All $p$ & All $q$ & $1$ & $\gen{\delta,\phi}$
\\ $[q^{47}]\cdot de\cdot (\PSL_2(q)\times \PSL_6(q))\cdot (q-1)$ & All $p$ & All $q$ & $1$ & $\gen{\delta,\phi}$
\\ $[q^{53}]\cdot f\cdot(\SL_3(q)\times \PSL_2(q)\times \PSL_4(q))\cdot (q-1)$& All $p$ & All $q$ & $1$ & $\gen{\delta,\phi}$
\\ $[q^{50}]\cdot (\SL_3(q)\times \SL_5(q))\cdot (q-1)/d$& All $p$ & All $q$ & $1$ & $\gen{\delta,\phi}$
\\ $[q^{42}]\cdot f\cdot(\Omega_{10}^+(q)\times \PSL_2(q))\cdot (q-1)$& All $p$ & All $q$ & $1$ & $\gen{\delta,\phi}$
\\ $[q^{27}]\cdot E_6(q)_{\mathrm{sc}}\cdot (q-1)/d$ & All $p$ & All $q$ & $1$ & $\gen{\delta,\phi}$
\\ $d\cdot(\PSL_2(q)\times \POmega_{12}^+(q))\cdot d$ & All $p$ & All $q$ & $1$ & $\gen{\delta,\phi}$
\\$f/d\cdot \PSL_8(q)\cdot g/f\cdot 2$ & All $p$ & All $q$ & $1$ & $\gen{\delta,\phi}$
\\ $f'/d\cdot \PSU_8(q)\cdot g'/f'\cdot 2$ & All $p$ & All $q$ & $1$ & $\gen{\delta,\phi}$
\\ $e\cdot \PSL_3(q)\times \PSL_6(q))\cdot e\cdot 2$ & All $p$ & All $q$ & $1$ & $\gen{\delta,\phi}$
\\ $e'\cdot \PSU_3(q)\times \PSU_6(q))\cdot e'\cdot 2$ & All $p$ & All $q$ & $1$ & $\gen{\delta,\phi}$
\\ $d^2\cdot(\PSL_2(q)^3\times \POmega_8^+(q))\cdot d^2\cdot \Sym(3)$ & All $p$ & All $q$ & $1$ & $\gen{\delta,\phi}$
\\ $(\PSL_2(q^3)\times {}^3\!D_4(q))\cdot 3$ & All $p$ & All $q$ & $1$ & $\gen{\delta,\phi}$
\\ $d^3\cdot \PSL_2(q)^7\cdot d^3\cdot \PSL_3(2)$ & All $p$ & $q\geq 3$ & $1$ & $\gen{\delta,\phi}$
\\ $\PSL_2(q^7)\cdot 7$ & All $p$ & All $q$ & $1$ & $\gen{\delta,\phi}$
\\ $e\cdot(E_6(q)\times (q-1)/de)\cdot e\cdot 2$ & All $p$ & All $q$ & $1$ & $\gen{\delta,\phi}$
\\ $e'\cdot({}^2\!E_6(q)\times (q+1)/de')\cdot e'\cdot 2$ & All $p$ & All $q$ & $1$ & $\gen{\delta,\phi}$
\\ $(q-1)^7/d\cdot W(E_7)$ & All $p$ & $q\geq 5$ & $1$ & $\gen{\delta,\phi}$
\\ $(q+1)^7/d\cdot W(E_7)$ & All $p$ & All $q$ & $1$ & $\gen{\delta,\phi}$
\\ $\PSL_2(q)\times F_4(q)$ & All $p$ & $q\geq 3$ & $1$ & $\gen{\delta,\phi}$
\\ $G_2(q)\times \PSp_6(q)$ & All $p$ & All $q$ & $1$ & $\gen{\delta,\phi}$
\\ $\PSL_2(q)\times G_2(q)$ & $p\geq 3$ & $q\geq 5$ & $1$ & $\gen{\delta,\phi}$
\\ $\PSL_2(q)\times \PGL_2(q)$ & $p\geq 5$ & All $q$ & $1$ & $\gen{\delta,\phi}$
\\ $\PGL_3(q)\cdot 2$ & $5$ & All $q$ & $2$ & $\gen\phi$
\\ $\PGU_3(q)\cdot 2$ & $5$ & All $q$ & $2$ & $\gen\phi$
\\ $N_G(\PGL_3(q))$ & $p\geq 7$ & All $p^{2n+1}$ & $1$? & $\gen{\delta,\phi}$?
\\ $N_G(\PGU_3(q))$ & $p\geq 7$ & All $p^{2n+1}$ & $1$? & $\gen{\delta,\phi}$?
\\ $\PGL_3(q)\cdot 2$ & $p\geq 7$ & All $p^{2n}$ & $2$ & $\gen{\phi}$
\\ $\PGU_3(q)\cdot 2$ & $p\geq 7$ & All $p^{2n}$ & $2$ & $\gen{\phi}$
\\ $\PSL_2(q)$ & $p\geq 17$ & All $q$ & $1$ & $\gen{\delta,\phi}$
\\ $\PSL_2(q)$ & $p\geq 19$ & All $q$ & $1$ & $\gen{\delta,\phi}$
\\ $(2^2\times \POmega_8^+(q)\cdot 2^2)\cdot \Sym(3)$ & $p\geq 3$ & $q\equiv\pm 1 \bmod 8$ & $2$ &$\gen{\phi}$
\\ $(2^2\times \POmega_8^+(q)\cdot 2^2)\cdot 3$ & $p\geq 3$ & $q\equiv\pm 3 \bmod 8$ & $1$ &$\gen{\delta,\phi}$
\\ ${}^3\!D_4(q)\cdot 3$ & $p\geq 3$ & All $q$ & $2$ & $\gen{\phi}$
\\ $E_7(q_0)\cdot \gcd(d,r)$ & All $p$ & $q=q_0^r$, $r$ prime & $\gcd(d,r)$ & $\gen{\delta^{\gcd(d,r)},\phi}$
\\ \hline
\end{tabular}
\end{center}
\caption{Maximal subgroups not in $\mathcal S$ for $G=E_7(q)$. Here, $d=\gcd(2,q-1)$, $e=\gcd(3,q-1)$, $e'=\gcd(3,q+1)$, $f=\gcd(4,q-1)$, $f'=\gcd(4,q+1)$, $g=\gcd(8,q-1)$, $g'=\gcd(8,q+1)$.}
\label{t:othermaximals}
\end{table}

\section{Determination of the almost simple subgroups}
\label{sec:detofgroups}

We now turn our attention to all almost simple groups that embed in $E_7$ for some prime $p$, but that do not come from a positive-dimensional subgroup as in the previous section. We split them up into families, and proceed through alternating groups, sporadic groups, groups of Lie type other than $\PSL_2(r)$ (where $r$ may be a power of $p$), and finally the groups $\PSL_2(r)$ (where again $r$ may be a power of $p$). In each case, and for each quasisimple group $H$, we consider the possible embeddings of $H$ into $\mb G=E_7$, whether there are any such subgroups that are not strongly imprimitive, and if so, determine the number of $\mb G$-conjugacy classes, determine $N_{\mb G}(H)$, and compute the actions of outer automorphisms of $G$ on the $G$-classes of subgroups $\bar H=HZ(\mbG)/Z(\mb G)$, whenever $\bar H\leq G$. (Recall our convention that if $H\cap Z(\mbG)=1$ then we identify $H$ and $\bar H$ to reduce notation, so $H\leq G$ in this case.)

The collection of all simple groups that embed in $\mbG$, other than groups of Lie type in characteristic $p$, is given in \cite{liebeckseitz1999}.\footnote{Note that ${}^2\!B_2(8)$, $p=5$ is missing from \cite{liebeckseitz1999}. It embeds in $A_7$.} If we consider Lie type groups $H$ in characteristic $p$, it is proved in \cite[Theorem 3]{liebeckseitz1990}, and then later improved in \cite[Theorem 1]{liebeckseitz1998} and \cite[Corollary 5]{liebeckseitz2004} that the only possible almost simple maximal subgroups that are Lie type in characteristic $p$ are either fixed points of proper subgroups (so appear in the previous section) or are $H(q)$ for $H$ of small rank and $q$ a small prime power. These restrictions were further tightened in \cite{craven2015un,craven2019un}, so that either $M$ is listed in Table \ref{t:othermaximals} or $H=N_G(\PSL_2(p^a))$ for $p^a=7,8,25$. We eliminate $p^a=25$ in Section \ref{sec:psl225} below, but $p^a=7,8$ are unresolved, and appear in Table \ref{t:e7stillleft}.

For the other simple groups $H$, we start with \cite{liebeckseitz1999}, and then remove those that were proved to always be strongly imprimitive in \cite{litterickmemoir} for all simple groups, and \cite{craven2017} for alternating groups in particular. This yields the groups in Table \ref{t:e7tocheck}. For each of these we have to consider each set of composition factors for $M(E_7)$ and $L(E_7)$ that appears in the tables of \cite{litterickmemoir} with `\textbf{P}'. For the three Lie type groups in characteristic $p$, the possible module structures were described in \cite{craven2015un}.

\begin{table}
\begin{center}
\begin{tabular}{cc}
\hline Prime & Group
\\\hline $p\nmid |H|$ & $\Alt(6)$, $\PSL_2(r)$, $r=7,11,13,19,27,29,37$, $\PSL_3(4)$, $\PSU_3(3)$, $\PSU_3(8)$, $\Omega_8^+(2)$
\\ $p=2$ & $J_2$, $\PSL_2(r)$, $r=8,11,13,19,27,29,37$
\\ $p=3$ &  $\PSL_2(r)$, $r=7,11,13,19,29,37$, $\PSL_3(4)$, $\PSU_3(8)$, $\Omega_8^+(2)$
\\ $p=5$ & $\Alt(6)$, $\Alt(7)$, $M_{12}$, $M_{22}$, $Ru$, $HS$,  $\PSL_2(r)$, $r=11,19,25,29$, $\PSL_3(4)$, $\Omega_8^+(2)$, ${}^2\!B_2(8)$
\\ $p=7$ & $\PSL_2(r)$, $r=7,8,13,27,29$, $\PSL_3(4)$, $\PSU_3(3)$, $\PSU_3(8)$, $\Omega_8^+(2)$
\\ $p=13$ & $\PSL_2(27)$
\\ $p=19$ & $\PSL_2(37)$, $\PSU_3(8)$
\\ \hline
\end{tabular}
\end{center}
\caption{Simple subgroups of $E_7$ that could yield subgroups that are not strongly imprimitive.}
\label{t:e7tocheck}
\end{table}

We will not prove that any of the subgroups are actually maximal in this section, delaying the proof to Section \ref{sec:proofofmaximality}. The reason for this is that one can prove maximality of all of the cases more or less simultaneously, and collating all proofs in one place reduces repetition.

\subsection{Alternating groups}
\label{sec:altgroups}

We rely heavily on \cite{craven2017} here, which deals with most cases. Let $H\cong \Alt(n)$, or $2\cdot \Alt(n)$ with $Z(H)=Z(\mbG)$. If $n=8$ then $H$ either stabilizes a line on $M(E_7)$ or $L(E_7)^\circ$, or $p=3$ and $H$ contains a `generic unipotent element' in the sense of \cite[Definition 1.1]{craven2017}, so is strongly imprimitive by \cite[Lemma 1.3]{craven2017}. If $n=9$ then for $p\geq 5$ we always have that $H$ is strongly imprimitive by \cite{litterickmemoir}. If $p=3$ then either $H$ stabilizes a line on $L(E_7)$ or contains a generic unipotent element, and if $p=2$ then $H$ stabilizes a line on $M(E_7)$. If $n\geq 10$ then $H$ is strongly imprimitive by \cite{litterickmemoir}. Thus $H$ is strongly imprimitive for $n\geq 8$.

If $n=5$ then for $p=3,5$, $H$ stabilizes a line on either $M(E_7)$ or $L(E_7)$. If $p=2$ then from \cite[Proposition 5.4]{craven2017} we see that $H$ stabilizes a line or $2$-space on $M(E_7)$, and one may then apply \cite[Proposition 4.7]{craven2015un} to obtain that $H$ is strongly imprimitive.

If $p=0$ then in \cite[Proposition 5.1]{craven2017} we see that $H$ cannot be Lie primitive, and $H$ is either strongly imprimitive or $H$ has two isomorphic $2$-dimensional composition factors on $M(E_7)$ (and none of the other isomorphism type). If $C_\mbG(H)\neq Z(\mbG)$ then the double centralizer $C_\mbG(C_\mbG(H))$ is positive dimensional and of course $N_{\Aut^+(\mbG)}(H)$-stable, whence $H$ is strongly imprimitive.

By simple dimension counting, the stabilizer of any $2$-space on $M(E_7)$ has dimension at least $25$ (see, for example, \cite[Lemma 1.4]{craven2017}), and so $H$ must lie in a positive-dimensional subgroup of $\mbG$ of dimension at least $25$. Furthermore, $H$ has no trivial composition factors on $L(E_7)$, so $H$ is not contained in a parabolic subgroup of $\mbG$.

Thus if $H$ is not strongly imprimitive then $H$ cannot lie in a parabolic subgroup of $\mbG$, and cannot lie in the centralizer of a semisimple element of $\mbG$, and has to lie in a subgroup of dimension at least $25$. Examining the maximal subgroups of $E_7$ (see, for example, \cite{thomas2016}), this leaves only the $G_2C_3$ and $A_1F_4$ maximal subgroups that could contain $H$. If $H$ lies in $A_1F_4$ then, since $L(F_4)$ is a composition factor of the $A_1F_4$-action on $L(E_7)$, the projection of $H$ on $F_4$ has finite centralizer in $F_4$ as well. But this subgroup is unique up to $F_4$-conjugacy and contained in $A_2A_2$ (this is due to Magaard: see \cite[Theorem 4.17]{frey1998}), which centralizes an element of order $3$ in $F_4$, hence in $\mbG$.

Finally, if $H$ lies in $G_2C_3$ then we see from the list of maximal subgroups of $G_2$ \cite{kleidman1988} that the image of $H$ in $G_2$ either centralizes an element of order $2$ or $3$, or lies inside the maximal $A_1$ acting irreducibly as $L(6)$ on $M(G_2)$. But $A_1C_3$ has dimension $24$, so cannot contain any $2$-space stabilizer. Thus $H$ must be contained in some other maximal positive-dimensional subgroup of $\mbG$, and this completes the proof.

\medskip

This leaves $n=6$ and $n=7$. In the latter case, $H$ is always strongly imprimitive if $p\neq 5$ by \cite[Propositions 7.1, 7.2, 7.4 and 7.5]{craven2017}, so we need only look at $p=5$. If $n=6$ and $p=2,3$ then $H$ stabilizes a line on $M(E_7)$ or $L(E_7)$, or a $2$-space on $M(E_7)$ in the case $H\cong 2\cdot \Alt(6)$ and $p=3$. To obtain strong imprimitivity we can again apply \cite[Proposition 4.7]{craven2015un} for the third possibility.

Thus we are left with $\Alt(7)$, $p=5$ (which we can resolve) and $\Alt(6)$, $p\geq 5$ (which we cannot), and we discuss the latter in Section \ref{sec:alt6}.

\subsubsection{\texorpdfstring{$\Alt(7)$}{Alt(7)}}
\label{sec:alt7}
From \cite[Section 7]{craven2017}, we see that there are two possibilities for $H$ a cover of $\Alt(7)$, both of which require $p=5$.

We start with the possibility of $H\cong 2.\Alt(7)$ in $\mbG$. Here, the actions on $M(E_7)$ and $L(E_7)$ are
\[ 20\oplus (4/14)\oplus (14/4^*),\quad 35^{\oplus 3}\oplus 10\oplus 10^*\oplus 8.\] As we see from \cite[Table 6.128]{litterickmemoir}, this is the only possible action of $H$ on $M(\mbG)$ and $L(\mbG)$. We therefore see that the copy of $2\cdot \Alt(7)$ in $A_2A_5$ (which is $3\cdot \Alt(7)$ in $A_2$ and $6\cdot \Alt(7)$ in simply connected $A_5$, acting irreducibly on each minimal module) also has this action.

\begin{prop}\label{prop:2alt7} If $H\cong 2\cdot \Alt(7)$ is a subgroup of $\mbG$ then $H$ is strongly imprimitive.
\end{prop}
\begin{proof} We utilize the subalgebra method from Section \ref{sec:subalgmethod} to prove uniqueness of $H$ up to conjugacy. In order to do this, first we need to determine $H$ up to conjugacy in $\Sp_{56}(k)$. Any bilinear form on $20\oplus (4/14)\oplus (14/4^*)$ is a direct sum of a form on the $20$ and a form on the other summand, thus lies inside a subgroup $\Sp_{20}(k)\times \Sp_{36}(k)$. The $20$ supports a unique alternating form, so given a form for the $36$-dimensional summand, one obtains a unique subgroup of $\Sp_{56}(k)$ up to conjugacy by Proposition \ref{prop:ssnofactorsincommon}. Thus if $H$ is unique up to conjugacy in $\Sp_{36}(k)$ with that action then $H$ is unique up to conjugacy in $\Sp_{56}(k)$, as needed.

Since this action of $H$ is the extension with quotient $4\oplus 14$ and submodule its dual, we see that $H$ stabilizes a---necessarily totally isotropic---$18$-space. Thus $H$ is contained in an $A_{17}$-parabolic subgroup of $\Sp_{36}(k)$. One wishes to apply Corollary \ref{cor:symp1cohom}, so one requires the $1$-cohomology of the module $S^2(4^*\oplus 14)$, which by a computer calculation (or even a hand calculation if you so desire) is $1$-dimensional. (The $1$-cohomology of $S^2(4\oplus 14)$ on the other hand is $0$-dimensional, which tells you which you need to take.)

Thus $H$ is unique up to $\Sp_{56}(k)$-conjugacy. Applying the subalgebra method in the supplementary materials yields a single orbit under the action of $C_{\Sp_{56}(k)}(H)$. Although there are two representations of $H$ of our form, they are $\Aut(H)$-conjugate, so yield a single class of subgroups of $\mbG$.

Thus $H$ is contained in $A_2A_5$ as there is also a copy of $H$ there. The subgroup $\mb X=A_2A_5$ has summands $(10,\lambda_1)$, $(01,\lambda_5)$ and $(00,\lambda_3)$ on $M(E_7)$, i.e., the product of the minimal modules, its dual, and the exterior cube of $M(A_5)$. These summands have dimensions $18$, $18$ and $20$, so that $H$ and $\mb X$ stabilize the same $20$-space. Furthermore, this is clearly stable under $N_{\Aut^+(\mb G)}(H)$, so $H$ is strongly imprimitive by Proposition \ref{prop:intersectionsubspace}.
\end{proof}

The other possibility is for $H\cong \Alt(7)$, acting on $M(E_7)$ and $L(E_7)$ as
\[ (10\oplus 10^*)^{\oplus 2}\oplus 8^{\oplus 2},\quad 10\oplus 10^*\oplus P(8)^{\oplus 3}\oplus 8.\]
There is a copy of this inside $D_4$, inside the maximal-rank subgroup $A_7$ of $E_7$. It turns out that this is the only conjugacy class.

\begin{prop} If $H\cong \Alt(7)$ embeds in $\mbG$ then $H$ is strongly imprimitive.
\end{prop}
\begin{proof} As stated above, the only case to consider is $p=5$ and $H$ acting as above. The $8$ carries a symmetric form, so there is a unique symplectic form on $8^{\oplus 2}$. There is also a unique symplectic form on $W=(10\oplus 10^*)^{\oplus 2}$, up to conjugacy. To see this, notice that the centralizer in $\GL(W)$ is isomorphic to $\GL_2(k)\times \GL_2(k)$, and since $H$ and its centralizer in $\Sp(W)$ must lie in the $A_{19}$ subgroup of $\Sp(W)$, the centralizer in $\Sp(W)$ is $\GL_2(k)$. Finally, the set of symplectic forms on $W$ is $4$-dimensional, and the non-singular ones are in natural bijection with $\GL_2(k)$ (one can see this by considering the fixed points of the exterior square of $W$, or simply use a computer, and we check this in the supplementary materials). We see that the centralizer in $\GL(W)$ must act transitively on the non-singular symplectic forms, and therefore $H$ is unique in $\Sp(W)$ up to conjugacy. Thus $H$ is unique in $\Sp_{56}(k)$ up to conjugacy as well.

We therefore may apply the subalgebra method, and prove that $H$ always stabilizes a $28$-dimensional subalgebra $\mathfrak h$ of type $\mathfrak d_4$. It seems difficult to proceed further and find the subalgebra of type $\mathfrak a_7$ that is stabilized, or the $\mathfrak e_7$ above it. Thus we proceed differently, understanding subalgebras of type $\mathfrak d_4$ in the Lie algebra $\mathfrak g=\mathfrak{e}_7$.

We consider the composition factors of the action of $\mathfrak h$ on $\mathfrak g$. Using the eigenvalues of semisimple elements, one easily sees that there are exactly two sets of dimensions of composition factors, namely $35^3,28$ (the irreducible $\mathfrak d_4$ inside $\mathfrak a_7$) and $28,8^{12},1^9$ (the $\mathfrak d_4$-Levi). As $H$ stabilizes the homogeneous components of the action of $\mathfrak h$ on $\mathfrak g$, these dimensions must be compatible with the composition factors of $H$ on $\mathfrak g$, which have dimensions $13^3,10^2,8^7,6^3$. One cannot make $9$ from these dimensions, so $\mathfrak h$ acts with factors of dimension $35^3,28$.

Since $\mathfrak h$ has no trivial factor on $\mathfrak g$, it cannot be contained in a parabolic. Also, if $\mathfrak h$ is contained in a maximal subalgebra $\mathfrak x_1\mathfrak x_2$ then both $\mathfrak x_i$ must contain a copy of $\mathfrak h$. This is impossible, so $\mathfrak h$ can only be contained in simple subalgebras. By \cite[Theorem 1.1]{premetstewart2019}, and the list of maximal $\mb G$-irreducible subgroups (see, for example, \cite{liebeckseitz2004} or \cite{thomas2016}), we see that $\mathfrak h$ is contained in $\mathfrak a_7$. Any $8$-dimensional irreducible representation of $\mathfrak d_4$ integrates to a representation of $D_4$, and so $\mathfrak h$ is $\mathrm{Lie}(\mb H)$ for $\mb H$ the $\mbG$-irreducible $D_4$.

Thus $H$ is contained in this subgroup. At this point it is relatively easy to prove strong imprimitivity, and there are multiple paths. One option is to place $H$ inside an $A_2$ subgroup $\mb X$ of $\mb H$ (subgroup 24 from the list in \cite{thomas2020}). This must act semisimply on $M(E_7)$ as $(10\oplus 10^*)^{\oplus 2}\oplus 8^{\oplus 2}$, and so $H$ is a blueprint for $M(E_7)$. Another is to note that $H$ embeds in $\Omega_8^+(5)$, and so Frobenius endomorphisms can only act as inner automorphisms. Either way, $H$ is strongly imprimitive, as needed.
\end{proof}

\subsection{Sporadic groups}

The sporadic groups that appear in \cite{litterickmemoir} with `\textbf{P}' are $J_2$ for $p=2$, and $2\cdot M_{12}$, $2\cdot M_{22}$, $2\cdot HS$ and $2\cdot Ru$, all for $p=5$. The groups $M_{22}$, $HS$ and $Ru$ are all irreducible on $L(E_7)$, and will be proved to be unique up to $\mbG$-conjugacy.

We start with $HS$ and $Ru$, and then we can do $M_{22}$ (as it is contained in $HS$). We obtain a previously (the author believes) unknown maximal subgroup $M_{12}$, and prove that $J_2$ does not yield a maximal subgroup via theoretical means.

In this section, $H$ will be a subgroup of $\mbG$ (and $G$) and $\bar H$ will denote its image in $\mbG/Z(\mbG)$ (and $\bar G$), unless $H\cong \bar H$, in which case we simply use $H$. Proofs of maximality will be delayed until Section \ref{sec:proofofmaximality}.

\subsubsection{\texorpdfstring{$HS$}{HS}}

Let $H$ be a cover of $HS$, so we require $p=5$ and $H\cong 2\cdot HS$. Up to $\Aut(H)$-conjugacy there is a unique $133$-dimensional simple module $M$, defined over $\F_5$, and $\mathcal{R}^H(M)$ is $1$-dimensional. Since it is known that $HS$ embeds in $E_7(5)$ by \cite{kmr1999}, this means that $H$ is unique up to $\mbG$-conjugacy. (Uniqueness is not noted in \cite{kmr1999}.) An independent construction that $HS$ is contained in $E_7(5)$ is given in the supplementary materials, as an intermediate step in the proof that $M_{22}.2$ is a novelty maximal subgroup. Proof of maximality is delayed until Section \ref{sec:proofofmaximality}.

\begin{prop} Let $H\cong 2\cdot HS$ and $p=5$.
\begin{enumerate}
\item There is exactly one $\mbG$-conjugacy class of subgroups $H$. If $\bar H\leq G$ there are exactly two $G$-conjugacy classes of subgroups $\bar H$. Furthermore, $N_\mbG(H)=H$.
\item $\bar H$ embeds in $G$ for all $q$ a power of $5$.
\item We have that $N_{\bar G}(\bar H)$ is maximal in $\bar G$ if and only if $\bar G=E_7(5)$.
\end{enumerate}
\end{prop}
\begin{proof} We have already noted that $H$ is unique up to $\mbG$-conjugacy. Since the outer automorphism of $H$ swaps the two $133$-dimensional simple modules for $H$, we must have $N_\mbG(H)=H$, and therefore there are exactly two $G$-conjugacy classes by Lemma \ref{lem:diagswaps}. Since the Lie product is defined over $\F_5$, $H\leq 2\cdot E_7(5)$, and we will prove maximality in Section \ref{sec:proofofmaximality} below.
\end{proof}

\subsubsection{\texorpdfstring{$Ru$}{Ru}}
\label{sec:ru}
This is almost the same as $HS$. Let $H\cong 2\cdot Ru$, so we require $p=5$. Up to isomorphism there is a unique $133$-dimensional simple module $M$, defined over $\F_5$, and $\mathcal{R}^H(M)$ is $1$-dimensional. Since $H$ embeds in $E_7(5)$ \cite{kmr2000}, we see that $H$ is unique up to $\mbG$-conjugacy. (The action on $M(E_7)$ has composition factors $28$ and $28^*$, and is not semisimple. The easiest way to see this is that it contains $\SL_2(29)$, and this is seen not to be semisimple in Section \ref{sec:sl229}.) In \cite{kmr2000} they suggest showing directly that the unique anti-commutative non-associative algebra structure on the $133$-dimensional module forms a Lie algebra, and therefore give a further proof of the existence (and uniqueness) of $Ru$ in $E_7(5)$. For completeness, we do this in the supplementary materials.

\begin{prop} Let $H\cong 2\cdot Ru$ and $p=5$.
\begin{enumerate}
\item There is exactly one $\mbG$-conjugacy class of subgroups $H$. If $\bar H\leq G$ there are exactly two $G$-conjugacy classes of subgroups $\bar H$. Furthermore, $N_\mbG(H)=H$.
\item $\bar H$ embeds in $G$ for all $q$ a power of $5$.
\item We have that $N_{\bar G}(\bar H)$ is maximal in $\bar G$ if and only if $\bar G=E_7(5)$.
\end{enumerate}
\end{prop}
\begin{proof} We have already noted that $H$ is unique up to $\mbG$-conjugacy. Since $\Out(H/Z(H))=1$, $N_\mbG(H)=H$, and therefore there are exactly two $G$-conjugacy classes by Lemma \ref{lem:diagswaps}. Since the Lie product is defined over $\F_5$, $H\leq 2\cdot E_7(5)$, and the proof of maximality is delayed until Section \ref{sec:proofofmaximality}.
\end{proof}

\subsubsection{\texorpdfstring{$M_{12}$}{M12}}
\label{sec:m12}
From \cite[Tables 6.145--6.150]{litterickmemoir}, we see that $M_{12}$ or $2\cdot M_{12}$ always stabilizes a line on either $M(E_7)$ or $L(E_7)^\circ$, unless $p=5$ and $H\cong 2\cdot M_{12}$ acts as $55\oplus 78$ on $L(E_7)$. The space $\mathcal{R}^H(L(E_7))$ is $12$-dimensional, and this case can be solved using the subalgebra method from Section \ref{sec:subalgmethod}, as we saw in Example \ref{ex:m12}. The corresponding factors for $H$ on $M(E_7)$ are $32,12^2$. Here we shall use the Lie product method from Section \ref{sec:lieprodmethod}.

The action of $u\in H$ of order $5$ on $L(E_7)$ has Jordan blocks $5^{26},3$, so by \cite[Table 8]{lawther1995}, we see that $u$ belongs to the unipotent class $A_4+A_2$. Since this acts on $M(E_7)$ with blocks $5^{10},3^2$, and on the sum $32\oplus 12^{\oplus 2}$ with blocks $5^{10},1^6$, we see that $M(E_7){\downarrow_H}$ has structure $12/32/12$.

Note that no extension of this module, or indeed the simple modules $12$ and $32$, exists over $\F_5$, and one must move to $\F_{25}$ in order to obtain it. Thus, while $H$ might be contained in $2\cdot E_7(5)$, the group $H.2$ cannot be contained in $2\cdot E_7(5)$. Proof of maximality is delayed until Section \ref{sec:proofofmaximality}.

\begin{prop} Let $H\cong 2\cdot M_{12}$ and $p=5$.
\begin{enumerate}
\item There is exactly one $\mbG$-conjugacy class of subgroups $H$. If $\bar H\leq G$ there is exactly one $G$-conjugacy class of subgroups $\bar H$. Furthermore, $N_\mbG(H)=H.2$.
\item $\bar H$ embeds in $G$ for all $q$ a power of $5$.
\item We have that $N_{\bar G}(\bar H)$ is maximal in $\bar G$ if and only if $q=5$.
\end{enumerate}
\end{prop}
\begin{proof} The action of $H$ on $L(E_7)$ is $M\cong 55\oplus 78$, and $\mathcal{R}^H(M)$ is $13$-dimensional. It is easy to prove, as is accomplished in the supplementary materials, that there is a unique Lie product structure on $M$. We also show that this action extends to one of $H.2$.

The action of $H$ on $M(E_7)$ has composition factors $32,12^2$. There are two isoclinic groups $2.M_{12}.2$, and for both of these the extensions of $12$ and $32$ to the whole group requires $\F_{25}$ to be present, not just $\F_5$. (See \cite[p.76]{abc} for one of these classes. Only one of these extensions results in a subgroup of $\Sp_{56}(k)$.) Thus the group $H.2$ cannot embed in $2\cdot E_7(5)\leq \GL_{56}(5)$. Thus $H$ is unique up to conjugacy, with the diagonal automorphism of $E_7(5)$ normalizing $\bar H$ (up to conjugacy).
\end{proof}

\subsubsection{\texorpdfstring{$M_{22}$}{M22}}

Let $H\cong 2\cdot M_{22}$, so we require $p=5$. Up to isomorphism there is a unique $133$-dimensional simple module $M$, defined over $\F_5$, and $\mathcal{R}^H(M)$ is $2$-dimensional. Since $M_{22}$ is contained in $HS$ \cite[p.80]{atlas} certainly $M_{22}\leq E_7(5)$. Thus there is a line in $\mathcal{R}^H(M)$ that consists of Lie products. The proof of maximality is delayed until Section \ref{sec:proofofmaximality}.

\begin{prop}\label{prop:m22} Let $H\cong 2\cdot M_{22}$ and $p=5$.
\begin{enumerate}
\item There is exactly one $\mbG$-conjugacy class of subgroups $H$. If $\bar H\leq G$ there is exactly one $G$-conjugacy class of subgroups $\bar H$. Furthermore, $N_\mbG(H)=H.2$.
\item $\bar H$ embeds in $G$ for all $q$ a power of $5$.
\item We have that $N_{\bar G}(\bar H)$ is maximal in $\bar G$ if and only if $\bar G=E_7(5).2$, and $N_{\bar G}(\bar H)$ is a novelty maximal subgroup.
\end{enumerate}
\end{prop}
\begin{proof}  We show in the supplementary materials that $\mathcal{R}^H(M)$ is $2$-dimensional, but that imposing the Jacobi identity reduces the space to at most a line (or zero). Since $M_{22}$ definitely embeds in $\mbG$, this means that $H$ is unique up to $\mbG$-conjugacy. In the supplementary materials we also prove that $M_{22}.2$ embeds in $E_7(5).2$, so $N_\mbG(H)=H.2$. Since the Lie product is defined over $\F_5$, $H\leq 2\cdot E_7(5)$. We need to know whether all of $N_\mbG(H)$ is contained in $G_{\mathrm{sc}}=2\cdot E_7(5)$ or whether the diagonal automorphism of $E_7(5)$ induces the outer automorphism of $\bar H$.

Note that $H$ is contained in $2\cdot HS$, but if one inspects \cite[p.80]{atlas}, one sees that $M_{22}.2$ does not embed in $HS$. Both $M_{22}$ and $HS$ act irreducibly on $L(E_7)$. If $x\in E_7(5).2$ induces an outer automorphism on $\bar H$, and $L\cong HS$ is a subgroup of $E_7(5)$ containing $\bar H$ then $L^x$ also contains $\bar H$.

In the supplementary materials we construct a copy of $2\cdot M_{22}$, and two copies of $2\cdot HS$ above it, that generate $2\cdot E_7(5)$. We then add an outer automorphism of $2\cdot M_{22}$ and check that the resulting group is $2\cdot E_7(5).2$, not $2\cdot E_7(5)$. This shows that $N_\mbG(H)=H.2$ and there is a unique $G$-class of subgroups $\bar H$, via Lemma \ref{lem:diagswaps}. Thus the diagonal automorphism of $E_7(5)$ must normalize $\bar H$, hence act as the outer automorphism, as needed.
\end{proof}

\subsubsection{\texorpdfstring{$J_2$}{J2}}

Here, $p=2$. From \cite[Table 6.157]{litterickmemoir}, we see that the only case of interest is $H\cong J_2$ acting with composition factors $36,6_1^2,6_2,1^2$ on $M(E_7)$, and on $L(E_7)^\circ$ with factors
\[ 84,14_1^2,14_2,6_2.\]
Note that $\Ext^1_{kH}(14_1,14_2)=\Ext^1_{kH}(14_1,6_1)=0$, so $14_2$ and $6_1$ split off as summands, and there are no self-extensions for $14_1$. Thus there are only two options for $L(E_7){\downarrow_H}$:
\[ 84\oplus 14_1^{\oplus 2}\oplus 14_2\oplus 6_2\quad\text{and}\quad (14_1/84/14_1)\oplus 14_2\oplus 6_1.\]
The action of $v$ of order $8$ in $H$ on both of these modules is the same, with Jordan blocks $8^{13},6^4,2^2$. Thus from \cite[Lemma 6.3]{craven2017}, the action of $v$ on $L(E_7)$ must be $8^{13},6^4,2^2,1$, but this does not appear in \cite[Table 8]{lawther1995}. Thus $H$ cannot embed in $\mbG$ with these composition factors.

\begin{prop} If $p=2$ and $H\cong J_2$ then $H$ stabilizes a line on either $M(E_7)$ or $L(E_7)^\circ$, and hence is strongly imprimitive.
\end{prop}

\subsection{Lie type not \texorpdfstring{$\PSL_2(r)$}{PSL(2,r)}}

Now let $H$ be a (cover of a) group of Lie type not of type $\PSL_2(r)$, which appear in Table \ref{t:e7tocheck}. If $H\cong \PSU_3(8)$ then $H$ acts irreducibly on $M=L(E_7)$, and is unique up to $\mbG$-conjugacy since $\mathcal R^H(M)$ is $1$-dimensional. The remaining cases are: $2\cdot \PSL_3(4)$ for $p\neq 2$; $\PSU_3(3)$ for $p=0,7$; $\Omega_8^+(2)$ for $p\neq 2$; and ${}^2\!B_2(8)$ for $p=5$.

In this section, $H$ will be a subgroup of $\mbG$ (and $G_{\mathrm{sc}}$) and $\bar H$ will denote its image in $\mbG/Z(\mbG)$ (and $\bar G$). Maximality proofs are completed in Section \ref{sec:proofofmaximality}.

\subsubsection{\texorpdfstring{$\PSL_3(4)$}{PSL(3,4)}}

Here the possible primes are $p=3$, $p=5$, $p=7$ and $p\neq 2,3,5,7$. From \cite[Tables 6.214--6.217]{litterickmemoir}, we see that the only cover of $\PSL_3(4)$ that embeds in $\mbG$ is $H\cong 2\cdot\PSL_3(4)$. For $p\neq 3$, the action of $H$ on $L(E_7)$ is the sum of two non-isomorphic $35$-dimensional modules and a $63$-dimensional module. For $p=3$ the action is slightly different: the sum of a $63$-dimensional module and a module $19/1,1,15_1,15_2/19$, as we shall see now.

Let $p=3$. From \cite[Table 6.217]{litterickmemoir} we see that composition factors of $L(E_7){\downarrow_H}$ are, up to automorphism of $H$,
\[ 63_1,19^2,15_1,15_2,1^2.\]
Since $\Ext^1_{kH}(19,1)$ has dimension $2$, $H$ need not stabilize a line on $L(E_7)$. The $63_1$ has no extension with $1$ or $19$, so splits off as a summand. Suppose that $H$ does not stabilize a line on $L(E_7)$. Suppose first that no $15_i$ is a submodule of $L(E_7){\downarrow_H}$. There is a module
\[ 1,1,15_1,15_2/19,\]
on which we may place two copies of $19$. If $u\in H$ has order $3$ then $u$ acts on the module above with blocks $3^{16},2,1$, and on $19$ with blocks $3^6,1$. The action of $u$ on the module $19,19/1,1,15_1,15_2/19$ has blocks $3^{29},2$, whence we see that there is a unique module $19/1,1,15_1,15_2/19$ with $u$ acting as $3^{22},2^2$, and on all others as $3^{23},1$. As $u$ acts projectively on $63$, $u$ acts on $L(E_7)$ with blocks either $3^{43},2^2$ or $3^{44},1$. Only the former appears in \cite[Table 8]{lawther1995}, and $u$ belongs to class $2A_2+A_1$. Thus $L(E_7){\downarrow_H}$ is unique up to isomorphism.

If there is at least one $15_i$ as a submodule, as all $15_i$ are $\Aut(H)$-conjugate, we may assume that $15_2$ is a submodule. On top of $1,1,15_1/19$ we may place a single $19$, and the resulting module is self-dual. Thus we obtain a module
\[ (19/1,1,15_1/19)\oplus 15_2\oplus 63_1.\]
The element $u$ acts on this with blocks $3^{44},1$, so this case is excluded.

If $L(E_7){\downarrow_H}$ has both $15_i$ as submodules then $H$ must stabilize a line on $L(E_7)$, as we see from the above modules.

In all cases we have the following result.

\begin{prop} Any subgroup $H\cong 2\cdot \PSL_3(4)$ in $\mbG$ is strongly imprimitive.
\end{prop}
\begin{proof} Suppose that we can place $H$ inside a maximal-rank $A_7$ subgroup $\mb X$. Note that both $H$ and $\mb X$ stabilize the same unique irreducible $63$-dimensional submodule of $L(E_7)$. Thus $H$ is strongly imprimitive by Proposition \ref{prop:intersectionsubspace}.

Thus we use the $53$- or $54$-dimensional (depending on $p$) space $\mathcal{R}^H(L(E_7))$ to prove that there is a unique $H$-invariant Lie product on $L(E_7){\downarrow_H}$. In the supplementary materials we show this for each of $p=3$, $p=5$, $p=7$ and $p\neq 2,3,5,7$. This completes the proof.
\end{proof}

\subsubsection{\texorpdfstring{$\PSU_3(3)$}{PSU(3,3)}}
\label{sec:psu33}
Let $H\cong \PSU_3(3)$, where the appropriate primes are $p=7$ and $p\neq 2,3,7$ (as $H\cong G_2(2)'$). From \cite[Tables 6.219--6.220]{litterickmemoir} we find two sets of factors labelled `\textbf{P}' for $p\neq 2,3,7$, and three for $p=7$. We can eliminate two of these five quickly. One case for $p\neq 2,3,7$ contains an irreducible summand of dimension $14$, which via the inclusion of $H$ into $G_2$ admits a simple Lie algebra structure. Since $\Lambda^2(14)=14\oplus 21_1\oplus 28\oplus 28^*$, and only $14$ from these appears in the composition factors of $L(E_7){\downarrow_H}$, we see that the $14$ forms a $G_2$-subalgebra of $L(E_7)$. This is enough to prove that there is no such Lie primitive subgroup $H$ over $E_7(\mathbb{C})$ (cf. \cite[p.146]{cohenwales1997}), and then we use Larsen's $(0,p)$-correspondence (Theorem \ref{thm:larsencorr}) to deduce the result for all $p\neq 2,3,7$. (This only shows that $H$ is Lie imprimitive, not strongly imprimitive. However, both $H$ and the $G_2$ subgroup stabilize the unique $14$-dimensional summand, so $H$ is strongly imprimitive by Proposition \ref{prop:intersectionsubspace}.)

The other case we can immediately remove is for $p=7$. Here only $26$ has non-trivial $1$-cohomology, and the structure of the projective is
\[ 26/1/26/1,6/26.\]
Examining \cite[Table 6.220]{litterickmemoir}, we see that Case 4 has four $26$s and three $1$s. From the structure of the projective above, it is clear that we cannot build such a module without a trivial submodule or quotient.

One of the remaining three cases is the above case with a $14$-dimensional summand, but for $p=7$ so we cannot use the previous method.

\begin{prop} Let $H\cong \PSU_3(3)$. If $p=7$ and $L(E_7){\downarrow_H}$ has composition factors
\[ 26^3,21_1,14,7_1,6^2,1\]
then $H$ is unique up to $\mbG$-conjugacy and lies in an $A_6$-subgroup of $\mbG$. In particular, $H$ is a blueprint for $M(E_7)$ so is strongly imprimitive.
\end{prop}
\begin{proof} The composition factors of $H$ on $M(E_7)$ are $7,7^*,21_2,21_2^*$. As each of these modules is projective, there are no extensions between them, so $M(E_7){\downarrow_H}$ is semisimple. The alternating form is unique up to $\Sp_{56}(k)$-conjugacy by Proposition \ref{prop:ssmultfree}.

In the supplementary materials we show that $H$ is unique up to $\mbG$-conjugacy, using the subalgebra method. On the other hand, since $H$ has a $7$-dimensional irreducible module there must be another copy $\hat H$ lying in the $A_6$-Levi subgroup, and this must act on $M(E_7)$ as $7\oplus 21_2\oplus 21_2^*\oplus 7^*$. Since $H$ is unique up to $\mbG$-conjugacy, we have that $H$ and $\hat H$ are conjugate. However, $A_6$ itself acts as the sum of four non-isomorphic modules of dimensions $7,7,21,21$, so $H$ and $A_6$ stabilize the same subspaces of $M(E_7)$. Thus $H$ is a blueprint for it, as claimed.
\end{proof}

The remaining two cases are where $H$ acts on $M(E_7)$ as $28\oplus 28^*$ and $p\neq 2,3,7$ or $p=7$. In this case there is a single $\mbG$-class of subgroups, and the difficulty is in understanding the action of the field automorphism on the class. Proof of maximality is delayed until Section \ref{sec:proofofmaximality}.

\begin{prop} Let $H\cong \PSU_3(3)$, let $p\neq 2,3$, and suppose that $M(E_7){\downarrow_H}\cong 28\oplus 28^*$.
\begin{enumerate}
\item There is exactly one $\mbG$-conjugacy class of subgroups $H$. If $H\leq G$ there are exactly two $G$-conjugacy classes of subgroups $H$. Furthermore, $N_\mbG(H)=H\times Z(\mbG)$.
\item $H$ embeds in $G$ if and only if $q\equiv \pm 1\bmod 8$.
\item If $N_{\bar G}(H)$ is maximal in $\bar G$ then $q$ is minimal such that $H$ embeds in $G$, and either $G=\bar G$ or $\bar G$ induces a field automorphism on $G$.
\end{enumerate}
\end{prop}
\begin{proof} The proof proceeds in stages.

\paragraph{Determination of $N_\mbG(H)$} This is surprisingly difficult to do. In \cite[p.15]{atlas} we see that the $28\oplus 28^*$ extends to the full automorphism group $H.2$ of $H$ as an irreducible $56$-dimensional module, but this irreducible module preserves a \emph{symmetric} bilinear form, not an alternating one. (The same holds for $p=7$, see \cite[p.24]{abc}.) Thus either $N_\mbG(H)$ is $H\times Z(\mbG)$ or of the form $H\cdot 4$. The latter subgroup does embed in $\Sp_{56}(p)$, as a simple computer construction proves, and we do this in the supplementary materials as part of the proof (the subgroup is of course inside the group $(\PSU_3(3)\times 2)\wr 2$). We will prove that this outer automorphism of $H$ interchanges two $\mathfrak e_7$-subalgebras for $p=7$, and hence it does not lie in $\mbG$ for all $p\neq 2,3$ by Proposition \ref{prop:p=0impallp}. We can see no `internal' way to prove it does not embed in $\mbG$.

\paragraph{Determination of $H$ up to conjugacy} We apply the subalgebra method from Section \ref{sec:subalgmethod}, noting that since $M(E_7)$ is the sum of two dual modules, the alternating form is unique. In the supplementary materials we prove that for both $p=7$ and $p\neq 2,3,7$ there is exactly one $\mbG$-conjugacy class of subgroups $H$.

\paragraph{Action of outer automorphisms} By Lemma \ref{lem:diagswaps} the diagonal automorphism must fuse two $G$-classes of subgroups $H$. It remains to decide on the action of $\sigma=F_p$. Of course, $\sigma$ normalizes the unique $\mbG$-class, so either $\sigma$ acts as an outer automorphism on a representative $H$ (by Lemma \ref{lem:fixedpoints}) or as an inner automorphism, hence centralizes some conjugate of $H$ by Lemma \ref{lem:inneriscent}. 

First, note that the centralizer in $\Sp_{56}(k)$ of $H$ acts on the $H$-invariant $\mathfrak e_7$-subalgebras and there are two orbits. (This must be the case as the centralizer times $H$ has index $2$ in the normalizer of $H$, and there is a single $\mbG$-class of subgroups $H$. We also prove it in characteristic $0$ over the ring $\mathbb Z[\zeta]$, where $\zeta$ is a primitive $8$th root of unity.)

We can see that $\sigma$ acts as an outer automorphism on $H$ if and only if it swaps these two orbits (which are also interchanged by the normalizer of $H$). First, by a direct calculation in the supplementary materials, we check this---and therefore that the statement holds---for $p=5,7,11,23$. (We must consider $p=7$ separately as this divides $|H|$, and $p=5,11,23$ cover the congruences $3$, $5$ and $7$ modulo $8$ for $p\nmid |H|$.)

Next, we must work over an extension of $\Z$. The $28$-dimensional simple modules are induced from elements of order $8$ in the group of $1$-dimensional modules for the Borel subgroup $3^{1+2}\rtimes 8$ of $H$. Thus there is a natural description of $M(\mbG){\downarrow_H}$ over the ring $\Z[\zeta]$, where as above $\zeta$ is a primitive $8$th root of $1$. Thus in the subalgebra method, we can construct $\Z[\zeta]H$-modules (actually $\Z[\zeta]H$-lattices because they are free over $\Z[\zeta]$) for $S^2(M(\mbG){\downarrow_H})$. We can also construct the modules $21$, $21^*$, $27$ and $32\oplus 32^*$ over $\Z[\zeta]$, so we obtain a module for $L(\mbG){\downarrow_H}$. Thus we can run the subalgebra method over $\Q(\zeta)$ (as we need a field for Magma to construct $\Hom$-spaces and so on).

Doing so, in the supplementary materials, yields points that generate exactly two candidate $\mathfrak e_7$-subalgebras up to the action of the centralizer, and they are defined over $\Z[\zeta]$ by clearing denominators. Since, over any algebraically closed field $k$ of characteristic not $2$, $3$ or $7$, there must be exactly two solutions (up to the centralizer), these must be the solutions for any such $k$. (As a sanity check, the primes dividing the denominators of the points above when they were originally constructed, before clearing denominators, were $2$ and $7$.)

Thus we see that the two orbits of $\mathfrak e_7$-invariant subalgebras are invariant under any field automorphism that fixes all $8$th roots of unity in $k$. Thus if $q\equiv 1\bmod 8$ then the result is proved. It remains to understand the action of the elements of the Galois group of $\Q(\zeta)$ over $\Q$ on the solutions, whether they swap or stabilize the solutions.

Of course, there are four elements in this Galois group, given by $\zeta\mapsto\zeta^i$ for $i=1,3,5,7$. If an element of the Galois group $\zeta\mapsto\zeta^i$ stabilizes the two subalgebras then any field automorphism of $k$ that induces the same map on the roots of unity (e.g., $F_p$ for $p\equiv i\bmod 8$) also stabilizes the two classes.

Now we use our three cases $p=5,11,23$, which we have computed in separate files in the supplementary materials. We show that the two orbits are swapped for $p=5,11$, and so $\zeta\mapsto\zeta^3$ and $\zeta\mapsto\zeta^5$ swap the two orbits of subalgebras. On the other hand, for $p=23$ the two orbits are stabilized, so $\zeta\mapsto\zeta^{-1}$ must stabilize the two orbits. This completes the proof.
\end{proof}

\subsubsection{\texorpdfstring{$\PSU_3(8)$}{PSU(3,8)}}

Let $H\cong \PSU_3(8)$, where the appropriate primes are $p=3$, $p=7$, $p=19$, and $p\neq 2,3,7,19$. This subgroup acts irreducibly on both $M(E_7)$ and $L(E_7)$ for all $p\neq 2$. The character of this for $p=0$ (and therefore its reduction modulo $p$ if $p\neq 0$) appears in \cite[pp.64--66]{atlas}. Existence, and uniqueness at least  for $p\nmid |H|$, can be found in \cite{griessryba1994}, but we give an independent verification, and extend to all odd primes. We delay the maximality proof until Section \ref{sec:proofofmaximality}.

\begin{prop} Let $H\cong \PSU_3(8)$ and $p\neq 2$.
\begin{enumerate}
\item There is a unique $\mbG$-conjugacy class of subgroups isomorphic to $H$, and $N_\mbG(H)=(Z(\mbG)\times (H.3)).2$.
\item If $q\equiv \pm 1\bmod 8$ then there are exactly two $G$-conjugacy classes of subgroups $H$ in $G=E_7(q)$, and $N_\mbG(H)=N_{G_\mathrm{sc}}(H)$. If $q\equiv\pm 3\bmod 8$ then there is a single $G$-conjugacy class of subgroups $H$ in $G=E_7(q)$, and $N_G(H)=H.3$.
\item We have that $N_{\bar G}(\bar H)$ is maximal in $\bar G$ if and only if $q=p$, and furthermore $\bar G=G$ when $p\equiv \pm 1\bmod 8$.
\end{enumerate}
\end{prop}
\begin{proof} In the supplementary materials we show that $\mathcal R^H(M)$ is $1$-dimensional for $M$ a $133$-dimensional simple module for $H$, and any odd prime (or $p=0$). We also prove that the alternating product forms a Lie algebra for $p=5$, thus proving that $\PSU_3(8)\leq E_7(5)$ (and verifying the result from \cite{griessryba1994}). Thus $H$ is unique up to $\mbG$-conjugacy. From \cite[pp.64--66]{atlas} we see that the character of $H$ on $M(E_7)$ and $L(E_7)$ is rational, and that the stabilizer of the $133$-dimensional character for $H$ is the group $H.6$. The trace of a 2B element of $H$ on $L(E_7)$ (which exists in $H.2$) is $0$. Thus $H.2$ does not embed in $\mbG$ (as the traces of non-central involutions are $\pm 8$). Thus if $|N_\mbG(H)/HZ(\mbG)|$ is even then the extension by $2$ does not split over $Z(\mbG)$.

The stabilizers of the $133$-dimensional modules are (in Atlas notation) $H.3_1$ (it stabilizes all three), and $H.2$ (which stabilizes one of the three). (The first group is not $\PGU_3(8)$, which is $H.3_2$.) Since the inner product of $\Lambda^2(M)$ and $M$ is $1$, one may simply check whether the same holds for some extension of $M$ to $H.6$. A simple GAP or Magma calculation (or, for the keen reader, a hand calculation using \cite{atlas}) shows that the inner product for a single extension of $M$ to $H.6$ is $1$ for $p=0$, and we do this in the supplementary materials for completeness. Hence $H.6$ embeds in $\mbG/Z(\mbG)$ for $p=0$. The same therefore holds for all odd $p$ by Proposition \ref{prop:p=0impallp}. We can check via computer, or using \cite{atlas}, that the particular extension $H.3.2$ that stabilizes the $133$ is the group $H.6$ (as opposed to a group $H.\Sym(3)$), and we also do this in the supplementary materials.

It is left to check the actions of automorphisms of $\mbG$ on $H$. Since $H$ is unique up to $\mbG$-conjugacy, all automorphisms stabilize the $\mbG$-class. Since the character is rational, field automorphisms must stabilize the character and thus centralize $N_{\mbG}(H)$. Thus $H.3$ embeds in $E_7(p)$ for all odd $p$, as claimed in the proposition.

The last question is whether the full group $\bar H.6$ embeds in $G=E_7(q)$, or just $\bar H.3$. The unique extension $(H\times 2).2$ is easily constructible in Magma (for example, inside the wreath product), and so the extension of $M(E_7){\downarrow_H}$ to this group can be computed. It turns out that this character is no longer rational \cite[pp.64--66]{atlas}, but has an irrationality $\zeta_8+\zeta_8^{-1}$, where $\zeta_8$ is a primitive $8$th root of unity. If $q\not\equiv \pm 1\bmod 8$ then this does not lie in $\F_q$, and so clearly $(H\times 2).2$ cannot embed in $2\cdot E_7(q)$, as claimed.

If $q\equiv \pm 1\bmod 8$ then $(H\times 2).2$ embeds in $\Sp_{56}(q)$, at least. Since $H$ acts irreducibly on $M(E_7)$, there exists a unique group $X\cong 2\cdot E_7(q)$ in $\Sp_{56}(q)$ containing a given copy of $H$. Thus if $g\in \Sp_{56}(k)$ normalizes $H$ then it normalizes $X$. Letting $g\in \Sp_{56}(q)$ normalize $H$ and induce $H.2$, we see that $g$ normalizes $X$; but $\gen{X,g}$ cannot be $2\cdot E_7(q).2$ because this group does not embed in $\Sp_{56}(q)$ (and only in $\Sp_{56}(q).2$). Thus $g\in X$, as needed.

We delay until Section \ref{sec:proofofmaximality} the proof of maximality.\end{proof}

\subsubsection{\texorpdfstring{$\Omega_8^+(2)$}{O+(8,2)}}

Let $H\cong \Omega_8^+(2)$, where the appropriate primes are $p=3$, $p=5$, $p=7$, and $p\neq 2,3,7$. The unique possible set of composition factors for this group are $28^2$ on $M(E_7)$ (see \cite[Table 6.228]{litterickmemoir}), and there are no self-extensions of this module for all odd primes, so $M(E_7){\downarrow_H}$ must be semisimple.

\begin{prop}\label{prop:omega8plus2} Let $p\neq 2$. Then any subgroup $H\cong \Omega_8^+(2)$ of $\mbG$ is a blueprint for $M(E_7)$.
\end{prop}
\begin{proof} Let $L$ denote the $\GL_4(2)$-parabolic subgroup inside $H$, which also acts irreducibly on the $28$-dimensional simple module for $H$ in all characteristics $p\neq 2$. Thus $H$ and $L$ stabilize the same subspaces of $M(E_7)$. It is clear that $L$ lies inside a positive-dimensional subgroup of $\mbG$, and the only appropriate maximal such subgroup, examining the tables of actions in, say, \cite[p.247]{thomas2016} is the $A_7$ maximal-rank subgroup. (This subgroup does not stabilize the same subspaces as $H$ because its $28$-dimensional factors are not self-dual, so the diagonal subspaces for the $H$-action are not stabilized by $A_7$.)

Since $\Omega_8^+(2)\leq \Omega_8^+(p)$ (i.e., the $8$-dimensional module for $\Omega_8^+(2)$ fixes a symmetric bilinear form of plus type) and $L$ is contained in $\Omega_8^+(2)$, we therefore see that $L$ is contained inside a $D_4$ subgroup of the $A_7$ subgroup, and this stabilizes the same subspaces of $L$ (as the $28$-dimensional factors for $D_4$ \emph{are} self-dual), i.e., $L$ and therefore $H$ are blueprints for $M(E_7)$, as needed.
\end{proof}

\subsubsection{\texorpdfstring{${}^2\!B_2(8)$}{2B2(8)}}

Here $p=5$, and the composition factors of $M(E_7)\downarrow_H$ are $(14,14^*)^2$ by \cite[Table 6.234]{litterickmemoir}. The only module action consistent with the action of a unipotent element is $(14/14^*)^{\oplus 2}$, using \cite[Table 7]{lawther1995}. (There is no extension $14^*/14$.) The composition factors on $L(E_7)$ are $14,14^*,35_1,35_2,35_3$. The $35$-dimensional modules are projective, hence become summands, and the module that has the correct action of $u$ has a summand $14/14^*$. Thus the action of $H$ on $L(E_7)$ is also determined up to isomorphism.

\begin{prop} If $H\cong {}^2\!B_2(8)$ then any subgroup $H\leq \mbG$ is strongly imprimitive.
\end{prop}
\begin{proof} Using the Lie product method on $\mathcal R^H(L(E_7))$, which is $37$-dimensional, we obtain a unique $H$-invariant Lie product on $L(E_7){\downarrow_H}$. (See the supplementary materials.) Notice that $2.H$, which is unique up to isomorphism, possesses an $8$-dimensional representation over $\F_5$, hence lies inside $A_7$. The exterior square of this $8$-dimensional module is $14/14^*$, and so the action is correct. Notice that there are exactly $q+1$ indecomposable $28$-dimensional submodules of the module $(14/14^*)^{\oplus 2}$ over $\F_q$.

By a computer check, this $28$-dimensional module supports a symmetric bilinear form, and hence $H$ is contained in a $D_4$ subgroup $\mb X$ of $A_7$. This also acts as the sum of two isomorphic $28$-dimensional modules, so $\mb X$ stabilizes exactly $q+1$ different $28$-dimensional subspaces of $M(E_7)$ that are definable over $\F_q$. Hence all indecomposable $28$-dimensional $H$-invariant subspaces of $M(E_7)$ are $\mb X$-invariant, and thus $H$ is strongly imprimitive by Proposition \ref{prop:intersectionsubspace}. (This is the same as the subspace stabilizer seen in the proof of Proposition \ref{prop:omega8plus2}.)
\end{proof}

\subsection{Lie type \texorpdfstring{$\PSL_2(r)$}{PSL(2,r)}}

We come to the main issue. The groups $\PSL_2(r)$ for $r\neq p$ form the largest collection of subgroups to consider, and also all of the most difficult cases, at least over the algebraic closure. (The issue for $\PSU_3(3)$ was for finite fields.)

Since we have considered $\Alt(5)$ and $\Alt(6)$ earlier, the cases under consideration here are $\PSL_2(r)$ for $r=7,8,11,13,19,25,27,29,37$. Each of these has at least one set of composition factors labelled with `$\mathbf{P}$' in the tables in \cite{litterickmemoir} (except for $r=25$, where we need the case $p=5$), so we will have to consider each in turn. Some of them, for example $r=11,25$, do not yield Lie primitive subgroups, or maximal subgroups. Most of the others do, but there are some we cannot at the moment determine either way. These are $r=7,8$.

In this section, we let $H$ be a subgroup of $\mbG$, and write $\bar H$ for the image of $H$ in $G$ and $\bar G$. We keep the convention that if $Z(H)=1$ then we identify $H$ and $\bar H$. Proofs of maximality of subgroups are delayed until Section \ref{sec:proofofmaximality}.

\subsubsection{\texorpdfstring{$\PSL_2(37)$}{PSL(2,37)}}

This is the first example where the version of $H$ that embeds in $\mbG$ depends on $p$. If $p\neq 2$ write $H=\SL_2(37)$, and if $p=2$ write $H\cong\PSL_2(37)$. Note that $H$ embeds in $\mbG$ and $H/Z(H)$ does not for $p\neq 2,37$. The appropriate primes here are $p=2$, $p=3$, $p=19$ and $p\neq 2,3,19,37$. If $p\neq 2,3,19,37$ then in \cite[Table 6.207]{litterickmemoir} we see that the action of $H$ on $M(E_7)$ is the sum of an $18$- and a $38$-dimensional module (see also \cite{kleidmanryba1993}). There are three potential $38$-dimensional modules, on which an element of $H$ of order $3$ acts with trace $2$. One has rational character, the other two are algebraic conjugates.

Next, consider $p=2$. In \cite[Table 6.210]{litterickmemoir} there are three potential Lie primitive sets of factors. As $P(18_1)$ has structure
\[18_1/1/18_2/1/18_1\]
we see that the second set of factors for $M(E_7){\downarrow_H}$---$18_1^3,1^2$---must yield a stabilized line. The first set of factors is $18_1,38$, so this must be semisimple. An element $u\in H$ of order $2$ acts on $H$ with Jordan blocks $2^{27},1^2$, which does not appear in \cite[Table 7]{lawther1995}, so this case cannot occur. The third set of factors for $M(E_7){\downarrow_H}$ is $18_1^2,18_2,1^2$, and so $H$ must act as $P(18_1)$. Since $H$ embeds in the adjoint group $E_7(\mathbb{C})$, $H$ must embed in $\mbG$, so this is the action on $M(E_7)$.

The reduction modulo $2$ of the two algebraically conjugate characters for characteristic $0$ is the case where $H$ stabilizes a line (these $38$s remain irreducible modulo $2$). Since the (connected component of the) line stabilizers on $M(E_7)$ are contained in an $E_6$-parabolic and a $D_6$-parabolic (see Proposition \ref{prop:1spacestabs}), into neither of which a cover of $H$ embeds, the rational character, which reduces modulo $2$ to the case not stabilizing a line, must be the case yielding an embedding of $H$ into $\mbG$. (This agrees with the analysis of \cite[Section 4.4]{pacherathesis}, which constructs the characteristic $0$ case.)

For $p=19$ there are again three possible $38$-dimensionals, but we cannot reduce modulo $2$ in this case. We show in the supplementary materials using the Lie product method that there is a unique subgroup $H$ in this case, and so the action on $M(E_7)$ must be the reduction modulo $19$ of that in characteristic $0$. For $p=3$, there is an ambiguity in \cite[Table 6.209]{litterickmemoir}; this can be resolved once one includes the action of an element $u$ of order $9$ in $H$. This must act on the unique module $M(E_7){\downarrow_H}=18_1\oplus 38$ with Jordan blocks $9^6,1^2$. From \cite{lawther1995}, we see that $u$ lies in class $E_6(a_1)$, and acts on $L(E_7)$ with blocks $9^{14},7$. In particular, this means that $L(E_7){\downarrow_H}$ has the structure
\[ 19_i/19_{3-i}/19_i/19_{3-i}/19_i/19_{3-i}/19_i,\]
and $i$ can be determined by the trace of an element of order $37$ in $H$.

With this description of the actions of $H$ on $M(E_7)$ and $L(E_7)$, we can now state our result. The claim about maximality will be proved in Section \ref{sec:proofofmaximality}.

\begin{prop} Let $p\neq 37$. If $p=2$ then write $H$ for the group $\PSL_2(37)$, and if $p$ is odd write $H$ for the group $\SL_2(37)$.
\begin{enumerate}
\item There is a unique $\mbG$-conjugacy class of subgroups $H$. When $\bar H\leq G$, if $p\neq 2$ then there are exactly two $G$-conjugacy classes of subgroups $\bar H$ and if $p=2$ there is a single $G$-conjugacy class. Furthermore, $N_\mbG(H)=H$.
\item $\bar H$ embeds in $G$ if and only if $x^2\pm x-9$ splits over $\F_q$ (which is true if and only if $q$ is a square modulo $37$).
\item If $N_{\bar G}(\bar H)$ is maximal in $\bar G$ then $q$ is minimal such that $x^2-x+9$ splits over $\F_q$, and $\bar G=G$ or $q=p^2$ and $\bar G$ induces the field automorphism on $G$.
\end{enumerate}
\end{prop}
\begin{proof} The proof proceeds in stages.

\paragraph{Determination of $N_\mbG(H)$} This is easy: since the composition factors of $M(E_7){\downarrow_H}$ are not stable under the diagonal automorphism of $H$, $N_\mbG(H)=H$.

\paragraph{Determination of $H$ up to conjugacy}

Existence of this subgroup was proved in \cite{kleidmanryba1993}, and then later in a more general setting in \cite{serre1996}, but in neither case is uniqueness addressed. For $p\neq 2,3,19,37$ this is accomplished in \cite[Theorem 4.4.2]{pacherathesis}, and we obtain a unique class. (We independently verify this in the supplementary materials, using the subalgebra method.) For $p=2,19$ one may use the space $\mathcal{R}^H(L(E_7)^\circ)$. The space has dimension $58$ for $p=19$ and $60$ for $p=2$. In each case one finds a unique $H$-invariant $\mathfrak e_7$-Lie algebra structure on $L(E_7)^\circ{\downarrow_H}$. For $p=3$, although $\mathcal{R}^H(L(E_7){\downarrow_H})$ has dimension $58$, the equations are very complicated, and it seems difficult to proceed in this way. One instead uses the subalgebra method from Section \ref{sec:subalgmethod}. Since the action of $H$ on $M(E_7)$ is $18_1\oplus 38$ and each of these is a symplectic space, $H$ is unique up to $\Sp_{56}(k)$-conjugacy by Proposition \ref{prop:ssmultfree}. The subalgebra method also yields a unique $\mathfrak e_7$-subalgebra (this time up to conjugacy in the centralizer of $H$ in $\Sp_{56}(k)$).

Thus there is a unique $\mbG$-conjugacy class in all cases. The number of $G$-classes follows from Lemma \ref{lem:diagswaps}.

\paragraph{Action of outer automorphisms}

Let $\sigma$ denote a field automorphism $F_q$. Since there is a unique $\mbG$-class, we may choose $H$ in the class such that $\sigma$ normalizes it. Either $\sigma$ acts as an outer automorphism on $H$, so induces the diagonal automorphism, or as an inner automorphism, so up to conjugacy $H\leq \mbG^\sigma$ by Lemma \ref{lem:inneriscent}.

To decide if $\sigma$ acts as an outer or inner automorphism on the subgroup we may examine the character. Since the action of $H$ on $M(E_7)$ is the sum of the $38$ and one of the two $18$s that are permuted by the diagonal automorphism of $H$, $\sigma$ centralizes $H$ (up to conjugacy) if and only if $M(E_7){\downarrow_H}$ is definable over $\F_q$, and acts as an outer automorphism on $H$ if and only if it is definable only over $\F_{q^2}$. The irrationalities in the character satisfy $x^2\pm x-9$, as claimed in the proposition.

By Lemma \ref{lem:diagswaps}, the diagonal automorphism always swaps two $G$-classes of subgroups $\bar H$.\end{proof}
%
%We summarize the conditions required in Table \ref{tab:psl237}.
%
%\begin{table}
%\begin{center}\begin{tabular}{ccccc}
%\hline $p$ & Number of classes & Minimal field & Action of $F_p$ & Action of $F_p^3$
%\\ \hline $2$ & $3$ & $\F_{2^6}$ & Permutes & Normalizes
%\\ $3$ & $1$ & $\F_3$ & Centralizes & Centralizes
%\\ $19$ & $3$ & $\F_{19^2}$ & Normalizes & Normalizes
%\\ 
%
%
%
%\\ \hline\end{tabular}
%\end{center}
%\caption{Conditions for $\PSL_2(37)$.}
%\label{tab:psl237}
%\end{table}

\subsubsection{\texorpdfstring{$\PSL_2(29)$}{PSL(2,29)}}
\label{sec:sl229}
As with the previous case, the group $H$ depends on $p$: if $p=2$, let $H\cong\PSL_2(29)$, and if $p\neq 2$ then let $H\cong\SL_2(29)$. The appropriate primes here are $p=2$, $p=3$, $p=5$, $p=7$ and $p\neq 2,3,5,7,29$. If $p\neq 2,3,5,7,29$ then the action of $H$ on $M(E_7)$ is the sum of two non-isomorphic $28$-dimensional modules (there are four such modules with non-rational character, on which an element of order $3$ has trace $1$). The same holds for $p=2,7$. However, for $p=3,5$, the two $28$-dimensional composition factors are isomorphic, and there is a (unique) non-trivial self-extension. In these two cases, the action of $u$ of order $p$ on $M(E_7)$ does not appear in \cite[Table 7]{lawther1995} if $M(E_7){\downarrow_H}$ is semisimple, so $M(E_7){\downarrow_H}$ must be the self-extension. In both cases, the $1$-cohomology of the symmetric square of the $28$ is $1$-dimensional, and so there is a unique $\Sp_{56}(k)$-class of subgroups $H$ with this module structure by Corollary \ref{cor:symp1cohom}.

The composition factors of $L(E_7){\downarrow_H}$ are uniquely determined for all primes, given the factors for $M(E_7)$, by \cite[Tables 6.202--6.206]{litterickmemoir}. The structure of $L(E_7){\downarrow_H}$ is easy to determine using these composition factors and the action of a unipotent element of order $p$ in $H$, via \cite[Tables 7 and 8]{lawther1995}.

We now state our result, delaying proofs of maximality to Section \ref{sec:proofofmaximality}, as usual.

\begin{prop}\label{prop:psl229} Let $p\neq 29$. If $p=2$ then write $H$ for the group $\PSL_2(29)$, and if $p$ is odd write $H$ for the group $\SL_2(29)$.
\begin{enumerate}
\item If $p\neq 5$ then there are exactly two $\mbG$-conjugacy classes of subgroups $H$, and if $p=5$ then there is a unique $\mbG$-class. If $\bar H\leq G$, and $p=2,5$ or $p\neq 2,5$, then there are exactly two or four $G$-conjugacy classes of subgroups $\bar H$ respectively. Furthermore, $H=N_\mbG(H)$.
\item $\bar H$ embeds in $G$ if and only if both $x^2-x-1$ and $x^2\pm x-7$ split over $\F_q$ (which is true if and only if $q\equiv 0,\pm1\bmod 5$ is a square modulo $29$).
\item If $N_{\bar G}(\bar H)$ is maximal in $\bar G$ then $p\neq 5$, and $q$ is minimal such that $q\equiv \pm1\bmod 5$ is a square modulo $29$. Either $\bar G=G$, or $q=p^2$ for $p\equiv \pm1\bmod 5$ a non-square modulo $29$ and $\bar G$ induces the field automorphism on $G$.
\end{enumerate}
\end{prop}
\begin{proof}The proof proceeds in stages.

\paragraph{Determination of $N_\mbG(H)$} This is easy: since the composition factors of $L(E_7){\downarrow_H}$ are not stable under the diagonal automorphism of $H$, $N_\mbG(H)=H$.

\paragraph{Determination of $H$ up to conjugacy} It appears that the original construction of this is due to Serre but is unpublished. The only construction the author is aware of is in \cite{pacherathesis}, but we provide a construction in the supplementary materials, as with all maximal subgroups. From \cite[Theorem 4.5.2]{pacherathesis}, we see that for $p\neq 2,3,5,7,29$, each $\Aut(H)$-class of representations yields a unique conjugacy class of subgroups $H$ of $\mbG$. (We independently verify this in the supplementary materials, using the subalgebra method.) Using the subalgebra method for $p=2,3,5,7$, we find the same result.

If $p=5$ then $H$ is contained in a copy of $2\cdot Ru$ by Proposition \ref{prop:maxcont} below, and this must be the same embedding as the actions on $M(E_7)$ and $L(E_7)$ are uniquely determined.

\paragraph{Action of outer automorphisms}

Let $\sigma$ denote a field automorphism $F_q$. If $p=5$ then in the supplementary materials we prove that $\bar H$ embeds in $E_7(5)$, as does $Ru$ (see Section \ref{sec:ru}). Thus $\bar H$ is contained in a $\sigma$-stable copy of $Ru$.

In all other cases, there are two $\mbG$-classes, so we first decide if $\sigma$ swaps the two classes or stabilizes them. An element of order $5$ acts on the two representations with complex numbers satisfying the equation $x^2-x-1$. Thus these lie in $\F_q$ if and only if $q\equiv \pm 1\bmod 5$. Hence $\sigma $ interchanges the two representations if $q\equiv \pm 2\bmod 5$ and stabilizes the two if $q\equiv \pm 1 \bmod 5$.

Thus we may assume that $\sigma$ stabilizes the class. Since the action of the diagonal automorphism of $H$ interchanges the two simple modules of dimension $15$ ($14$ for $p=2$) and they appear with differing multiplicities in $L(E_7)$, we can determine if $\sigma$ acts as an outer automorphism on $H$ by whether it interchanges these two modules. The irrationalities in the $15$-dimensional (and $14$-dimensional) modules are for elements of order $29$, and they satisfy $x^2-x-7$. This equation splits over $\F_q$ if and only if $q$ is a square modulo $29$.

By Lemma \ref{lem:diagswaps}, the outer diagonal automorphism of $G$ always swaps two $G$-classes of subgroups $\bar H$.
\end{proof}

\subsubsection{\texorpdfstring{$\PSL_2(27)$}{PSL(2,27)}}

Again, the group $H$ depends on $p$: if $p=2$, let $H\cong\PSL_2(27)$, and if $p\neq 2$ then let $H\cong\SL_2(27)$. The appropriate primes here are $p=2$, $p=7$, $p=13$ and $p\neq 2,3,7,13$. Examining \cite[Tables 6.195--6.201]{litterickmemoir}, we find two options for the composition factors of $M(E_7){\downarrow_H}$ if $p$ is odd: either two non-isomorphic $28$-dimensional modules for $p\neq 13$ (and necessarily we have the direct sum), or $(14,14^*)^2$ for $p=7,13$. If $p=2$ then \cite[Table 6.201]{litterickmemoir} shows that there are three possible sets of composition factors, two of which yield strongly imprimitive subgroups. The remaining set is, again, the sum of two non-isomorphic $28$-dimensional modules. For $p\neq 13$ there are six possible options for this sum, all algebraically conjugate. Notice that the field automorphism of $\SL_2(27)$ of order $3$ acts semiregularly on the six $28$-dimensional modules, and thus there are up to automorphism of $H$ exactly two possible representations (the diagonal automorphism stabilizes all six modules).

If $p=7$ and the factors on $M(E_7)$ are $(14,14^*)^2$ then the composition factors on $L(E_7)$ are $26^5,1^3$, and the projective cover of $26$ has composition factors $26^4,1$. Thus $H$ always stabilizes a line in this scenario.

Thus the only case where we do not yet understand the module $M(E_7)$ is $p=13$, and the composition factors of $M(E_7){\downarrow_H}$ are $(14,14^*)^2$. The action of $u\in H$ of order $13$ on $14^\pm$ has Jordan blocks $13,1$, and the action on $14^\pm/14^\mp$ has blocks $13^2,2$. Finally, the action on $14^\pm/14^\mp/14^\pm/14^\mp$ has blocks $13^4,4$. The actions $13^4,1^4$ and $13^4,4$ are valid actions for $E_7$ (see \cite[Table 7]{lawther1995}), corresponding to $13^{10},1^3$ and $13^{10},3$ on $L(E_7)$ respectively, but $13^4,2^2$ does not exist.

For $L(E_7){\downarrow_H}$, the composition factors are $27^2,1$ and a sum of three of the six non-isomorphic $26$-dimensional modules. The $27^2,1$ must assemble into $27/1/27$ (or else $H$ stabilizes a line on $L(E_7)$ and is strongly imprimitive by Lemma \ref{lem:fixedline}), and the $26$s are projective: $u$ acts on $27/1/27$ with blocks $13^4,3$, so $M(E_7){\downarrow_H}$ must be uniserial.

\medskip

It remains to describe the action of $H$ on $L(E_7)$ for $p\neq 13$. If $p\neq 2,3,7,13$ then it is a sum of five distinct modules, of dimensions $28,27,26,26,26$. If $p=7$ the composition factors are $28,26^4,1$, and from the factors of $P(26)$ given above, we see that $L(E_7){\downarrow_H}=28\oplus P(26)$. Finally, for $p=2$ the composition factors are $28,26_1,26_2,26_3,13,13^*,1$. The $28$- and $26$-dimensional modules must split off, and by examining the action of $u\in H$ of order $4$, we obtain that the $13,13^*,1$ must assemble into module of the form $13/1,13^*$ (for some choice of labels for $13$ and $13^*$).

In all cases then, the actions of $H$ on $M(E_7)$ and $L(E_7)$ are determined uniquely up to isomorphism (given the caveat about algebraic conjugacy). Maximality proofs are delayed until Section \ref{sec:proofofmaximality}.

\begin{prop}\label{prop:psl227} Let $p\neq 3$. If $p=2$ then write $H$ for the group $\PSL_2(27)$, and if $p$ is odd write $H$ for the group $\SL_2(27)$. Suppose that $H$ is not strongly imprimitive.
\begin{enumerate}
\item If $p\neq 13$ then there are exactly two $\mbG$-conjugacy classes of subgroups $H$, and if $p=13$ then there is a unique $\mbG$-class. If $\bar H\leq G$, and $p=2,13$ or $p\neq 2,13$, then there are exactly two or four $G$-conjugacy classes of subgroups $\bar H$ respectively. Furthermore, $N_\mbG(H)=H$ for $p\neq 13$ and $N_\mbG(H)=H.3$ for $p=13$.
\item $\bar H$ embeds in $G$ if and only if $q\equiv 0,\pm1\bmod 13$.
\item If $N_{\bar G}(\bar H)$ is maximal in $\bar G$ then $q$ is minimal such that $q\equiv 0,\pm1\bmod 13$. Either $\bar G=G$, or $q=p^{3a}$ for $a=1,2$ and $\bar G$ induces the field automorphism of order $3$ on $G$, with $|N_{\bar G}(\bar H):\bar H|=3$.
\end{enumerate}
\end{prop}
\begin{proof} As usual by now we proceed in stages.

\paragraph{Determination of $N_\mbG(H)$} In all cases, the composition factors of $H$ on $M(E_7)$ and $L(E_7)$ are invariant under the diagonal automorphism, so in theory $N_\mbG(H)$ could include that. However, for $p\neq 13$ the composition factors of $L(E_7){\downarrow_H}$ are not invariant under the field automorphism (it's easiest to see for $\PSL_2(27)$) and so $|N_\mbG(H)/H|\leq 2$ if $p\neq 13$.

Assume first that $p\neq 13$. Note that $H$ extends to $H.2$ in $\Sp_{56}(k)$. For $p\neq 2,3,7,13$ in \cite{griessryba2002}, and for $p=2,7$ and $p\neq 2,3,7,13$ in the supplementary materials, we show that this normalizing element swaps two classes of $\mathfrak e_7$-subalgebras of the Lie algebra of $\Sp_{56}(k)$, and hence cannot normalize it.

(For $p\neq 2,3,13$, we find that there is no extension of $L(E_7){\downarrow_H}$ to a module for $\PGL_2(27)$ where the traces of the classes of involutions are those of involutions in (adjoint) $E_7$, offering another proof for these primes.)

For $p=13$, however, this is all false. The Brauer characters of $H$ on $M(E_7)$ and $L(E_7)$ are extensible to both diagonal and field automorphisms of $H$. However, the module $M(E_7){\downarrow_H}$, which has structure $14/14^*/14/14^*$ is not extensible to $H.2$, since the diagonal automorphism swaps $14$ and $14^*$, hence has $28$-dimensional simple modules. However, in the supplementary materials we do find an element of order $3$ in $\Sp_{56}(k)$ that induces the field automorphism on $H$ and lies in $E_7(k)$.

\paragraph{Determination of $H$ up to conjugacy}

From \cite{griessryba2002}, we see that for $p\neq 2,3,7,13$, each $\Aut(H)$-class of representation yields a unique conjugacy class of subgroups of $\mbG$. Using the subalgebra method for $p=2,7,13$, we find the same result. In the supplementary materials we do these cases and redo the $p\neq 2,3,7,13$ case for completeness. Thus for $p\neq 13$ there are exactly two $\mbG$-classes of subgroups $H$, and for $p=13$ there is exactly one.

\paragraph{Action of outer automorphisms}

Let $\sigma$ denote a field automorphism $F_q$. Our first task is to determine the conditions on $q$ for both $M(E_7){\downarrow_H}$ and $L(E_7){\downarrow_H}$ to be defined over $\F_q$. The Brauer characters of $H$ on $M(E_7)$ and $L(E_7)$ only have irrationalities for elements of order $13$. Specifically, for some element of order $13$ and choice of $13$th root of unity $\zeta$, the Brauer character value for $M(E_7)$ and $L(E_7)$ are $\zeta+\zeta^{-1}+\zeta^4+\zeta^{-4}$ and $\zeta^5+\zeta^{-5}+1$ respectively. These are clearly preserved under the map $\zeta\mapsto\zeta^{-1}$, and the second is only preserved by this map from any field automorphisms of $k$. Thus we require $q\equiv \pm 1\bmod 13$.

Suppose first that $q$ is the smallest power of $p$ such that $q\equiv \pm 1\bmod 13$, and write $q=p^a$. The action of $F_p$ depends on the value of $a$: if $a=1$ then there is nothing to say. If $a=2$ then $F_p$ must swap the two $\mbG$-classes of subgroups, and if $a=3$ then it must permute the three sets of composition factors associated to each class of subgroups, so induce the field automorphism on $H$. Finally, if $a=6$ then it does both of these, swapping and permuting.

We therefore assume that $q\equiv \pm 1\bmod 13$. Since $F_q$ stabilizes each set of composition factors, it stabilizes each of the two classes of subgroups, so we may assume that $F_q$ normalizes $H$. We now show that some Borel subgroup $L\cong 3^3\rtimes 13$ of some conjugate of $H$ is always centralized by $F_q$. To see this, note that $L$ acts on $M(E_7)$ as the sum of four $13$-dimensional and four non-trivial $1$-dimensional modules. It is easy to see that $L$ therefore lies diagonally in the $E_6T_1$-Levi subgroup, and then inside $F_4T_1\leq F_4A_1$. Notice that our condition on $q$ guarantees that $A_1(q)$ contains representatives of all classes of elements of order $13$ in $A_1$. Furthermore, by \cite{clss1992}  we have that the $3^3\rtimes \SL_3(3)$ subgroup (which contains a copy of $L$) lies in $F_4(p)$ for all $p$, so in particular in $F_4(q)$. Thus $L$ lies in $E_7(q)$ and so is centralized by $F_q$. Furthermore, we see that $L$ is not contained in a $3^3\rtimes 26$ in adjoint $E_7$, since it would still lie in $F_4A_1$. If $H$ is normalized by $F_q$ and contains $L$ then $F_q$ cannot act as a diagonal automorphism, since that results in a $3^3\rtimes 26$ on $L$. In addition, since the field automorphism of $H$ does not stabilize the composition factors, $F_q$ cannot induce the field automorphism or the product of a field and a diagonal. Thus if $F_q$ normalizes $H$ then it acts as an inner automorphism of $H$. In particular, $\bar H\leq E_7(q)$ (up to conjugacy).

Since diagonal automorphisms of $G_\mathrm{sc}$ have order $2$ and $|N_\mbG(H):H|$ is either $1$ or $3$, we see that $N_\mbG(H)=N_{G_\mathrm{sc}}(H)$. Thus Lemma \ref{lem:diagswaps} implies that the diagonal automorphism of $G$ fuses two $G$-classes of subgroups $\bar H$.

This completes the proof.
\end{proof}

\subsubsection{\texorpdfstring{$\PSL_2(25)$}{PSL(2,25)}}
\label{sec:psl225}

If $p\neq 5$ then all examples for either $\PSL_2(25)$ or $\SL_2(25)$ are strongly imprimitive by \cite[Tables 6.191--6.194]{litterickmemoir}. Of course, $p=5$ is a defining-characteristic embedding, and is one of the cases left from \cite{craven2015un}, and $H\cong \SL_2(25)$. In this case, the actions of $H$ on both $M(E_7)$ and $L(E_7)$ were completely determined in \cite{craven2015un}. We will prove that any such $H$ is strongly imprimitive, using the subalgebra method.

\begin{prop} If $p=5$ and $H\cong \SL_2(25)$ is a subgroup of $\mbG$ then $H$ is strongly imprimitive.
\end{prop}
\begin{proof} We use the notation from \cite{craven2015un} for the modules for $\SL_2(25)$, with $2_1$ being the natural module and $3_1$, $4_1$ and $5_1$ being its symmetric powers. Write $2_2$ and so on for the image of $2_1$ under the Frobenius morphism. Write $6_{1,2}=2_1\otimes 3_2$, $10_{2,1}=2_2\otimes 5_1$, and so on.

The action of $H$ on $M(E_7)$ is
\[ 10_{2,1}\oplus (12_{1,2}/2_1,4_1,6_{1,2},10_{1,2}/12_{1,2}),\]
and the action on $L(E_7)$ is
\[(15_{2,1}/8_{2,1}/1,3_2/8_{2,1}/15_{2,1})\oplus (15_{1,2}/8_{1,2}/1,3_1/8_{1,2}/15_{1,2})\oplus (3_2/8_{2,1},16/3_2)\oplus 3)_1.\]

Up to conjugacy in $\Sp_{56}(k)$, there is a unique conjugacy class of subgroups $H$ with this action on $M(C_{28})$. To see this, first note that because the two summands have no composition factors in common, any alternating form is a sum of alternating forms for the two summands, so $H$ lies in a copy of $\Sp_{10}(k)\times \Sp_{46}(k)$. If the form is unique (up to conjugacy) on both factors then it is unique up to conjugacy by Proposition \ref{prop:ssnofactorsincommon}. Of course, this holds for the first factor because $H$ acts irreducibly on it. Let $M$ denote the $46$-dimensional summand.

The $12$-space in the socle is always totally isotropic; a short computer program shows that any symplectic form on the $46$-space is the sum of the standard form (which vanishes on $12_{1,2}$) and a form of rank $12$ (and hence vanishes on the Jacobson radical of the module). Thus the $12_{1,2}$ is always totally isotropic and $H$ lies in a $C_{11}A_{11}T_1$-parabolic subgroup of $\Sp_{46}(k)$. Notice that the symplectic forms on $M$ are isomorphic to a projective line over $k$.

We clearly see a $1$-dimensional unipotent subgroup that centralizes the corresponding subgroup of $\GL_{46}(k)$, and since the map has rank $12$ and squares to $0$, any (non-trivial) element $v$ in the subgroup has Jordan blocks $2^{12},1^{22}$. If $v$ lies in a copy of $\Sp_{46}(k)$ then its centralizer is a positive-dimensional subgroup whose reductive part has type $C_{11}D_6$ by \cite[Theorem 3.1(iv)]{liebeckseitzbook}. But the module $12_{1,2}$ supports a symplectic form, not an orthogonal form, and so $H$ cannot also lie in this copy of $\Sp_{46}(k)$. We see therefore that $v$ must act semiregularly on the symplectic forms on $M$ that have rank $46$ (and must stabilize the unique one of rank $12$). This means that the unipotent subgroup regularly permutes the symplectic forms of rank $46$. This means that there is a unique conjugacy class of $\SL_2(25)$ subgroups acting on the minimal module $M(\Sp_{46}(k))$ as $M$.

This proves that $H$ is unique up to conjugacy in $\Sp_{56}(k)$, so we may choose any representative for the form. We can now apply the subalgebra method, and in the supplementary materials we see that this subgroup is unique up to $\mbG$-conjugacy.

In \cite{craven2015un} we stated that $H$ would be strongly imprimitive if this were the case. Let us see this. Let $\mb X$ be the $A_1A_1$ maximal subgroup of $\mbG$. In \cite[Table 10.2]{liebeckseitz2004} we see that for $p=5$ the subgroup $\mb X$ acts with summands of dimensions $10$ and $46$. Subgroup $19$ from \cite[Table 7]{thomas2016} (with $r=0,s=1$) contains a copy of $\SL_2(25)$ with the same actions on $M(E_7)$ and $L(E_7)$, so this is $\mbG$-conjugate to $H$. Thus we may assume that $H\leq \mb X$.

All $10$-dimensional $H$-submodules of $M(E_7)$ are stabilized by $\mb X$, so $H$ is strongly imprimitive by Proposition \ref{prop:intersectionsubspace}.
\end{proof}

\subsubsection{\texorpdfstring{$\PSL_2(19)$}{PSL(2,19)}}
\label{sec:psl219}
This section is separated from the next one because there is a class of Lie primitive subgroups $\PSL_2(19)$ and a class of Lie primitive subgroups $\PGL_2(19)$. The second of these is due to Serre \cite{serre1996} (although uniqueness for all fields of odd characteristic is proved here) but this first class appears to be new.

We start with considering subgroups $L$ of $\mbG$ of the form $19\rtimes 9$, where $p\neq 19$. Such groups have simple $kL$-modules of dimensions $1$ and $9$, with the two $9$-dimensional modules being dual. Suppose that $L$ acts on $M(E_7)$ with composition factors $(9,9^*)^3,1_a,1_b$, where $1_a$ and $1_b$ are two $1$-dimensional modules. If $p\neq 19$ then $L$ stabilizes a line on $M(E_7)$: if $p\neq 3,19$ then the action is semisimple, and if $p=3$ then the $9$-dimensional factors are projective and so break off as summands. Thus $L$ lies in a line stabilizer, and cannot lie in a $D_6$-parabolic or a subgroup of type $F_4$ by comparing dimensions of composition factors on $M(E_7)$. Hence $L$ lies in an $E_6T_1$-parabolic subgroup by Proposition \ref{prop:1spacestabs}, as $L$ has no subgroup of index $2$, so if it lies in a group $E_6.2$ it lies in the Levi subgroup $E_6$, hence in an $E_6T_1$-parabolic subgroup.

If $p\neq 3$ then $L$ must lie in the $E_6T_1$-Levi subgroup as $p$ and $|L|$ are coprime, and if $p=3$ then $L$ acts on $M(E_6)$ with factors of dimension $9$ only, hence has no $1$-cohomology on the module. Thus there is a single conjugacy class of complements in the $E_6T_1$-parabolic, and again $L$ lies in $E_6T_1$. There is a unique such conjugacy class of subgroups $L$ in $E_6$ by \cite[Proof of Proposition 5.20]{craven2020un}, so if the modules $1_a$ and $1_b$ are trivial modules then $L$ is uniquely determined up to conjugacy. (This is the case for $p=3$.)

If $1_a$ and $1_b$ have kernel of index $3$ in $L$ then $L$ is again unique up to conjugacy: there are two such diagonal subgroups, and they must be interchanged in the subgroup $E_6T_1.2$ (the normalizer of the $E_6T_1$-Levi subgroup in $E_7$). We also see that the graph automorphism normalizes the subgroup $L$ inside the $E_6$ factor, and forms a subgroup $19\rtimes 18$ in the adjoint version of $E_7$. Since such a subgroup must lie in $E_6T_1.2$, the other conjugacy class of subgroups $L$ cannot be contained in a subgroup $19\rtimes 18$ in the adjoint version of $E_7$.

In particular, we see that if $\PSL_2(19)$ lies in the adjoint version of $E_7$ and contains one of these subgroups $L$, then it cannot extend to $\PGL_2(19)$ unless the subgroup $L$ is the one lying inside $E_6$ rather than diagonally in $E_6T_1$.

\bigskip

With these preliminaries considered, we return to groups $\PSL_2(19)$. From \cite[Tables 6.184--6.190]{litterickmemoir} we see that for $p$ odd any copy of $\PSL_2(19)$ is strongly imprimitive, so if $p$ is odd, let $H\cong \SL_2(19)$. If $p=2$ we let $H\cong \PSL_2(19)$, as usual.

One of the possibilities in characteristic $0$ is that $H$ acts on $M(E_7)$ as $18\oplus 18'\oplus 20$, where $18$ and $18'$ are two non-isomorphic (faithful) $18$-dimensional modules for $\SL_2(19)$ and the $20$ is the faithful module with rational character. The action on $L(E_7)$ is
\[ 18_1\oplus 18_2^{\oplus 2}\oplus 19\oplus 20_1\oplus 20_2\oplus 20_3,\]
where $18_1$ and $18_2$ are two of the four non-isomorphic $18$-dimensional modules and the $20_i$ are all three $20$-dimensional modules for $\PSL_2(19)$ that have non-rational character.

For the reduction modulo $5$ of this module, the two modules of dimension $18$ in $M(E_7)$ become isomorphic, and using the action of an element of order $5$ in $H$ and \cite[Table 7]{lawther1995}, as in previous cases, we must obtain a module of the form $20\oplus (18/18)$. For $p=3$ the reduction of this and the $\PGL_2(19)$ case in the next section are the same, and we delay treatment of this until that section. For $p=2$ the two $18$-dimensional modules in $M(E_7)$ again become isomorphic, and we again are forced to have a module $20\oplus (18/18)$.

We have the following proposition, delaying proof of maximality until Section \ref{sec:proofofmaximality}.

\begin{prop} Let $p\neq 3,19$. If $p=2$ then write $H$ for the group $\PSL_2(19)$, and if $p$ is odd write $H$ for the group $\SL_2(19)$. Suppose that $H$ is a subgroup of $\mbG$ and that $M(E_7){\downarrow_H}$ contains a composition factor of dimension $20$. If $p=2$, suppose that $H$ does not have a composition of dimension $9$ on $M(E_7)$.
\begin{enumerate}
\item There are exactly two $\mbG$-conjugacy classes of subgroups $H$ unless $p=5$, in which case there is exactly one. If $\bar H\leq G$ and $p\neq 2,5$ there are exactly four $G$-conjugacy classes of subgroups $\bar H$, and if $p=2,5$ and $\bar H\leq G$ there are two $G$-conjugacy classes. Furthermore, $H=N_\mbG(H)$.
\item $\bar H$ embeds in $G$ if and only if $q\equiv \pm1\bmod 5$, and the orders of $q$ modulo $3$ and $19$ are both odd or both even, or $q$ is a power of $5$.
\item If $N_{\bar G}(\bar H)$ is maximal in $\bar G$ then $q$ is minimal such that $\bar H$ embeds in $G$, and either $\bar G=G$, or $q=p^2$ for $p\equiv \pm 1\bmod 5$ and $\bar G$ induces the field automorphism on $G$.
\end{enumerate}
\end{prop}
\begin{proof} The proof proceeds in stages.

\paragraph{Determination of $N_\mbG(H)$} Since the odd part $L\cong 19\rtimes 9$ of the Borel subgroup $19\rtimes 18$ of $H$ acts on $M(E_7)$ with two non-trivial $1$-dimensional composition factors (and six $9$-dimensional ones) we see from the discussion at the start of this section that $L$ is not contained in a group $19\rtimes 36$ (as $\mbG$ is simply connected). Thus $N_\mbG(H)=H$ for all primes.
%
%The action of $H$ on $L(E_7)^\circ$ for $p=2$ contains $9/9^*$ as a summand (and this is the only summand containing $9$ and $9^*$). Since this summand cannot extend to a module for $\PGL_2(19)$, $H.2$ cannot embed in $\mbG$ for $p=2$. 
%
%Since the case for $p=0$ reduces modulo $2$ to that of $p=2$, this means that $H.2$ does not embed for $p\neq 2,3,5,19$ either, but leaves out $p=5$. For this case (and this method applies to the others as well) it is easiest to proceed as for $\PSL_2(27)$. In the supplementary materials it is found that the element normalizing (but not centralizing) $H$ in $\Sp_{56}(k)$ swaps two $\mathfrak e_7$-subalgebras, and therefore cannot normalize $H$ in $\mbG$.
%
%We offer another proof of this later when considering outer automorphisms, which is only valid for $p$ odd.
%IS THIS TRUE FOR P=2 AS WELL????

\paragraph{Determination of $H$ up to conjugacy}

In the supplementary materials, for $p=2$, $p=5$ and $p\neq 2,3,5,19$ we provide a complete proof of the uniqueness of this case.

For $p\neq 2,3,5,19$ and $p=5$ we use the subalgebra method from Section \ref{sec:subalgmethod}. While it is clear that the form is unique for $p\neq 2,3,5,19$, for $p=5$ it is not clear.

As in previous cases, we must prove that the alternating form on $M(E_7){\downarrow_H}$ is unique up to conjugacy in $\Sp_{56}(k)$. Since the two summands have no composition factors in common, the form is the sum of a form for each factor, and can be scaled independently by Proposition \ref{prop:ssnofactorsincommon}. The first summand has a unique form; the second yields a subgroup of $\Sp_{36}(k)$, lying in the $A_{17}T_1$-parabolic subgroup. The $1$-cohomology of the module $S^2(18)$ is $1$-dimensional, so $H$ is unique up to conjugacy by Corollary \ref{cor:symp1cohom}.

For $p=2$ things are worse because the form is \emph{not} unique up to $\Sp_{56}(k)$-conjugacy, and in fact there are infinitely many such classes in $\mbG$. We thus must use the modified subalgebra method where we search in $\SL_{56}(k)$ instead of $\Sp_{56}(k)$. We still obtain a unique class of subgroups, however. (We also provide in the supplementary materials a proof using $\Sp_{56}(k)$, but with a $1$-parameter family of symplectic forms instead, but this is \emph{much} slower.)

\paragraph{Action of outer automorphisms}

Since $N_\mbG(H)=H$, diagonal automorphisms of $G$ fuse two $G$-classes of subgroups $\bar H$ by Lemma \ref{lem:diagswaps}. Thus we must decide on the action of the field automorphism $\sigma=F_p$.

For $p=2$ we require $\F_4$, and for $p=5$ we require $\F_5$ for $M(E_7)$ and $L(E_7)$ to be defined. In general, irrationalities in the Brauer characters of $M(E_7)$ and $L(E_7)$ occur only for elements of order $5,10,20$, and all satisfy $x^2\pm 2x-4$ (for $M(E_7)$) and $x^2-x-1$ (for $L(E_7)$). These equations split over $\F_q$ if and only if $q\equiv\pm 1\bmod 5$.

Thus we may replace $\sigma=F_p$ by $\sigma=F_{p^2}$ if $p\equiv\pm2\bmod 5$, and assume that $\sigma$ stabilizes the $\mbG$-class of subgroups. Choose $H$ in the class such that $\sigma$ normalizes $H$. We now must determine if $\sigma$ centralizes $H$ (up to conjugacy) or acts as the outer diagonal automorphism on $H$. This is equivalent to asking if $\sigma$ centralizes (up to conjugacy) the subgroup $L\cong 19\rtimes 9$ of $H$, or equivalently if $L\leq E_7(p^a)_{\mathrm{sc}}$ for a certain $a$.

\medskip

We have already seen in the discussion at the start of this section that $L$ lies diagonally in an $E_6T_1$-Levi subgroup; let $L_0$ denote the projection of $L$ onto the $E_6$ factor. The subgroup $L_0$ is unique in $E_6$, so the question is whether this lies in $E_6(q)_{\mathrm{sc}}$, ${}^2\!E_6(q)_{\mathrm{sc}}$, both or neither. Since the graph automorphism of $E_6$ cannot centralize $L_0$ (as it does not lie in $F_4$ or $C_4$) we have that it must act non-trivially on it. Since there is a unique $E_6$-class of subgroups $L_0$, any field automorphism of also normalizes it (up to conjugacy). Thus exactly one of the field and the graph-field must centralize $L_0$, so $L_0$ lies in exactly one of $E_6(q)_{\mathrm{sc}}$ and ${}^2\!E_6(q)_{\mathrm{sc}}$ for each $q$. If $d$ is the order of $q$ modulo $19$ (thus $d$ is one of $1,2,3,6,9,18$), we claim that $L_0$ lies in $E_6(q)_{\mathrm{sc}}$ if and only if $d$ is odd. We first construct $L_0$ inside $E_6(q)_{\mathrm{sc}}$ for $d$ odd, then by duality $L_0$ lies in ${}^2\!E_6(q)_{\mathrm{sc}}$ for $d$ even.

Thus assume that $d$ is odd. The result clearly holds for $d=1$, as $L_0$ lies inside the normalizer of a split torus. For the other cases, note that the orders of $q$ and $q^2$ modulo $19$ are the same, and $L_0$ definitely lies in $E_6(q^2)_{\mathrm{sc}}$. On the other hand, the Sylow $19$-subgroups of $E_6(q)_{\mathrm{sc}}$ and $E_6(q^2)_{\mathrm{sc}}$ are the same, and the automizers of the Sylow $19$-subgroups are the same (as they only depend on $d$). Thus $L_0$ lies in $E_6(q)_{\mathrm{sc}}$ if and only if it lies in $E_6(q^2)_{\mathrm{sc}}$, completing the proof.

We now use this to decide if $L$ lies in $E_7(q)_{\mathrm{sc}}$. If $q\equiv 1\bmod 3$ then $L$ must lie in $E_6(q)_{\mathrm{sc}}.(q-1)$, and this occurs if and only if $L_0$ lies in $E_6(q)_{\mathrm{sc}}$, so $d=1,3,9$. If $q\equiv -1\bmod 3$ then $L$ must lie in ${}^2\!E_6(q)_{\mathrm{sc}}.(q+1)$, and this occurs if and only if $L_0$ lies in ${}^2\!E_6(q)_{\mathrm{sc}}$, so $d=2,6,18$. Thus $L$ embeds in $E_7(q)_{\mathrm{sc}}$ if and only if the orders of $q$ modulo $3$ and $19$ are both odd or both even. In the other case it does not, and so we must move to $E_7(q^2)_{\mathrm{sc}}$, as claimed in the proposition.
\end{proof}

\subsubsection{\texorpdfstring{$\PSL_2(19)$}{PSL(2,19)} (Serre embedding)}

This time, as we will see that $p=2$ does not occur, let $H\cong \SL_2(19)$. This case is similar to the previous one, but there are some important distinctions. Here, the action of $H$ on $M(E_7)$ has composition factors of dimensions $10,10,18,18$. The two $18$-dimensional modules are the same as with the previous embedding, and the two $10$-dimensional modules are $10$ and $10^*$. The appropriate primes are $p=3$, $p=5$ and $p\neq 2,3,5,19$. (The case $p=2$ is missing as $10^\pm$ reduces modulo $2$ with factors $1$ and $9^\pm$, so $H$ stabilizes a line on $M(E_7)$.) If $p=5$ then as with the previous case the two $18$-dimensional modules become isomorphic, and as before the embedding must have structure $(18/18)\oplus 10\oplus 10^*$. If $p=3$ then $M(E_7){\downarrow_H}$ is semisimple.

We have the following proposition, with the proof of maximality delayed until Section \ref{sec:proofofmaximality}.

\begin{prop} Let $H\cong \SL_2(19)$ and $p\neq 2,19$. Suppose that $H$ is a subgroup of $\mbG$ and that $M(E_7){\downarrow_H}$ contains a composition factor of dimension $10$.
\begin{enumerate}
\item There are exactly two $\mbG$-conjugacy classes of subgroups $H$ unless $p=5$, in which case there is exactly one. In all cases $N_\mbG(H)=H.2$. If $\bar H\leq G$ and $p\neq 5$ then there are exactly four $G$-conjugacy classes of subgroups $\bar H$ if $q\equiv \pm1,\pm9\bmod 40$, and if $q\not\equiv \pm1,\pm9\bmod 40$ there are two $G$-conjugacy classes. If $q=5^{2n+1}$ then there is exactly one $G$-class of subgroups $\bar H$, and if $q=5^{2n}$ then there are exactly two.
\item $\bar H$ embeds in $G$ if and only if $q\equiv 0,\pm1\bmod 5$. $\bar H.2$ embeds in $G$ if and only if $q\equiv 0,\pm1\bmod 5$ and $q\equiv\pm 1\bmod 8$.
\item If $N_{\bar G}(\bar H)$ is maximal in $\bar G$ then $q$ is minimal such that $\bar H$ embeds in $G$, and either $\bar G=G$, or $q\not\equiv \pm1\bmod 8$ and $\bar G$ is $G$ extended by the diagonal automorphism.
\end{enumerate}
\end{prop}
\begin{proof}
The proof proceeds in stages.

\paragraph{Determination of $N_\mbG(H)$} Since this case appears in \cite{serre1996} for $p\neq 2,3,5,19$, and there it is shown that $\PGL_2(19)$ embeds in the adjoint version of $\mbG$ with these composition factors for $H$, we see that $N_\mbG(H)=H.2$ for all primes $p\neq 2,19$ by Proposition \ref{prop:p=0impallp}. As an independent proof, in the supplementary materials we also give an embedding of $H.2$ into $\mbG$ in characteristic $7$, hence it embeds for all $p\neq 2,3,5,19$ by Theorem \ref{thm:larsencorr}.

\paragraph{Determination of $H$ up to conjugacy}

In the supplementary materials, for $p=3$, $p=5$ and $p\neq 2,3,5,19$ we provide a complete proof of the uniqueness of this case. In all cases we use the subalgebra method from Section \ref{sec:subalgmethod}. The form is unique in all cases, following the same method as in the previous case for $H$.

\paragraph{Action of outer automorphisms}

We first decide on the action of the field automorphism $\sigma=F_p$.

For $p=5$ we require $\F_5$ for $M(E_7)$ and $L(E_7)$ to be defined. For $p\neq 5$, irrationalities in the Brauer characters of $M(E_7)$ and $L(E_7)$ occur only for elements of order $5,10,20$, and all satisfy $x^2\pm 2x-4$ (for $M(E_7)$) and $x^2-x-1$ (for $L(E_7)$). These equations split over $\F_q$ if and only if $q\equiv\pm 1\bmod 5$.

Thus we may replace $\sigma=F_p$ by $\sigma=F_{p^2}$ if $p\equiv\pm 2\bmod 5$, and assume that $\sigma$ stabilizes the $\mbG$-class of subgroups. Choose $H$ in the class such that $\sigma$ normalizes $H$. Thus $\sigma$ acts as an inner automorphism on $H_1=N_\mbG(H)$ and thus by Lemma \ref{lem:inneriscent} we may assume that $\sigma$ centralizes $H_1$ in the adjoint group (as $\PGL_2(19)$ is its own automorphism group). Thus $H\leq G_{\mathrm{sc}}=\mbG^\sigma$ in this case.

It remains to understand the action of the diagonal automorphism, whether it induces the diagonal automorphism on $H$ or fuses two classes of subgroups $\bar H$ in $G$. Equivalently, it remains to understand if $\bar H.2$ lies in $G$, or whether $\bar H=N_G(\bar H)$.

The character of the extension of $M(E_7){\downarrow_H}$ can be seen in \cite[p.11]{atlas}, or can be found from the matrices in the supplementary materials. Elements of order $8$ in $H.2$ act on $M(E_7)$ with Brauer character $\pm \sqrt 8$ (checked in the supplementary materials). Thus if $\sqrt8$ does not lie in $\F_q$ then $H.2$ cannot embed in $G$. We may assume therefore that $q\equiv \pm 1\bmod 8$ (which is equivalent to $\sqrt8$ lying in $\F_q$ by quadratic reciprocity).

Since $H$ is centralized by $\sigma$ (by assumption), $N_\mbG(H)$ is normalized by $\sigma$. Notice that, since $\F_q$ contains the character values of the two classes of elements of order $8$, $\sigma$ cannot induce an automorphism that swaps these two classes. It is an easy exercise, or a quick check with Magma (and we do so in the supplementary materials in the file for characteristic $0$, for completeness), that $\Out(H.2)$ has order $2$, and any outer automorphism indeed swaps these classes. Thus $\sigma$ induces an inner automorphism on $H.2$, so $\sigma$ centralizes some conjugate of $H.2$ by Lemma \ref{lem:inneriscent}. This completes the proof.
\end{proof}

\subsubsection{\texorpdfstring{$\PSL_2(13)$}{PSL(2,13)}}
\label{sec:psl213}
The appropriate primes here are $p=2$, $p=3$, $p=7$ and $p\neq 2,3,7,13$. From \cite[Tables 6.174--6.180]{litterickmemoir}, we find a number of possible sets of composition factors for both $\PSL_2(13)$ for $p=2,3,7$ and for $\SL_2(13)$ for $p\neq 2$. 

Suppose first that $p$ is odd and $H\cong \PSL_2(13)$. If $p=3$ then the simple $kH$-modules in the principal block are $1$ and $13$. Furthermore, $P(13)=13/1/13$. Thus if there are not at least twice as many $13$s as $1$s in $L(E_7){\downarrow_H}$ then $H$ stabilizes a line on $L(E_7)$. In fact, \cite[Table 6.178]{litterickmemoir} shows that there are at most six $13$s and at least four $1$s, so $H$ stabilizes a line on $L(E_7)$ in all cases.

If $p=7$ then the simple $kH$-modules in the principal block are $1$ and $12$. Furthermore, $P(12)=12/12/1,12/12$. Thus if there are not at least four times as many $12$s as $1$s in $L(E_7){\downarrow_H}$ then $H$ stabilizes a line on $L(E_7)$. In fact, \cite[Table 6.176]{litterickmemoir} shows that there are at most six $12$s and at least three $1$s, so $H$ stabilizes a line on $L(E_7)$ in all cases (and hence is strongly imprimitive by Lemma \ref{lem:fixedline}).

Thus we may assume that $H\cong \PSL_2(13)$ for $p=2$, and $H\cong \SL_2(13)$ for $p\neq 2$. The case $p=2$ is different to the rest, and we despatch this now.

\begin{prop} If $p=2$ then $H$ is strongly imprimitive.
\end{prop}
\begin{proof}
For $p=2$, the simple modules in the principal block are $1$, $6_1$ and $6_2$, and the projective cover of $6_i$ is
\[ 6_i/1/6_{3-i}/1/6_i.\]
Assume that $H$ stabilizes no lines on $M(E_7)$ or $L(E_7)^\circ$, as otherwise we are done. Thus we need at least three modules of dimension $6$ for every two trivial composition factors in $L(E_7)^\circ{\downarrow_H}$. Of the two rows of \cite[Table 6.180]{litterickmemoir} labelled `\textbf{P}', one does not satisfy this, so this case can be ignored.

The other option is that $M(E_7){\downarrow_H}$ has composition factors one each of all three (projective) $12$-dimensional modules, and then $6_1^2,6_2,1^2$, which must assemble into $P(6_1)$ if $H$ is to not stabilize a line on $M(E_7)$. The corresponding factors on $L(E_7)^\circ$ are
\[ 14^5,12_1,12_2,12_3,6_1,6_2^3,1^2,\]
and again to avoid stabilizing a line we must have $P(6_2)\oplus 6_2$. In order for the action of an element of order $4$ in $H$ to match up between $M(E_7)$ and $L(E_7)$ (using \cite[Tables 7 and 8]{lawther1995}), the $14$s must assemble into $P(14)\oplus 14^{\oplus 3}$. This determines the action on $L(E_7)^\circ$ uniquely. In order to apply the subalgebra method, we need a symplectic form on the $56$-dimensional module. Given a form on the $P(6_1)$, we can apply Proposition \ref{prop:ssnofactorsincommon} to obtain a unique (non-degenerate) alternating form (up to conjugacy in $\Sp_{56}(k)$) for the whole module. The space of alternating forms on $P(6_1)$ is $2$-dimensional (so up to scalar multiples it is a projective line), and so in the supplementary materials we introduce a parameter in the form, allowing us to deal with all forms simultaneously.

We find a unique conjugacy class of subgroups up to $\mbG$-conjugacy. But this set of factors is the reduction modulo $2$ of the set of factors mentioned earlier in this proof in characteristic $0$. Thus we find $H$ inside $G_2C_3$. For $p=2$, the composition factors of $G_2C_3$ on $M(E_7)$ have dimensions $36$, $8$, $6$ and $6$, with one of dimension $6$ being a submodule. Since $H$ stabilizes a unique $6$-space, this proves that $H$ is again strongly imprimitive via Proposition \ref{prop:intersectionsubspace}.
\end{proof}

We therefore assume that $p$ is odd (or $0$) and that $H\cong \SL_2(13)$ in what follows. From the tables in \cite{litterickmemoir} we find two similar but distinct sets of composition factors for $M(E_7)$ and $L(E_7)$. One always yields a strongly imprimitive subgroup, and so we prove this first, leaving the more interesting case for afterwards.

\begin{prop} If $p$ is odd and the action of $H$ on $M(E_7)$ has a composition factor of dimension $6$ then $H$ is strongly imprimitive.
\end{prop}
\begin{proof} The tables \cite[Tables 6.175, 6.177 and 6.179]{litterickmemoir} determine the composition factors and module structure of $M(E_7){\downarrow_H}$, but not necessarily that of $L(E_7){\downarrow_H}$.

If $p\neq 2,3,7,13$ then the actions of $H$ on our two modules are
\[ 12_4\oplus 12_5\oplus 12_6\oplus 14_4\oplus 6_1\quad\text{and}\quad 14_1\oplus 14_2^{\oplus 4}\oplus 13\oplus 12_1\oplus 12_2\oplus 12_3\oplus 7_2^{\oplus 2}.\]
Because $M(E_7){\downarrow_H}$ is multiplicity-free we may use Proposition \ref{prop:ssmultfree} to see that the embedding of $H$ into $\Sp_{56}(k)$ is unique.

If $p=7$ then the composition factors of $M(E_7){\downarrow_H}$ and $L(E_7){\downarrow_H}$, and the fact that a unipotent element acts with Jordan blocks from \cite{lawther1995}, forces the actions to be
\[ 14_3\oplus P(6_1)\quad \text{and}\quad 14_1^{\oplus 4}\oplus 14_2\oplus 7_2^{\oplus 2}\oplus P(12).\]
(Note that $P(12)=12/(1\oplus (12/12))/12$, so Rows 3 and 4 of \cite[Table 6.177]{litterickmemoir} must yield a trivial submodule, so $H$ is strongly imprimitive in these cases.)

If $p=3$ then we do the same analysis as for $p=7$, and obtain
\[ 14\oplus 12_4\oplus 12_5\oplus 12_6\oplus 6_1\quad \text{and}\quad P(13)\oplus 12_1\oplus 12_2\oplus 12_3\oplus M,\]
where $M$ is a module with composition factors $7_1^4,7_2^6$.

In each case we can apply the subalgebra method. For $p=3$, although the action of $H$ on $L(E_7)$ is not known precisely, if we just assume that $P(7_1)\oplus P(7_2)$ is a summand of $L(E_7){\downarrow_H}$ then this deals with all options. In each case, in the supplementary materials we obtain a single subalgebra up to conjugacy by the centralizer.

The last part is to notice that $H$ lies in the $G_2C_3$ positive-dimensional subgroup, acting irreducibly on both $M(G_2)$ and $M(C_3)$. (If one fixes a $7$-dimensional simple module for the action on $M(G_2)$ then the two options for the $6$-dimensional module $M(C_3){\downarrow_H}$ yield two sets of composition factors: the one in this proposition and the one we will consider next. Thus both possibilities come from $G_2C_3$.) Since $G_2C_3$ acts on $M(E_7)$ with composition factors of dimensions $42$ and $14$, we see that it stabilizes the unique $14$-space stabilized by $H$, and $H$ is strongly imprimitive by Proposition \ref{prop:intersectionsubspace}.
\end{proof}

We now may assume that there is no $6$-dimensional composition factor in $M(E_7){\downarrow_H}$, and from the tables in \cite{litterickmemoir} we obtain a unique set of composition factors for $M(E_7){\downarrow_H}$ and $L(E_7){\downarrow_H}$. (For $p=7$, note that $P(12)=12/(1\oplus (12/12))/12$, so Rows 3 and 4 of \cite[Table 6.177]{litterickmemoir} must yield a trivial submodule, so $H$ is strongly imprimitive in these cases, as noted in the previous proof.)

If $p\neq 2,3,7,13$ then the actions of $H$ on $M(E_7)$ and $L(E_7)$ are given by
\[ 14_4^{\oplus 2}\oplus 14_5\oplus 14_6\quad\text{and}\quad 14_1\oplus 14_2^{\oplus 4}\oplus 13\oplus 12_1\oplus 12_2\oplus 12_3\oplus 7_1\oplus 7_2\]
respectively, according to \cite[Table 6.175]{litterickmemoir}.

For $p=7$ we have
\[ 14_2^{\oplus 2}\oplus 14_4\oplus 14_5\quad\text{and}\quad 14_1^{\oplus 4}\oplus 14_2\oplus 7_1\oplus 7_2\oplus P(12)\]
respectively, according to \cite[Table 6.177]{litterickmemoir}.

For $p=3$, as in the previous proof we can determine $M(E_7){\downarrow_H}$ uniquely but not the adjoint module. Using \cite[Table 6.179]{litterickmemoir}, $H$ has composition factors $14^4$ on $M(E_7)$. If $u\in H$ has order $3$ then $u$ acts on $14$ with blocks $3^4,1^2$, so $u$ acts with at least sixteen blocks of size $3$ on $M(E_7)$. The only allowed action, according to \cite[Table 7]{lawther1995}, is $3^{18},1^2$, and therefore $H$ acts on $M(E_7)$ as $P(14)\oplus 14$.

There are two such unipotent classes: $2A_2$ and $2A_2+A_1$. They act on $L(E_7)$ with blocks $3^{42},1^7$ and $3^{43},2^2$ respectively. In each case the module is a sum of $P(13)\oplus 12_1\oplus 12_2\oplus 12_3$, and a module with factors $7_1^5,7_2^5$ (similar to the case in the proposition above). There are six ways of distributing these factors so that the unipotent action is as claimed. The lack of complete understanding of $L(E_7){\downarrow_H}$ does not affect our ability to use the subalgebra method, and we do so in the next proposition.

\begin{prop}\label{prop:sl213primitive} Let $H\cong \SL_2(13)$ and $p\neq 2,13$. Suppose that $H$ is a subgroup of $\mbG$, that $M(E_7){\downarrow_H}$ contains no composition factor of dimension $6$, and that $H$ stabilizes no line on $M(E_7)$ or $L(E_7)$.
\begin{enumerate}
\item There is exactly one $\mbG$-conjugacy classes of subgroups $H$, and $N_\mbG(H)=H.2$. The subgroup $H$ is Lie imprimitive, contained in $G_2C_3$, but the subgroup $H>2$ is Lie primitive.
\item The subgroup $\bar H$ embeds in $G$ for all $q$ a power of $p$.
\item If $N_{\bar G}(\bar H)$ is maximal in $\bar G$ then $q=p$, and one of the following holds:
\begin{enumerate}
\item $p\equiv \pm 1\bmod 8$ and $\bar H.2\leq G$ is maximal in $G=\bar G$;
\item $p\equiv \pm 3\bmod 8$, $p\equiv \pm2,\pm5,\pm6\bmod 13$ and $\bar G$ is either $G$ or $G.2$, where $H$ and $H.2$ respectively is maximal;
\item $p\equiv \pm 3\bmod 8$, $p\equiv \pm1,\pm3,\pm4\bmod 13$ and $H.2$ is a novelty maximal subgroup of $G.2$.
\end{enumerate}
\end{enumerate}
\end{prop}
\begin{proof}
The proof proceeds in stages.

\paragraph{Determination of $N_\mbG(H)$} In the supplementary materials we prove that $H.2$ embeds in $\mbG$ in characteristics $3$, $7$ and $181$, hence it embeds for all $p\neq 2,13$ by Theorem \ref{thm:larsencorr}.

\paragraph{Determination of $H$ up to conjugacy}

In the supplementary materials, for $p=3$, $p=7$ and $p\neq 2,3,7,13$ we provide a complete proof of the uniqueness of this case. In all cases we use the subalgebra method from Section \ref{sec:subalgmethod}. The form is unique in all cases by a direct calculation with the space of forms in the supplementary materials.

\paragraph{Action of outer automorphisms}

Fix $H\leq \mbG$, and we first decide on the action of the field automorphism $\sigma=F_p$. Since $H$ is unique up to $\mbG$-conjugacy, if $H^\sigma=H^g$ for some $g\in \mbG$, whence $H$ is normalized by the automorphism $\sigma\cdot c_{g^{-1}}$, where $c_x$ denotes conjugation by $x$. Replacing $H$ by a conjugate, we assume that $\sigma$ normalizes $H$. Since $H.2$ embeds in $\mbG$, we have $g\in \mbG$ such that $\sigma\cdot c_{g^{-1}}$ in fact centralizes $H$, so $H$ embeds in $2\cdot E_7(p)$.

We now must determine whether $H.2$ embeds in $G_{\mathrm{sc}}$ or whether the outer diagonal automorphism of $G_{\mathrm{sc}}$ induces an outer automorphism on $H$. First, we determine the (Brauer) character of $H.2$ on $M(E_7)$, and the character values are
\[ \{\pm 56,\pm 4,\pm 2,0,\pm (\zeta_8+\zeta_8^{-1}),\pm 2(\zeta_8+\zeta_8^{-1})\},\]
where $\zeta_8$ is a primitive $8$th root of unity. These lie in $\F_p$ if and only if $p\equiv \pm 1\bmod 8$, so certainly $\bar H.2$ cannot embed in $G$ if $p\equiv\pm 3\bmod 8$. We show that $\bar H.2$ \emph{does} embed in $G$ if $p\equiv\pm 1\bmod 8$ now.

Let $H_0=N_\mbG(H)$, and assume that $\sigma$ centralizes $H$, so it normalizes $H_0$. We check in the supplementary materials that the outer automorphism group of $H_0$ has order $2$, and that an outer automorphism of $H_0$ swaps two classes of elements of order $8$ with traces $\pm2(\zeta_8+\zeta_8^{-1})$ on $M(E_7)$. Thus if $\zeta_8+\zeta_8^{-1}\in \F_p$ then $\sigma$ cannot induce this outer automorphism on $H_0$. Thus $\sigma$ incudes an inner automorphism on $H_0$, so a conjugate of $H_0$ lies in $G_{\mathrm{sc}}$ by Lemma \ref{lem:inneriscent}.

\paragraph{Containment in $G_2C_3$}

To determine the rest of the maximality statement we have to show when $H$ and $H.2$ are contained in $G_2C_3$. From \cite[Tables 8.29, 8.41 and 8.42]{bhrd}, $\PSL_2(13)$ is contained in $G_2(q)$ and $\PSp_6(q)$ precisely if $q\equiv \pm1,\pm3,\pm4 \bmod 13$. Thus if $q\equiv\pm2,\pm5,\pm6\bmod 13$ then $H$ does not embed in $G_2(q)C_3(q)$ and hence $N_{\bar G}(\bar H)$ must be maximal in $\bar G$. On the other hand, if $q\equiv \pm1,\pm3,\pm4\bmod 13$ then $H$ \emph{is} contained in $G_2(q)C_3(q)$. Since the $6$- and $7$-dimensional modules are not stabilized by the diagonal automorphism of $\SL_2(13)$ (or by checking the tables in \cite{bhrd}) one sees that $H.2$ does not embed in $G_2C_3$, and hence $H.2$ is Lie primitive.

(There could, of course, be another subgroup containing $H$ or $H.2$, and we exclude this possibility in Section \ref{sec:proofofmaximality}.)
\end{proof}

\medskip

\subsubsection{\texorpdfstring{$\PSL_2(11)$}{PSL(2,11)}}
\label{sec:psl211}
Let $p\neq 11$ be a prime. The sets of composition factors for $\PSL_2(11)$ and $\SL_2(11)$ on $M(E_7)$ and $L(E_7)$ are in \cite[Tables 6.167--6.173]{litterickmemoir}. From these, we see that if $p$ is odd then $\PSL_2(11)$ is strongly imprimitive, and so we may assume that $H\cong\SL_2(11)$ for $p$ odd, and $H\cong \PSL_2(11)$ for $p=2$.

We start by proving that $H$ is always strongly imprimitive for $p=5$, despite there being a case labelled `\textbf{P}' in \cite[6.170]{litterickmemoir}. Using this we show the same thing for $p\neq 2,3,5$. In addition, for $p=2$ we can also show that $H$ is strongly imprimitive, leaving only $p=3$. The proof for this case is much more complicated, and requires both a computer calculation and theoretical arguments.

\begin{prop} If $p\neq 3$ then $H$ is strongly imprimitive.
\end{prop}
\begin{proof} We start with $p=5$. If $H$ is not strongly imprimitive then $H\cong \SL_2(11)$ and the composition factors of $L(E_7){\downarrow_H}$ are (up to automorphism) $11^8,10_1^3,5,5^*,1^5$. (See \cite[Tables 6.169--6.170]{litterickmemoir}.) Although this has pressure $3$, any such module must have a trivial submodule. To see this, note that the projective cover of $11$ is $11/1/11/1/11$, so for every two trivials we need three $11$s in order to not to have a trivial submodule. Since this is not the case, $H$ stabilizes a line on $L(E_7)$. Thus $H$ lies in either a parabolic or maximal-rank subgroup, by Lemma \ref{lem:paramaxl}, and in a connected one by Lemma \ref{lem:inconnected}.

We now show that $H$ cannot lie in such a subgroup, and therefore no such $H$ embeds in $\mbG$ with these factors on $L(E_7)$. The corresponding factors on $M(E_7)$ are $10_4,10_5,(6,6^*)^3$. (It is not important which two modules $10_4$ and $10_5$ are.) Notice that the minimal faithful degree for $H$ is $6$ and no faithful simple module for $H$ carries a symmetric form. Thus $H$ does not embed in $\mb X$ with $Z(H)\leq Z(\mb X)$ for $\mb X$ of type $A_n$ for $n\leq 5$ and $n=7$, or of type $D_n$ for $n\leq 5$.

If $H$ embeds in a parabolic subgroup then there is some copy of $H$ with the same composition factors in the Levi subgroup $\mb X$. From above we see that $\mb X$ is of type $D_6$ or $A_5$, and if $\mb X=D_6$ then $H$ embeds into $A_5$. There are two classes of $A_5$-Levi subgroup of $\mbG$, and they have composition factors as given in, for example, \cite[Table 21]{thomas2016}. One of these lies in $A_7$ and clearly cannot work. The other does work for $M(E_7)$, with $H$ acting irreducibly on $M(A_5)$. However, this subgroup has eight trivial factors on $L(E_7)$, so this is not correct. (It is the other option in \cite[Table 6.170]{litterickmemoir}.) 

If $H$ embeds in a maximal-rank subgroup of $\mbG$ then the same arguments apply, and we see that $H$ lies in $D_6A_1$ (and then inside the $D_6$-Levi subgroup) or in $A_2A_5$ (and then inside the $A_5$-Levi subgroup). Since these have already been considered, we prove the result.

\medskip

We now complete the case $p\neq 2,3,5$, so $p=0$. From \cite[Tables 6.167--6.168]{litterickmemoir} there is a unique possible set of composition factors that is labelled `\textbf{P}', and the reduction modulo $5$ of these factors for $M(E_7)$ and $L(E_7)$ would be the case above that we have excluded. Thus it does not exist either by Proposition \ref{prop:p=0impallp}.

Finally, for $p=2$, we just show that the two potential cases in \cite[Table 6.173]{litterickmemoir} are in fact strongly imprimitive. The projective cover of $5$ is
\[ 5/1,5^*/5,\]
and so in any module with trivial factors we require three $5$-dimensional factors for every trivial factor (except remember to remove one trivial from $L(E_7){\downarrow_H}$ as $L(E_7)^\circ\neq L(E_7)$). There are ten $5$-dimensional factors and five trivial factors in the case from that table, and so $H$ stabilizes a line on $L(E_7)^\circ$ and is strongly imprimitive as well by Lemma \ref{lem:fixedline}.
\end{proof}

For the rest of this section, let $p=3$. From \cite[Tables 6.171 and 6.172]{litterickmemoir} we find exactly two cases labelled with `$\mathbf{P}$', both for $\SL_2(11)$. Letting $10_1$ denote the $10$-dimensional module for $\PSL_2(11)$, the projective cover of $10_1$ is $10_1/1/10_1$. Thus in order for $L(E_7){\downarrow_H}$ not to have a trivial submodule we need two copies of $10_1$ for every $1$. This is not the case in Row 2 of \cite[Table 6.172]{litterickmemoir}, so $H$ always stabilizes a line on $L(E_7)$ in that case. The remaining one has factors
\[ 12_1^3,12_2^2,10_1^6,5,5^*,1^3.\]
The $12_i$ are projective so split off, we must have $P(10_1)^{\oplus 3}$ in order to avoid stabilizing a line (and hence being strongly imprimitive by Lemma \ref{lem:fixedline}), and so it remains to deal with $5$ and $5^*$.

We have a projective module of dimension $123$ and two $5$-dimensional factors, hence at least 43 Jordan blocks of size $3$ in the action of $u$ in $H$ of order $3$ on $L(E_7)$. We see using \cite[Table 8]{lawther1995} that this is only consistent with $u$ coming from class $2A_1+A_1$, acting as $3^{43},2^2$ on $L(E_7)$ and $3^{18},1^2$ on $M(E_7)$.

The actions of $u$ on $5\oplus 5^*$ and $5/5^*$ are $3^2,2^2$ and $3^3,1$ respectively, so we see that we have the semisimple case. The composition factors of $M(E_7){\downarrow_H}$ are $12_4^2,10_2^2,6,6^*$. The actions of $u$ on $10_2^{\oplus 2}$ and $10_2/10_2$ are $3^4,2^4$ and $3^6,1^2$ respectively, so the second case must be the correct module. Thus $H$ acts on $M(E_7)$ and $L(E_7)$ as the modules
\[ 12_4^{\oplus 2}\oplus 6\oplus 6^*\oplus (10_2/10_2),\qquad P(10_1)^{\oplus 3}\oplus 12_1^{\oplus 2}\oplus 12_2^{\oplus 3}\oplus 5\oplus 5^*.\]
(This cannot lie in $D_6$ as the $12_i$ stabilize a symplectic form.) This is the reduction modulo $3$ of the non-existent case for $p=0$ in the proof above, but of course this does not prove that $H$ cannot exist for $p=3$. It does however show that any such $H$ is Lie primitive.

\medskip

Our first task when analysing this case is to determine the number of such classes in $\Sp_{56}(k)$, where $k$ is an algebraically closed field of characteristic $3$.

\begin{lem} Let $M(E_7){\downarrow_H}$ be as above. There is exactly one conjugacy class of subgroups of $\mb X=\Sp_{56}(k)$ isomorphic to $H$, with action on $M(\mb X)$ as above.
\end{lem}
\begin{proof} Let $H$ be as above. We first examine the possible symplectic forms on $6\oplus 6^*$, $10_2/10_2$ and $12_4^{\oplus 2}$, and then put them together.

Of course, there is a unique symplectic form on $6\oplus 6^*$. For $10_2/10_2$, a quick calculation (e.g., on a computer) shows that, although $10_2$ supports a symplectic form, there is a $2$-space of symplectic forms on $10_2/10_2$, and the socle lies in the radical of all of them. In particular, this means that the socle is degenerate and so the subgroup lies in an $A_9$-parabolic subgroup. (Another way to see this is that if the socle were non-degenerate, $H$ would lie in a subgroup $\Sp_{10}\times \Sp_{10}$, which acts semisimply on $M(\Sp_{20})$.) We may now apply Corollary \ref{cor:symp1cohom} since the statement about the symmetric square is true in this case, to see that $H$ uniquely embeds (up to conjugacy) in $\Sp_{20}(k)$.

The last summand is $12_4^{\oplus 2}$. Here an individual submodule $12_4$ is either totally isotropic or non-degenerate. In the first case, this places $H$ into an $A_{11}$-Levi subgroup, and then into a $C_6$-subgroup of $A_{11}$. The alternative places $H$ inside a diagonal $C_6$ inside the $C_6C_6$ maximal-rank subgroup. This yields two potential embeddings of $H$ into $\Sp_{24}(k)$. One may check over $\F_9$ that there are two distinct classes of subgroups $H$ with this action in $\Sp_{24}(k)$. (It is easier to check that there are two non-conjugate symplectic forms, because they yield different orders of centralizers in $\Sp_{24}(9)$.) However, over $\F_{81}$ these two subgroups are conjugate. (An example of this computation is given in the supplementary materials.)

We can also show this theoretically: if $V=W_1\oplus W_2$ with $W_1,W_2$ non-degenerate and isomorphic as modules, with $u_1,v_1\in W_1$ such that $(u_1,v_1)\neq 0$, then let $u_2,v_2\in W_2$ be images of $u_1,v_1$ under an isomorphism. Then $(u_1+\lambda u_2,v_1+\lambda v_2)=(1+\lambda^2)(u_1,v_1)$, which is zero if $\lambda$ is a square root of $-1$. Thus the $C_6$ inside $C_6C_6$ lies in $A_{11}$.

Since each of these summands shares no composition factors with any others, $H\leq \Sp_{56}(k)$ lies inside a product of symplectic groups, $\Sp_{12}\times \Sp_{20}\times \Sp_{24}$. As there are one, one and two forms respectively for these subgroups, one obtains exactly one conjugacy class of subgroups using Proposition \ref{prop:ssnofactorsincommon}.
\end{proof}

For this class, one must show that there is no subgroup $H$ of $E_7(k)$ arising from that symplectic form. The proof of this is computer-based, and in the supplementary materials, but because it is significantly harder than most proofs here, we give an outline of the proof. We use the subalgebra method from Section \ref{sec:subalgmethod}.

\begin{enumerate}
\item Let $W_1$ denote the subspace $5$ of $L(E_7){\downarrow_H}$, $W_2$ be $5^*$, $W_3$ be the $24$-dimensional subspace $12_1^{\oplus 2}$ and $W_4$ be the $36$-dimensional subspace $12_2^{\oplus 3}$. Let $a$ and $b$ lie in $W_1$, and $c,d$ lie in $W_2$. There are ten copies each of $W_1$ and $W_2$ in $S^2(M(E_7){\downarrow_H})$.
\item The commutator $[a,d]$ projects onto $W_3$ and $W_4$, so let $[a,d]_3$ and $[a,d]_4$ denote these projections. We show that $[[a,d]_i,b]$ and $[[a,d]_i,e]$ have zero image on $W_1$ and $W_2$, yielding some equations in the 20 variables for the location of $W_1\oplus W_2$. This eliminates entirely four of those 20 variables.
\item Using the centralizer of $H$ in $\Sp_{56}(k)$, we may assume various combinations of variables are either $0$ or $1$. Together with a few linear relations, this allows us to construct some candidates for $W_1$ or $W_2$ (or both). This includes linear pencils of subspaces as candidates.
\item By checking specific elements in these candidates, we find nilpotent elements with incorrect Jordan normal form on $S^2(M(E_7))$. Thus these candidates cannot be correct. This imposes more restrictions on the variables.
\item Products of elements in $W_1\oplus W_2\oplus W_3\oplus W_4$ always seem to lie in a specific subspace of $S^2(M(E_7))$, whose homogeneous $12_1$- and $12_2$-components consist of a particular subspace of dimension $72$, with three factors from each $12_i$. This $72$-space is abelian and commutes with $W_1$. Thus in order not to have a $29$-dimensional abelian subalgebra (contradicting Proposition \ref{prop:abeliansubalgebra}), the intersection of this $72$-space with $W_3\oplus W_4$ must be a single copy of either $12_1$ or $12_2$.
\item Using this extra information, deduce a final contradiction using much more solving of equations, and knowing that we cannot lie in the candidate subspaces.
\end{enumerate}

Thus we obtain the following result.

\begin{prop} Let $H$ be a central extension of $\PSL_2(11)$ and $p\neq 11$. Then $H$ is strongly imprimitive.
\end{prop}

In the supplementary materials we show this for where the $12_4^{\oplus 2}$ comes directly from both $A_{11}$ and $C_6C_6$. Of course, these are conjugate, so they give two independent proofs of the same statement.

\section{Cases not determined}
\label{sec:leftovers}
In this section we discuss the progress made for groups $\PSL_2(r)$ for $r=7,8,9$. Although we are not able to completely eliminate them, significant progress can be made in all cases, and some primes are eliminated for each.

\subsection{\texorpdfstring{$\PSL_2(7)$}{PSL(2,7)}}
\label{sec:psl27}
The appropriate primes here are $p=2$, $p=3$, $p=7$ and $p\neq 2,3,7$. For $p=2$ this the group $\PSL_3(2)$, a defining-characteristic case of rank $2$, and was proved to be strongly imprimitive in \cite[Proposition 9.1]{craven2019un}. The groups $\PSL_2(7)$ and $\SL_2(7)$ were proved to always be strongly imprimitive for $p=3$ in \cite[Proposition 4.7]{craven2020un}.

\medskip

For $p=7$ there were two cases that could not be proved to be strongly imprimitive in \cite[Propositions 13.4 and 14.4]{craven2015un}: for $\PSL_2(7)$ in $\mbG$, the actions on $M(E_7)$ and $L(E_7)$ can be
\[ 7^{\oplus 4}\oplus P(3)^{\oplus 2}\quad\text{and}\quad 7^{\oplus 5}\oplus P(5)^{\oplus 6}\oplus P(3)\]
respectively. For $\SL_2(7)$ in $\mbG$ (with centres coinciding), the actions on $M(E_7)$ and $L(E_7)$ can be
\[ P(6)^{\oplus 2}\oplus P(4)\oplus 6\oplus 4^{\oplus 2}\quad\text{and}\quad 7^{\oplus 5}\oplus P(5)^{\oplus 3}\oplus 5\oplus 3^{\oplus 3}\]
respectively.

\medskip

For $p\neq 2,3,7$, from \cite[Tables 6.160 and 6.161]{litterickmemoir} there is again one option for $\PSL_2(7)$ and one for $\SL_2(7)$, if the group is not strongly imprimitive. These are
\[ 8_1^{\oplus 2}\oplus 7^{\oplus 4}\oplus (3\oplus 3^*)^{\oplus 2}\quad\text{and}\quad 8_1^{\oplus 7}\oplus 7^{\oplus 5}\oplus 6^{\oplus 6}\oplus (3\oplus 3^*)\]
for $\PSL_2(7)$, and 
\[ 8_2^{\oplus 2}\oplus 6_2^{\oplus 2}\oplus 6_3\oplus (4\oplus 4^*)^{\oplus 2}\quad\text{and}\quad 8_1^{\oplus 7}\oplus 7^{\oplus 5}\oplus 6_1^{\oplus 3}\oplus (3\oplus 3^*)^{\oplus 4}\]
for $\SL_2(7)$, with some arbitrary labelling of the modules $6_2$ and $6_3$ (so that there are two such embeddings into $\GL_{56}(k)$). These possibilities all exist inside $A_2A_5$, with $H$ acting irreducibly on $M(A_2)$ and $M(A_5)$. (One obtains $\PSL_2(7)$ if one chooses a $6$-dimensional module for $\PSL_2(7)$, and $\SL_2(7)$ if one chooses a $6$-dimensional module for $\SL_2(7)$.)

\subsection{\texorpdfstring{$\PSL_2(8)$}{PSL(2,8)}}
\label{sec:psl28}
Let $H\cong \PSL_2(8)$. The appropriate primes are $p=2$, $p=3$, $p=7$ and $p\neq 2,3,7$. For all but $p=2$, we can prove that $H$ is strongly imprimitive.

\begin{prop} If $H\cong \PSL_2(8)$ and $p\neq 2$ then $H$ is strongly imprimitive.
\end{prop}
\begin{proof} Since $\PSL_2(8)\cong {}^2\!G_2(3)'$, if $p=3$ then $H$ is strongly imprimitive by \cite[Proposition 9.4]{craven2019un}. If $p=0$ then $H$ stabilizes a line on $L(E_7)$, as we see from \cite[Table 6.164]{litterickmemoir}, and thus $p=7$ remains. Here the simple modules are $1$, $8$ and $7$-dimensional projective modules.

The projective cover of $8$ is $8/1/8/1/8/1/8$, so in any module we need five $8$s for every four trivials, otherwise $H$ stabilizes a line or hyperplane. We see that this does not occur for $H$ from \cite[Table 6.165]{litterickmemoir}, so $H$ stabilizes a line on $L(E_7)$. In particular, $H$ is strongly imprimitive by Lemma \ref{lem:fixedline}.
\end{proof}

For $p=2$ there is a case, left open in \cite[Proposition 12.5]{craven2015un} where $M(E_7)$ is the sum 
\[ P(4_{1,2})\oplus P(4_{1,3})\oplus P(4_{2,3})\oplus 8,\]
where the $4_{i,j}$ are the $4$-dimensional simple modules. Each $P(4_{i,j})$ has a $2$-space of alternating forms, so $M(E_7){\downarrow_H}$ has a $3$-parameter space of alternating forms. The action on $L(E_7)^\circ$ is currently undetermined. Its composition factors are
\[ 8^2,(4_{1,2},4_{1,3},4_{2,3})^4,(2_1,2_2,2_3)^9,1^{14}.\]
Since $u\in H$ of order $2$ acts projectively on $M(E_7)$, $u$ acts on $L(E_7)^\circ$ with blocks either $2^{53},1^{27}$ (class $(3A_1)''$, which is not possible) or $2^{63},1^6$ (class $4A_1$).

Such a subgroup embeds in the normalizer of a maximal torus (and certainly exists in $E_7(q)$, as it exists in $3^7\rtimes W(E_7)$, which is a subgroup of $E_7(2)$), and that copy at least stabilizes a line on $L(E_7)^\circ$.

\medskip

We cannot prove that any such subgroup $H$ must stabilize a line on $L(E_7)^\circ$, although that appears to be the case. We can, however, prove it at least for $q=4$, and hence also for $q=2$. (This was already proved for $q=2$ in \cite{bbr2015}, but their proof appears not to generalize in any reasonable way.) In theory the proof could extend to all $q$, but would need more sophisticated ideas, possibly involving intersections of multiple parabolics.

\medskip

In general, if $L$ denotes a Borel subgroup of $H$ then the permutation module of $H$ on the cosets of $L$ is the sum $1\oplus 8$. Thus 
\[ \dim(L(E_7)^\circ)^L=\dim (L(E_7)^\circ)^H+2,\]
since there are two copies of $8$ in $L(E_7)^\circ{\downarrow_H}$. The next result completes the proof for $q=2,4$ by showing that the $L$-fixed space has dimension at least $3$, thus proving that $H$ stabilizes a line on $L(E_7)^\circ$.

\begin{prop}\label{prop:psl28q=4} All copies of $L$ in $E_7(4)$ that act as the free module $kL$ on $M(E_7)$ lie in an $A_5$-parabolic subgroup of $E_7(4)$, and centralize at least a $3$-space on $L(E_7)^\circ$.
\end{prop}
\begin{proof} We will first show that $L$ must be contained in an $A_5$-parabolic subgroup of $G=E_7(4)$. At that point we use a computer to enumerate all copies of $L$ with the correct traces, first in the $A_5$-Levi subgroup and then in an $A_5$-parabolic subgroup of $E_6$. It turns out that, with some obvious constraints, there is a single option for the actions of $L$ on $M(E_6)$ and $L(E_6)$. Furthermore, there is no extension of these modules with top $L(E_6)$ and socle $M(E_6)^*$ that does not centralize a $3$-space, proving the result.

\medskip

Since $L$ acts on $M(E_7)$ as the module $kL$, it centralizes a unique line on it, and this line is not complemented. This means that this line must be centralized by either an $E_6$-parabolic subgroup or a group $\mb U\cdot C_5(q)$. (It cannot be in a group $\mb U\cdot F_4$, since that group centralizes a $2$-space on $M(E_7)$.) We will eliminate the latter by showing that the trace of an element $x\in L$ of order $7$, which is $0$ on $M(E_7)$, cannot work.

First, we consider the collection of elements of order $7$ in $D_6$ whose action on $M(E_7)$ has trace $0$, so they can be conjugated into $L$. By a computer calculation, the trace on $M(D_6)$ is either $-2$ or, for some choice of $7$th root of unity, $3\zeta_7+3\zeta_7^{-1}+\zeta_7^2+\zeta_7^{-2}+4$. Since $C_5$ acts on $M(D_6)$ with two trivial composition factors, the former is impossible as it has no eigenvalue $1$. The latter is missing two eigenvalues, and no submodule of the free module $kL$ can be constructed with four trivial composition factors and missing one non-trivial module and its dual. Thus $L$ cannot act on $M(D_6)$ with this trace for an element of order $7$, and $L$ lies inside an $E_6$-parabolic subgroup of the algebraic group $\mbG$. (Notice not an $E_6T_1$-parabolic, since $L$ centralizes the line.)

\medskip

The subgroup $L_1\cong 2^3$ of $L$ cannot lie in a $D_5$-parabolic subgroup. If it did then $L$ would lie inside a $D_5$-parabolic subgroup. One of these centralizes a $2$-space on $M(E_7)$, the other has a submodule of dimension $10$ modulo the centralized line. Thus we look for submodules of $kL$ of dimension $11$. The trace of $x$ on $M(D_5)$ must be $3\zeta^2+3\zeta^{-2}+\zeta^3+\zeta^{-3}+2$, so again two submodules cannot occur. Using the labelling $i$ for the module on which $x$ acts with eigenvalues $\zeta^i$ (so that the trivial module becomes $0$ in this notation), the largest submodule of $kL$ with these factors is
\[ (35/0)\oplus (3/05/2)\oplus (2/3)\oplus (5/023/4)\oplus (2/34/5).\]
We must have all three copies of $0$, so need the $4$ underneath one of them. This means we cannot have the $4$ in the last term, so so we cannot have the top $2$ on that either. This leaves three $2$s, all of which must occur in our submodule. But now both $3$ and $4$ are submodules but neither is a summand, so this module cannot be self-dual.

The consequence of this is that $L_1$ cannot centralize a line on $M(E_6)$ (for then it would lie in a $D_5$-parabolic, rather than a $D_5T_1$-parabolic).

\medskip

If $q$ is not a power of $8$ then $L$ has composition factors of dimension $1$ and $3$ on $M(E_6)$. Note also that, under these conditions, $L$ cannot embed in $\SL_2(q)$. Letting $\bar L$ denote the image of $L$ in $E_6(q)$, if $\bar L$ embeds in a parabolic subgroup of $E_6(q)$, then it cannot be $D_5T_1$ (as this has a $1$-dimensional submodule or quotient on $M(E_6)$), or $A_4A_1T_1$ (as $\bar L$ then lies in the $D_5T_1$-parabolic). If $\bar L$ lies in the $A_2A_2A_1T_1$-parabolic then it lies in the $A_5T_1$-parabolic, and so we may assume that. Furthermore, since $7\nmid(q-1)$, $L$ lies inside the $A_5$-parabolic subgroup.

For reductive subgroups, if $L$ lies in $A_2A_2A_2$ then $q$ is a power of $8$, since otherwise $L$ cannot embed in $A_2$. If $L$ embeds in $A_1A_5$ then it embeds in $A_5$, so we are in the $A_5$-parabolic again. The only other subgroup is $A_2G_2$, and again $L$ does not lie in either $A_2^\pm(q)$ or $G_2(q)$ when $q$ is not a power of $8$. Thus $L$ lies in the $A_5$-parabolic. This completes the proof of the first statement.

\medskip

The composition factors of $A_5$ on $M(E_6)$ are two copies of $L(\lambda_1)$ and one of $L(\lambda_4)$, the exterior square of the dual of $L(\lambda_1)$. The composition factors of $\bar L$ on $M(E_6)$ are four copies of all non-trivial modules and three trivials, and it is an easy calculation that $\bar L$ must act on $M(E_6)$ with factors one of each of the two $3$-dimensional modules.

Now we simply enumerate all $\F_4L$-modules with these composition factors (there are seven of them, including the semisimple case), and then on a computer construct the group $q^{1+20}\cdot \bar L$, the preimage of $\bar L$ in the $A_5$-parabolic subgroup, for each option for $\bar L$. This group is small enough for $q=4$ to enumerate all subgroups $L$, so one now has a list of all possible $L$.

We can exclude all those whose action $M$ on $M(E_6)$ has a trivial submodule or quotient. We can also exclude those for which $\Ext^1(M,k)$ is zero. This is because the structure of $M(E_7){\downarrow_{E_6}}$ is $1/M^*/M/1$, and one sees from the action of $L$ that the non-split extension in the first two layers must remain non-split on restriction to $L$. We can also exclude with $\soc(M)$ of dimension greater than $9$, because that is the dimension of the socle of $M(E_7){\downarrow_L}$ modulo its fixed point.

Using a computer, we impose these restrictions on the set of such $L$. We cannot compute conjugacy in $E_6(4)$, but we can show that all of the $L$-actions on $M(E_6)$ and $L(E_6)$ are now isomorphic. Thus we can choose one and work with that. Notice that $L(E_7)^\circ$, viewed as a $kE_6$-module, has socle $M(E_6)^*$ and second layer $L(E_6)$. Let $M$ denote the action of $L$ on $M(E_6)$. A direct-sum decomposition of the action of $L$ on $L(E_6)$ always has two copies of the projective $P(1)$, two summands with no fixed point, and a single $19$-dimensional summand $M'$ with a $1$-dimensional fixed space.

We find that the group $\Ext^1(M',M^*)$ is $3$-dimensional, and we simply construct all $64$ points of this over $\F_4$. Each has a fixed point. Together with the two copies of $P(1)$ that must become summand in any extension, this shows that $L$ always centralizes a $3$-space on $L(E_7)^\circ$, as needed.
\end{proof}

\subsection{\texorpdfstring{$\PSL_2(9)$}{PSL(2,9)}}
\label{sec:alt6}

In Section \ref{sec:altgroups} we noted that $\Alt(6)=\PSL_2(9)$ and $2\cdot \Alt(6)=\SL_2(9)$ are strongly imprimitive for $p=2,3$, and this is not known for $p=5$ and $p\neq 2,3,5$. The open cases are enumerated in \cite{craven2017}, and we just list them now.

For $p=0$ and $H\cong \Alt(6)$, there is a unique possible set of composition factors, namely 
\[ 10^{\oplus 4}\oplus 8_1^{\oplus 2}\quad\text{and}\quad 10^{\oplus 2}\oplus 9^{\oplus 3}\oplus 8_1^{\oplus 4}\oplus 8_2^{\oplus 3}\oplus 5_1^{\oplus 3}\oplus 5_2^{\oplus 3}.\]

For $p=0$ and $H\cong 2\cdot \Alt(6)$, there are two possible sets of composition factors for $M(E_7){\downarrow_H}$, and in both cases $H$ acts on $L(E_7)$ as
\[ 10_1^{\oplus 5}\oplus 9^{\oplus 3}\oplus 8_1^{\oplus 4}\oplus 8_2^{\oplus 3}.\]
The two actions on $M(E_7)$ are
\[ 8_4^{\oplus 2}\oplus 10_2^{\oplus 3}\oplus 10_3\quad\text{and}\quad 8_4^{\oplus 2}\oplus 10_2\oplus 10_3^{\oplus 3}.\]
(Here $8_4$ is a specific choice of one of the two $8$-dimensional faithful irreducible modules, and this depends on the choice of $8_1$. The two cases correspond to the fact that an outer automorphism of $H$ acts as $(8_1,8_2)(8_3,8_4)(10_2,10_3)$.) These appear to be missing from Frey's table \cite[Table 14]{frey2016}. From the character values, this can only yield a subgroup of $E_7(q)$ if $q\equiv \pm 1\bmod 5$, and for the $2\cdot \Alt(6)$ case if $q\equiv \pm 1\bmod 5$ and $q\equiv \pm 1\bmod 8$.

\medskip

For $p=5$, we have the restrictions of the two possibilities for $\Alt(7)$ and $2\cdot \Alt(7)$ in Section \ref{sec:alt7}, which appear in \cite[Proposition 6.1]{craven2017}. Thus $H\cong \Alt(6)$ can act on $M(E_7)$ and $L(E_7)$ as
\[ 10^{\oplus 4}\oplus 8^{\oplus 2}\quad\text{and}\quad 10^{\oplus 2}\oplus 5_1^{\oplus 3}\oplus 5_2^{\oplus 3}\oplus P(8)^{\oplus 3}\oplus 8.\]
Alternatively, $H\cong 2\cdot\Alt(6)$ can act as 
\[ 10_2^{\oplus 3}\oplus 10_3\oplus (4_1/4_2)\oplus (4_2/4_1)\quad\text{and}\quad 10_1^{\oplus 5}\oplus P(8)^{\oplus 3}\oplus 8.\]
The former can exist over any field of characteristic $5$ but the latter requires $\F_{25}$ to be a subfield.

We have not attempted it, but the presence of no $5_i$ modules in the actions on $L(E_7)$ in three of the five cases might mean that the subalgebra method can be successful. The $\Alt(6)$ cases though look out of reach.

\section{Proof of maximality}
\label{sec:proofofmaximality}

As with the proof given in the previous paper in this series \cite{craven2020un}, the first sections of this article just found some conjugacy classes of subgroups, and asserted that they are maximal. Let $\mathcal S$ denote the set of subgroups enumerated in Table \ref{t:themaximals} (which of course depends on the prime $p$). If the assertion that $\bar H\in \mathcal S$ is maximal is false then an obstruction to this (i.e., a subgroup $X$ such that $\bar H<X<G$) must come either from $\mathcal S$ as well, or from one of the other maximal subgroups, which are enumerated in Table \ref{t:othermaximals} above.

Our first proposition starts to examine the case where the obstruction $X$ lies in $\mathcal S$.

\begin{prop}\label{prop:maxcont} Let $\bar H$ be one of the groups $M_{12}$, $M_{22}$, $HS$, $Ru$, $\PSL_2(r)$ for $r=13,19,27,29,37$, $\PSU_3(3)$ and $\PSU_3(8)$. Let $X$ be one of these groups or of those in Table \ref{t:e7stillleft}. If $\bar H$ is isomorphic to a subgroup of $X$ then the pair $(\bar H,X)$ appears in the table below:
\begin{center}
\begin{tabular}{cc}
\hline $\bar H$ & Possibilities for $X$
\\\hline $M_{12}$, $HS$, $Ru$, $\PSL_2(r)$, $r=19,27,37$, $\PSU_3(8)$ & None
\\  $M_{22}$ & $HS$
%\\ $\PSL_2(11)$ & $M_{12}$, $M_{22}$, $HS$
%\\ $\PSL_2(13)$ & $Ru$
\\ $\PSL_2(13),\PSL_2(29),\PSU_3(3)$ & $Ru$
\\ \hline
\end{tabular}
\end{center}\end{prop}
\begin{proof} The easiest way to prove this is to look through the Atlas \cite{atlas} for each $X$ checking if there exists $\bar H$ (or a subgroup with $\bar H$ as a subgroup) as a maximal subgroup, and then use transitivity. One finds the following:
\begin{itemize}
\item If $X\cong \Alt(6),M_{12},M_{22},\PSL_2(r),\PSU_3(3),\PSU_3(8)$ then there is no possible $\bar H$;
\item If $X\cong HS$ then $\bar H\cong M_{22}$;
\item If $X\cong Ru$ then $\bar H\cong \PSL_2(13),\PSL_2(29),\PSU_3(3)$.
\end{itemize}
To obtain the proposition one simply `inverts' this list.
\end{proof}

Thus we must exclude the options for $\bar H$ and $X$ above. If $\bar H\cong M_{22}$ then we have already dealt with the option $X\cong HS$ in Proposition \ref{prop:m22}. If $\bar H\cong\PSL_2(13)$, so $X\cong Ru$, then in fact the preimage of $\bar H$ in $2\cdot Ru$ is $2\times \PSL_2(13)$, and so this is not the Lie primitive example encountered in Section \ref{sec:psl213}. (One can also check that $\PSL_2(13)$ stabilizes a line on the $133$-dimensional module for $Ru$ by a character calculation.) If $\bar H\cong \PSL_2(29)$, so $X\cong Ru$, then $p=5$ and this was already in Proposition \ref{prop:psl229}. If $\bar H\cong \PSU_3(3)$ and $X\cong 2\cdot Ru$ then the restriction of $28$ to $\bar H$ is $7\oplus 21$, not irreducible. Thus it is not an obstruction to the copy of $\bar H$ we have claimed is maximal. This completes the proof that no element of $\mathcal S$ is an obstruction to $\bar H\in \mathcal S$ being maximal beyond the situation of $\PSL_2(29)<Ru$ already mentioned in Section \ref{sec:sl229}.

It remains to consider as possible obstructions the exotic $r$-local subgroups, and the members of $\mathscr X$. The only exotic $r$-local subgroup has connected component $D_4$, and this is contained in $A_7$, so we may ignore it, and so we need only consider members of $\mathscr X$. Thus assume that $\mb X\in \mathscr X$ contains $\bar H$. By Lemma \ref{lem:inconnected}, unless $\bar H\cong \PSU_3(3)$, we may assume that $\mb X$ is connected.

If $\mb X$ is an $E_6$-parabolic subgroup then $\mb X$ stabilizes a line on $M(E_7)$. Thus $\bar H$ cannot lie in $\mb X$ as we do not have such actions. If $\mb X$ is $F_4A_1$ then either some cover of $\bar H$ has a representation of dimension $2$ (which it does not) or $H\leq F_4\leq E_6$. Thus we may assume that the reductive part of $\mb X\in \mathscr X$ is a classical group. In particular, some cover of $\bar H$ must have a faithful representation of dimension at most $12$ (there is no $D_7$ subgroup), and at most $8$ if it is not self-dual.

If $\bar H$ is one of $M_{22}$, $HS$ and $Ru$, so $p=5$, then the minimal faithful dimension of any cover of $\bar H$ is greater than $12$, so $\bar H$ cannot lie in a member of $\mathscr X$. If $\bar H\cong \PSU_3(8)$, so $p$ is odd, then the minimal dimension is $56$, so this certainly does not lie in a member of $\mathscr X$. Similarly, if $\bar H$ is $\PSL_2(r)$ for $r=19,27,29,37$ then the minimal faithful dimension is $(r-1)/2$, but these simple modules, and those of dimension $(r+1)/2$, are not self-dual. The smallest self-dual module has dimension $r-1$, and so there is no such $\mb X$ for these groups either.

If $\bar H\cong M_{12}$ then $\bar H$ acts on $L(E_7)$ with factors of dimension $55$ and $78$. The only maximal positive-dimensional subgroups of $\mbG$ acting with two composition factors are $A_7$ and $A_2$, and neither of these acts with the correct dimensions, so $H$ is Lie primitive.

If $\bar H\cong \PSL_2(13)$ then $H=\SL_2(13)$ acts on $M(E_7)$ with composition factors all of dimension $14$. This forces $\mb X$ to have composition factors of dimension $14$, $28$ or $56$ on $M(E_7)$, yielding:
\begin{enumerate}
\item $G_2C_3$ (which was already considered in Section \ref{sec:psl213});
\item $A_1G_2$ (which means $H$ embeds in $G_2$, but this acts on $M(E_7)$ with factors of dimensions $14$ and $7$);
\item $A_2$ (into which $\bar H$ does not projectively embed);
\item $A_7$ (and $\SL_2(13)$ cannot embed into $\SL_8(q)/2$ with central involutions coinciding since the Schur multiplier of $\PSL_2(13)$ is $2$).
\end{enumerate}
\medskip

This leaves $H\cong\bar H\cong \PSU_3(3)$, so $p\neq 2,3$. In this case $H$ acts on $M(E_7)$ with two factors of dimension $28$, so if $H\leq \mb X^\circ$ then $\mb X^\circ$ either acts irreducibly or with two dual factors. Thus $H$ lies in $A_7$ or $A_2$, with the latter clearly impossible. If $H$ lies in $A_7$ though then $H$ stabilizes a line on $M(A_7)$ (as $H$ has only a $7$-dimensional simple module), so acts on $\Lambda^2(M(A_7))$ with factors of dimension $7$ and $21$, not $28$.

Thus $H\not\leq \mb X^\circ$, in which case $\mb X$ must be the normalizer of a maximal torus. However, in this case, $\mb X$ has an indecomposable summand of dimension at most $7$ on $L(E_7)$ (a subspace of the $0$-weight space) but $H$ has no such summand (as we can see from the supplementary materials, where $L(E_7){\downarrow_H}$ is constructed). Thus $H$ cannot lie in $\mb X$, and so $H$ is Lie primitive.

\section{Intrinsically imprimitive subgroups}
\label{sec:intrinsicimp}
In order to apply induction to results about Lie primitivity to classify subgroups of reductive groups, one needs to know, if a subgroup $H$ is Lie imprimitive, whether it is contained in a \emph{connected} subgroup, or just in any subgroup. Following \cite[Definition 4.9]{griessryba2002a}, call a Lie imprimitive subgroup of a reductive group $\mbG$ \emph{intrinsically imprimitive} if it is contained in a proper, closed, connected subgroup of $\mbG$. There are clearly Lie imprimitive subgroups that are not intrinsically imprimitive, for example a finite normalizer of a maximally split torus, $(q-1)^n\cdot W(\mbG)\leq G^\sigma\leq \mbG$ if $p\neq 0$ and $q\neq 2$. In this case it is certainly Lie imprimitive, since it is contained in $N_\mbG(\mb T)$, but there will be no \emph{connected} proper subgroup containing it.

If $H$ is taken to be a simple group, however, then this rules out this sort of example. The question is, is every simple subgroup of a reductive group $\mbG$ either Lie primitive or intrinsically imprimitive? For exceptional groups we provide an almost complete answer now, extending the known results in the literature.

To begin with though, we make the following observation. Let $H$ be a simple subgroup of $W(\mbG)$, where $\mbG$ has adjoint type for simplicity. Suppose that $T$ is the set of all elements of a maximal torus $\mb T$ of order dividing a given integer $n$ (e.g., $n=2$) and that there is a non-split extension $T\cdot \bar H$, where $\bar H$ has image $H$ in $W(\mbG)$. In this case, there is no \emph{split} extension $\mb T\rtimes \hat H$ (where $\hat H$ also has image $H$ in $W(\mbG)$, and therefore the class $H$ of $W(\mbG)$ cannot yield a subgroup of $\mbG$. This is clear: if there were such an extension, since $T\cdot \bar H$ exists, $H$ must normalize $T$, and so we may form the group $T\rtimes \hat H$. But this must contain the group $T\cdot \bar H$, a clear contradiction.

\begin{prop}\label{prop:maximaltorus} Let $H$ be a finite simple subgroup of an exceptional simple algebraic group $\mbG$ of adjoint type. If $H$ is contained in the normalizer of a maximal torus of $\mbG$ then $H$ is intrinsically imprimitive, except possibly if $p=2$, $\mbG=E_7$, $H=\PSL_2(8)$, and $M(E_7){\downarrow_H}$ acts as the open case left in Section \ref{sec:psl28},
\[ P(4_{1,2})\oplus P(4_{1,3})\oplus P(4_{2,3})\oplus 8.\]
\end{prop}
\begin{proof} Let $\mb T$ be a maximal torus of $\mbG$. By \cite[Lemma 3.10]{litterickmemoir} if $H\leq N_\mbG(\mb T)$ then $H$ is contained in a subsystem subgroup unless possibly $\mbG=E_6$ and $H\cong \PSU_4(2)$, or $\mbG=E_7$ and $H\cong \PSL_2(8),\PSU_3(3),\PSp_6(2)$. We deal with each of these in turn.

\medskip

\noindent \textbf{The case $\mbG=E_6$}: Suppose first that $\mbG=E_6$ and $H\cong \PSU_4(2)$. If $p=2$, by \cite[Proposition 10.5]{craven2019un}, $H$ stabilizes a line on $M(E_6)$. Thus $H$ is contained in a line stabilizer, and these are $F_4$, a $D_5T_1$-parabolic, or a subgroup of a $D_5T_1$-parabolic subgroup. Regardless of which, $H$ is contained in a connected subgroup, as needed. Thus we assume that $p$ is odd or zero. We simply now construct a non-split subgroup in the normalizer of a torus, using the observation just before this result. This is done in the supplementary materials, and therefore there is no subgroup $\PSU_4(2)$ of $N_\mbG(\mb T)$.

%Suppose that $H$ is a subgroup of $N_\mbG(\mb T)$, so that there is a group $L_1\cong 2^6\rtimes \PSU_4(2)$ of $N_\mbG(\mb T)$. We construct in the supplementary materials a \emph{non-split} extension $L_2=2^6\cdot \PSU_4(2)$ in $\mbG$ in characteristic $7$, extending the existence of this subgroup to all odd characteristics via Proposition \ref{prop:p=0impallp}. Since the $2^6$ subgroup must be toral by \cite[Table 1]{griess1991}, both of these subgroups must lie inside $N_\mbG(\mb T)$, which is clearly impossible since $W(E_6)=\PSU_4(2).2$, so both would have to lie in the subgroup of $N_\mbG(\mb T)$ of index $2$ containing $\mb T$.

\medskip

\noindent \textbf{The case $\mbG=E_7$, $H\cong \PSp_6(2)$}: The same argument as above works for $\mbG=E_7$ and $H=\PSp_6(2)$ if $p$ is odd or $0$. We have constructed a subgroup $2^6\cdot \PSp_6(2)$ in $N_\mbG(\mb T)$, proving that $\PSp_6(2)$ is not a subgroup for these primes.
%
%By \cite[Table 1]{griess1991} there are no non-toral subgroups $2^6$ with $\PSp_6(2)$ in its normalizer, so if we can construct a non-split subgroup $2^6\cdot \PSp_6(2)$ then we are again done. In the supplementary materials we do this.

If $p=2$ then $N_\mbG(\mb T)$ splits (consider the case $E_7(2)$ for example). We check that there is a unique class of complements to $\mb T$ in $N_\mbG(\mb T)$. To check this we check via computer (in the supplementary materials) that $\PSp_6(2)$ has zero $1$-cohomology on the $7$-dimensional module for its action on $\mb T$ in characteristics $3$, $5$ and $7$. (Since $|\PSp_6(2)|$ is divisible by primes $2,3,5,7$, and $p=2$, these are the only cases that need to be checked.) Thus $H$ is unique up to conjugacy in $N_\mbG(\mb T)$ for $p=2$. But then $W(E_7)=2\times \Sp_6(2)$, so a copy of $H$ in $N_\mbG(\mb T)$ centralizes an involution. Thus $H$ lies in a maximal parabolic subgroup of $\mbG$, and the result holds.

\medskip

\noindent \textbf{The case $\mbG=E_7$, $H\cong \PSU_3(3)\cong G_2(2)'$}: There is clearly a unique subgroup of $\Sp_6(2)$ isomorphic to $G_2(2)'$ (although we check this in the supplementary materials). We claim that if $H$ centralizes a line on either $M(E_7)$ or $L(E_7)^\circ$ then it lies in a proper connected subgroup, as needed. The line stabilizers on $M(E_7)$ are given in Proposition \ref{prop:1spacestabs} and so this is clear. If $H$ centralizes a $1$-space $W$ on $L(E_7)^\circ$ then it lies in a parabolic subgroup or a maximal-rank subgroup by Lemma \ref{lem:paramaxl}. In the former case we are done, and in the latter case we let $\mb X$ denote the centralizer of $W$, which is positive-dimensional. If $\mb X$ is contained in any maximal-rank subgroup other than the normalizer of a maximal torus then we are done, so $\mb X^0$ must be toral, and $\mb X/\mb X^0$ must contain $\PSU_3(3)$. Certainly $H$ cannot centralize $\mb X^0$ for then it would lie in the centralizer of an involution or an element of order $3$ in $\mb X^0$ (depending on $p$), so $H$ acts on $\mb X^0$. Thus $\mb X^0$ has rank at least $7$ (as this is the dimension of a minimal representation of $\mb X$ in arbitrary characteristic), and so $\mb X^0$ is a maximal torus $\mb T$.

But this means that $\mb X^0$ centralizes $W$, and maximal tori in $E_7$ centralize a $7$-space, and $\PSU_3(3)\leq W(E_7)$ acts non-trivially on this $7$-space. This contradiction means that $\mb X^0$ in fact is not $\mb T$, and so $H$ is intrinsically imprimitive.

This means that $H$ is intrinsically imprimitive if $p=2$ by \cite[Proposition 9.4]{craven2019un}. If $p=7$ then $H$ is intrinsically imprimitive, unless $H$ is in one of the three rows labelled `\textbf{P}' in \cite[Table 6.220]{litterickmemoir}. Row 4 is eliminated in Section \ref{sec:psu33}, and in that section the other two rows were shown to yield a unique $\mbG$-class of subgroups. The former lies in $A_6$ and the latter was proved to be Lie primitive, and hence not inside $N_\mbG(\mb T)$. If $p\neq 2,3,7$ then a similar argument works. In this case $H$ is intrinsically imprimitive, unless $H$ is in one of the two rows labelled `\textbf{P}' in \cite[Table 6.219]{litterickmemoir}. Row 1 was proved to be unique and Lie primitive, and Row 2 was proved to stabilize a $\mathfrak g_2$ subalgebra of $L(E_7)$ in Section \ref{sec:psu33}. Thus this case cannot only lie in $N_\mbG(\mb T)$ either.

It remains to consider $p=3$. Here we claim that there is a unique conjugacy class of complements to $\mb T$ in $\gen{\mb T,H}$. To see this, notice that we need only consider the $1$-cohomology of $H$ on the Sylow $p$-subgroups of $\mb T$ for $p=2,7$. We have already noted that it is zero for $p=7$. If $p=2$ then the composition factors of the $\F_2H$-module consisting of the involutions and identity of $\mb T$ are $1$ and $6$. These either assemble to have a fixed point (i.e., $H$ centralizes an involution in $\mbG$) or they form a module $1/6$, which has zero $1$-cohomology. The former immediately leads to $H$ being intrinsically imprimitive, but we rule it out nevertheless with a computer calculation in the supplementary materials.

Since there is zero $1$-cohomology we obtain that all complements are conjugate. Finally, the central involution from the Weyl group commutes with $H$, so $H$ does centralize an involution, and is intrinsically imprimitive.

\medskip

\noindent \textbf{The case $\mbG=E_7$, $H\cong \PSL_2(8)$}: The non-split extension $2^6\cdot\PSp_6(2)$ restricts to a non-split extension $2^6\cdot \PSL_2(8)$, so this shows that there is no subgroup $H$ of $N_\mbG(\mb T)$ for $p$ odd.

If $p=2$ then there are multiple conjugacy classes of subgroups $H$ inside $N_\mbG(\mb T)$. We show in the supplementary materials that there are two conjugacy classes of subgroups $H$ in $N_G(T)$ for $G=E_7(4)$ and $T$ a maximally split torus. The subgroup $H\leq \PSp_6(2)$ is of course intrinsically imprimitive, but the other subgroup is not. It acts on $M(E_7)$ as the open case given in Section \ref{sec:psl28}, and need not lie in another positive-dimensional subgroup.
\end{proof}

There is a conjugacy class of subgroups $\PSL_2(8)$ in $A_1D_6$ which might be conjugate to the one in the normalizer of the torus, so we cannot say more at this time.

There are other subgroups than the normalizer of a torus. The maximal ones with insoluble component group are $A_1^7\cdot \PSL_3(2)$ in $E_7$ and $A_1^8\cdot \AGL_3(2)$ in $E_8$. These are not quite wreath products, but are sandwiched between $\SL_2(k)\wr K$ and $\PGL_2(k)\wr K$. The conjugacy classes of complements of wreath products $L\wr K$ (with $K$ transitive) are well-understood (see \cite{houghton1975}, and especially \cite[p.\ 201]{hassanabadi1978}) and are in bijection with conjugacy classes of homomorphisms from the point stabilizer of $K$ to $L$. If $K$ acts on $n$ points and the image of a given homomorphism is $\bar L$ then the corresponding subgroup is contained in the group $\bar L^n\cdot K$.

The point stabilizer of $\PSL_3(2)$ acting on seven points is $\Sym(4)$, and acting on eight points is the Frobenius group $7\rtimes 3$, so we need to understand homomorphisms from these groups into $\PGL_2(k)$.

We first deal with the subgroup of $E_8$.

\begin{prop}\label{prop:E8A1} If $H\cong \PSL_3(2)$ is contained in the subgroup $A_1^8\cdot \AGL_3(2)$ of $E_8$ then $H$ is intrinsically imprimitive.
\end{prop}
\begin{proof} The subgroup $\mb X=A_1^8\cdot \AGL_3(2)$ can be seen from \cite[Table 5.1]{liebecksaxlseitz1992} to have structure $\SL_2(k)\wr \AGL_3(2)$ in characteristic $2$, and $2^4\cdot (\PGL_2(k)\wr \AGL_3(2))$ if $p$ is odd. The case $p=2$ is clearly less complicated, and we do this now.

Inside $\AGL_3(2)$ there are two conjugacy classes of subgroups $\PSL_3(2)$: one fixes a point and the other is transitive on eight points. Inside $\mb X$, we see that any subgroup $H$ of $\mb X$ with image the $\PSL_3(2)$ fixing a point must centralize an $A_1$ subgroup, and actually lie in $A_1E_7$ (and inside the subgroup $A_1\circ (A_1^7\cdot \PSL_3(2))$). Thus such a subgroup is intrinsically imprimitive, and we may assume that $H$ acts transitively on the eight $A_1$s.

In this case there are two classes of homomorphisms from $7\rtimes 3$ to $\SL_2(k)$, with images $1$ and $3$. Thus we find all classes of subgroups $H$ in the group $3^8\rtimes \PSL_3(2)$. This is clearly contained in the normalizer of a maximal torus, and so $H$ is intrinsically imprimitive by Proposition \ref{prop:maximaltorus}.

If $p$ is odd then we first need to enumerate the complements $\bar H$ in $\PGL_2(k)^8\rtimes \PSL_3(2)$, and then count the number of subgroups $H$ in $\mb X$ with image $\bar H$ modulo the central subgroup of order $2^4$. Again, there are two classes of homomorphisms $7\rtimes 3\to\PGL_2(k)$, with images $1$ and $3$, except if $p=7$, where there is a subgroup $7\rtimes 3$ as well. Hence all subgroups are in a group $2^4\cdot 3^8\cdot \PSL_3(2)$, or in the extra case $2^4\cdot 7^8\cdot 3^8\cdot \PSL_3(2)$. The $2^4$ subgroup is central, and so the groups $2^4\cdot 3^8$ and $2^4\cdot 7^8$ are actually direct products. Hence our subgroup $H$ normalizes a subgroup $3^8$ or a subgroup $7^8$. If we are in the first case and $p=3$, or we are in the second case (so $p=7$) then this places $H$ inside a parabolic subgroup, hence is intrinsically imprimitive, and if $p\neq 3$ in the first case then $H$ lies inside the normalizer of a maximal torus again, so again intrinsically imprimitive.
\end{proof}

If we try to follow this method for the subgroup of $E_7$ then it works, and we can classify the complements, but we cannot yet prove the same result though. This corrects the result of \cite[Lemma 1.12]{craven2017}.

\begin{prop} If $H\cong \PSL_3(2)$ is contained in the subgroup $A_1^7\cdot \PSL_3(2)$ of $E_7$ then $H$ is intrinsically imprimitive, or $p\neq 2$, the preimage of $H$ in simply connected $E_7$ is $\SL_2(7)$ and possibly $H$ is not intrinsically imprimitive.
\end{prop}
\begin{proof} We follow the proof of Proposition \ref{prop:E8A1}. The subgroup $\mb X=A_1^7\cdot \PSL_3(2)$ can be seen from \cite[Table 5.1]{liebecksaxlseitz1992} to have structure $\SL_2(k)\wr \PSL_3(2)$ in characteristic $2$, and $2^3\cdot (\PGL_2(k)\wr \PSL_3(2))$ if $p$ is odd (in the adjoint group). We start with $p=2$, as before.

The point stabilizer of $\PSL_3(2)$ acting on seven points is $\Sym(4)$. There are three classes of homomorphism $\Sym(4)\to \SL_2(k)$ for $p=2$, with images $1$, $C_2$ and $\Sym(3)$, yielding three complements, $H_1$, $H_2$ and $H_3$. The group $H_2$ lies in a group $2^7\rtimes \PSL_3(2)$, and so $H_2$ normalizes a $p$-subgroup of $\mbG$. This places $H_2$ inside a maximal parabolic subgroup of $\mbG$, and hence $H_2$ is intrinsically imprimitive. In the first case, $H_1$ normalizes a maximal torus $\mb T$ of the subgroup $A_1^7$, and so $H_1\leq N_\mbG(\mb T)$, and is intrinsically imprimitive by Proposition \ref{prop:maximaltorus}. In the third case, $H_3$ lies in a group $\Sym(3)\wr \PSL_2(7)$, so $H_3$ normalizes a subgroup $3^7$ of $\mbG$. This again places $H_3$ inside $N_\mbG(\mb T)$ for some torus $\mb T$, so is again intrinsically imprimitive.

We now turn to $p$ odd. There are four classes of homomorphisms $\phi:\Sym(4)\to\PGL_2(k)$, with images $1$, $2$, $\Sym(3)$ and $\Sym(4)$, yielding four classes of complements modulo the central $2^3$. These yield subgroups $2^3\cdot \PSL_3(2)$ of $\mb X$ which might or might not be split extensions. To determine which, it seems easiest to just check on a computer.

Doing so in the supplementary materials, we find that two of them yield split extensions (with images $1$ and $S_4$) and two do not. One of these comes from the normalizer of a torus and so is intrinsically imprimitive. The other is proved in the supplementary materials to lift to $\SL_2(7)$ in the simply connected group (as opposed to the class in $N_\mbG(\mb T)$, which lifts to $2\times \PSL_2(7)$).
\end{proof}

We cannot determine whether this class is intrinsically imprimitive because we do not know how many classes of such subgroups there are in $E_7$. In characteristic not $2,3,7$, the action of the $\SL_2(7)$s on $M(E_7)$ is the unresolved case in Section \ref{sec:psl27} (confirmed in the supplementary materials), and such an action does exist inside connected subgroups as well. So at the moment it is not clear whether these subgroups are intrinsically imprimitive or not. (It \emph{is} true that the subgroup $2^3\rtimes \SL_2(7)$ is not intrinsically imprimitive, but this is not enough for us.)

\begin{cor}\label{cor:intrinsicimp} Let $H$ be a quasisimple, Lie imprimitive subgroup of an exceptional algebraic group $\mbG$ of simply connected type, with $Z(H)\leq Z(\mbG)$. Either $H$ is intrinsically imprimitive or one of the following occurs:
\begin{enumerate}
\item $p=2$, $\mbG=E_7$, $H\cong \PSL_2(8)$;
\item $p\neq 2$, $\mbG=E_7$, $H\cong \SL_2(7)$.
\end{enumerate}
\end{cor}

As we have stated above, neither of the remaining cases is known not to be intrinsically imprimitive.

\appendix
\section{Detailed descriptions of the methods}

As we have mentioned before, the subalgebra method is \emph{not} an algorithm. Every time it is used is different, depending on the specific group, and even the specific embedding in $\Sp_{56}(k)$.

All of the files start similarly. First the group is defined, then $56$- and $133$-dimensional modules for it called \texttt{MG} and \texttt{LG} respectively. The tensor square \texttt{MG2} of \texttt{MG} is constructed and inside that the symmetric square \texttt{SMG2} is defined. A function \texttt{LieProd(v,w)} takes two vectors in \texttt{MG2} and returns a third, but crucially parameters are allowed in this function, so one may take linear combinations of basis elements whose coefficients are parameters. The set $\Hom_{kH}(L(E_7),S^2(M(E_7)))$ is constructed, the sum of the images of these maps being defined as \texttt{U}, and the Hom-space is split according to the summands of $L(E_7){\downarrow_H}$. Individual (simple) submodules of $L(E_7){\downarrow_H}$ are defined, called $W_1,\dots,W_r$, and $\Hom_{kH}(W_i,U)$ is constructed. A basis of \texttt{MG2} is constructed by first placing the copies of the $W_i$, then the rest of \texttt{U}, then the rest of \texttt{MG2}.

We then use $\Hom_{kH}(W_1,U)$ to define a point \texttt{a} in a generic copy of $W_1$ as follows: if the dimension of $\Hom_{kH}(W_1,U)$ is \texttt{dh1}, take variables \texttt{R.1},\dots,\texttt{R.dh1} from a polynomial ring \texttt{R}. Given $v\in W_1$, form a linear combination of \texttt{R.i} times the image of $v$ under \texttt{homs1.i}, where \texttt{homs1} is $\Hom_{kH}(W_1,U)$. Letting \texttt{b} denote another such point, \texttt{LieProd(a,b)} constructs the Lie product of \texttt{a} and \texttt{b}, with the coordinates being quadratics in the \texttt{R.i}. Rewriting this `vector' in terms of our new basis allows us to read off the coordinates of the product that lie outside of \texttt{U}.

The broad outline is to define some points of elements of the $W_i$, take the Lie products of these, force the products to lie in \texttt{U}, and solve the resulting equations. Solutions that yield $133$-dimensional subalgebras of type $\mathfrak e_7$ are kept, and ones that provably cannot yield such subalgebras are discarded.

\paragraph{$2\cdot \Alt(7)$, $p=5$} Let $W_1$ denote the $8$, $W_2$ and $W_3$ the $10$ and $10^*$, and $W_4$ one of the $35$-dimensional submodules. The module $U$ is $1052$-dimensional. Let $a,b,c$ lie in $W_1$, $d,e$ in $W_2$, $f,g$ in $W_3$ and $p,q,r$ in $W_4$. We compute 
\[ [a,b],\;[[a,b],c],\;[d,e],\;[a,d],\;[a,g],\;[d,g],\;[e,g],\;[[a,d],g],\;[p,q],\;[[p,q],r],\;[a,q],\;[d,q],\]
and require each of them to lie in $U$. This yields many equations in the parameters governing $W_1$, $W_2$, $W_3$, $W_4$. Since there are three copies of $35$ in $L(E_7){\downarrow_H}$, we must have at least three free parameters for the location of $W_4$ at all times.

The fact that $L(E_7)$ is a summand of $L(C_{28}){\downarrow_{E_7}}$ and $8$ is a summand of the $H$-action on $L(E_7)$, the $8$ is a summand of the $H$-action on $L(C_{28})$. This already proves that a specific one of the four parameters describing $W_1$ is non-zero, so set it to be $1$. Then we force two others to be $0$, and the parameters describing $W_1$ are $(0,1,0,r_4)$. The centralizer of $H$ in $\Sp_{56}(k)$ scales $r_4$, so without loss of generality $r_4=1$ or $r_4=0$.

We then prove that the former of these yields too few possibilities for $W_4$, so we must have $r_4=0$. This fixes $W_1$. We now eliminate a few variables for $W_2$ until there are two remaining. The centralizer again acts to yield three orbits, and we construct the subalgebras generated by these three submodules, $L_1$, $L_2$ and $L_3$. One of these has dimension $133$, and the other two have a $35$-dimensional abelian subalgebra, contradicting Proposition \ref{prop:abeliansubalgebra}. Thus $H$ is unique up to $\mbG$-conjugacy, and we even find an element of order $3$ centralizing $H$ and stabilizing the $E_7$-subalgebra, suggesting $H$ lay in $A_2A_5$.

\paragraph{$\Alt(7)$, $p=5$} Let $W_1$ denote the $8$-dimensional summand, $W_2$ and $W_3$ the $10$ and $10^*$. The module $U$ is $863$-dimensional. The aim is to show that $H$ always stabilizes a $28$-dimensional simple subalgebra of $M(C_{28})$. The centralizer of $H$ in $\Sp_{56}(k)$ is fairly complicated, and we use a three $1$-dimensional subgroups (one toral, two unipotent) of the centralizer, which is enough to understand the situation.

The space $\Hom_{kH}(W_1,U)$ has dimension $6$, so there are parameters $r_1,\dots,r_6$. The centralizing elements mentioned above move these parameters, allowing us to specify more than one at a time. We first prove that $r_3\neq 0$ yields a $\mathfrak d_4$-subalgebra. If $r_3=0$ and $r_1\neq 0$ then we end up with a $1$-parameter family of $28$-dimensional subalgebras, even after applying our centralizing elements, and if $r_3=r_1=0$ then we derive a contradiction.

This is the first place where we use the possible Jordan normal forms of nilpotent elements of $\mathfrak e_7$ on $M(E_7)$ and $M(E_7)\otimes M(E_7)$. If we obtain a nilpotent element of $\mathfrak c_{28}$ that has the wrong Jordan normal form on $M(E_7)\otimes M(E_7)$ then it be an element of an $\mathfrak e_7$-subalgebra.

\paragraph{$\PSU_3(3)$, $p=7$} We start with the case where $H$ acts on $M(E_7)$ with factors of dimensions $7$ and $21$. Let $W_1$ be the $7$, $W_2$ be the $14$, $W_3$ be the $21$ and $W_4$ be the $26$ in the socle of $L(E_7){\downarrow_H}$. We first calculate lots of Lie products between basis elements of these, too many to reproduce here.

The centralizer in $\Sp_{56}(k)$ of $H$ is a $2$-dimensional torus, and this acts on the five parameters $r_1,\dots,r_5$ governing $W_1$ and the four governing $W_2$. Since $\Lambda^2(W_1)\cong W_1\oplus W_2$, either $[W_1,W_1]$ has image containing $W_1$ or it does not. The former case yields a contradiction fairly quickly, but the latter is more difficult. It yields $r_3=0$, and the $2$-torus allows us to independently scale $r_1$, $r_2$ and $r_4$. Thus there are eight options for these, as each could be $0$ or $1$. If $r_1=0$ then we may scale $r_5$ independently of the others, and altogether there are ten possibilities for $W_1$. Three of these yield inconsistent equations for the other parameters for the $W_i$, and the others yield a $7$-dimensional abelian subalgebra of $L(A_{55})$ consisting of nilpotent elements. The Jordan normal form of an element of these is not consistent with coming from $E_7$ in all but one case. This last case yields a $133$-dimensional subalgebra once the other $W_i$, which are uniquely determined, are included in the generating subspace. This is $L(E_7)$.

\medskip Now suppose that $H$ acts on $M(E_7)$ as $28\oplus 28^*$. Let $W_1$ be $21$ and $W_2$ be the $26$ in the socle of $L(E_7){\downarrow_H}$. Let $a,b\in W_1$, $c,d\in W_2$, and compute $[a,b]$, $[c,d]$ and $[a,c]$. There is one specific parameter, $r_{13}$, that we use, which labels a copy of $W_2$. If $r_{13}=0$ then $[a,b]=[c,d][a,c]=0$, actually, and this is enough to prove that $W_1+W_2$ is a $47$-dimensional abelian subalgebra, contradicting Proposition \ref{prop:abeliansubalgebra}. Thus $r_{13}=1$ without loss of generality, and this is enough to determine the set of solutions. They fall into two orbits under the centralizer, swapped by a normalizing element, and the field automorphism centralizes the two orbits.

\paragraph{$\PSU_3(3)$, $p\neq 7$} These all follow the same plan as the $28\oplus 28^*$ case for $p=7$, but with very minor tweaks.

\paragraph{$\PSL_2(37)$, $p=3$} Let $W_1$ be the $19$-dimensional submodule, and let $a,b\in W_1$. Note that $\Hom(W_1,U)$ is $2$-dimensional. The Lie product $[a,b]$, when forced to lie inside $U$, yields a single equation $r_1^2-r_2^2=0$, so $r_1=\pm r_2$, and this yields two $133$-dimensional subalgebras. There is a non-central involution in $\Sp_{56}(k)$ centralizing $H$ and swapping the two subalgebras. Thus $H$ is unique up to conjugacy.

\paragraph{$\PSL_2(29)$, $p\neq 2,3,5$} Let $W_1$ denote the $15$-dimensional submodule, and let $a,b,c$ be distinct points from $W_1$. The product $[[a,b],c]$ yields a few cubic equations in the three parameters $r_1,r_2,r_3$ governing $W_1$. If $r_1=0$ then $r_2=r_3=0$, so without loss of generality $r_1=1$. This forces the other equations to yield two solutions, which are swapped by the centralizer of $H$ in $\Sp_{56}(k)$. Thus $H$ is unique up to conjugacy.

\paragraph{$\PSL_2(29)$, $p=2$} We cannot use the $15$-dimensional submodule as in the previous cases because it decomposes and we prefer simple submodules. In this case let $W_1$ denote the $28$-dimensional submodule, and let $a,b$ be distinct points from $W_1$. The product $[a,b]$ yields a few quadratic equations in the four parameters governing $W_1$, and these factorize to yield two solutions. These yield $133$-dimensional subalgebras that are non-isomorphic, and are swapped by an element that normalizes $H$. Thus $H$ is unique up to conjugacy (with a given module action on $L(E_7)$).

\paragraph{$\PSL_2(29)$, $p=3,5$} Let $W_1$ denote the $15$-dimensional submodule, and let $W_2$ denote the $28$-dimensional submodule. Let $a,b,c$ be distinct points from $W_1$ and $d,e$ denote distinct points of $W_2$. We compute $[[a,b],c]$, $[a,d]$ and $[d,e]$ to give us equations in the three parameters for $W_1$ and four for $W_2$. Just $[[a,b],c]$ lying in $U$ gives us two options for $W_1$. One generates a $133$-dimensional $\mathfrak e_7$-subalgebra, the other yields an abelian $15$-space. This latter space, when combined with the equations for $W_2$, yields a unique option for $W_2$. The sum $W_1+W_2$ is a $43$-dimensional abelian subalgebra of $L(A_{55})$, and so cannot be a subalgebra of $\mathfrak e_7$ by Proposition \ref{prop:abeliansubalgebra}.

\paragraph{$\PSL_2(27)$, $p\neq 13$} Let $W_1$ denote the $28$-dimensional submodule, and let $a,b$ be distinct points from $W_1$. The product $[a,b]$ yields a few quadratic equations in the six parameters $r_1,\dots,r_6$ governing $W_1$. It is very easy to find the four solutions. A non-central involution of $\Sp_{56}(k)$ centralizing $H$ swaps these in pairs, and an element normalizing $H$ makes the action transitive.

\paragraph{$\PSL_2(27)$, $p=13$} Let $W_1$ denote a $26$-dimensional summand and $W_2$ denote the $27$-dimensional submodule in the socle of $L(E_7){\downarrow_H}$. Letting $a,b\in W_1$ and $d,e\in W_2$, we compute $[a,b]$, $[d,e]$ and $[a,d]$. Since $\Lambda^2(W_1)$ has a summand $W_1$, either $[a,b]$ projects non-trivially onto $W_1$ or it does not. In the former case we quickly arrive at a unique solution for $W_1$, and by taking the subalgebra generated by $W_1$, all of $L(E_7)$. If $[a,b]$ misses $W_1$ then we choose a single variable, $r_5$, which is either $0$ or $1$. In both cases we easily find that $W_1\oplus W_2$ is an abelian subalgebra of dimension $53$, contradicting Proposition \ref{prop:abeliansubalgebra}.

\paragraph{$\PSL_2(25)$, $p=5$} Let $W_1$ be the $3$-dimensional summand and $W_2$ be the other $3$-dimensional submodule. The modules $W_3$ and $W_4$ are the two $15$-dimensional submodules. We compute many Lie products in this case. The dimensions of Hom-spaces are $3$, $3$, $4$ and $4$ for the $W_i$ respectively. We start by showing that $r_{10}$, the last variable parametrizing $W_3$, cannot be $0$, by showing that this yields a subalgebra whose socle has too large a dimension as a $kH$-module. Thus $r_{10}=1$ by scaling. This yields a quadratic for $r_9$, one root of which easily yields a $133$-dimensional subalgebra of type $\mathfrak e_7$, so we have to exclude the other root, $r_9=0$. Next, if $r_7=r_{11}=0$ then we find a $30$-dimensional abelian subalgebra $W_3\oplus W_4$, so this cannot occur, and so at least one of $r_7$ and $r_{11}$ is non-zero. Finally, we show that $r_7=0$, so $r_{11}=1$, and this yields a $35$-dimensional abelian subalgebra. Using Proposition \ref{prop:abeliansubalgebra} we therefore conclude that $H$ is unique up to $\mbG$-conjugacy.

\paragraph{$\PSL_2(19)$, $p=0$} Let $W_1$ be the $18$-dimensional submodule that appears with multiplicity $1$, and $W_2$ be one of the $20$-dimensional submodules. We let $a,b,c$ be distinct points from $W_1$ and $d,e,f$ be distinct points from $W_2$. We construct Lie products $[[a,b],c]$, $[[d,e],f]$, $[a,d]$, $[[a,b],d]$ and $[[a,d],e]$, and use these to obtain a number of quadratics and cubic in the $21$ variables $r_1,\dots,r_{21}$. If $r_{21}=0$ then all parameters for $W_2$ are $0$, which is not allowed. Hence $r_{21}=1$ without loss of generality, and we immediately obtain eight solutions. These are permuted transitively by the normalizer of $H$ in $\Sp_{56}(k)$, and $H$ is unique up to $\mbG$-conjugacy.

\paragraph{$\PSL_2(19)$, $p=5$} Let $W_1$ be an $18$-dimensional submodule not in the radical of the module, and $W_2$ be $18$-dimensional submodule that is in the radical. We let $a,b,c$ be distinct points from $W_1$ and $d,e,f$ be distinct points from $W_2$. We construct Lie products $[[a,b],c]$, $[[d,e],f]$, $[a,d]$, $[[a,b],d]$ and $[[a,d],e]$, and use these to obtain a number of quadratics and cubic in the $23$ variables $r_1,\dots,r_{23}$. Since $W_1$ is parametrized by 12 elements and $W_2$ by 11, we see that there is a single parameter corresponding to a homomorphism $W_1\to U$ whose image is \emph{not} in the radical of \texttt{U}. This parameter must be non-zero, hence set to be 1. This immediately yields four options for the parameters for $W_2$, and the normalizer acts transitively on the solutions. Hence $H$ is unique up to $\mbG$-conjugacy.

\paragraph{$\PSL_2(19)$, $p=2$} Here the form is not unique up to conjugacy in $\Sp_{56}(k)$, so we have two options: work inside $L(A_{55})$, in which case we compute homomorphisms into the module \texttt{MG2} rather than \texttt{SMG2} described above; work inside the symplectic group again, but allow the form to have a parameter. In the supplementary materials we prove the uniqueness of $H$ via both methods.

In both cases let $W_1$ be the $18$-dimensional submodule that appears with multiplicity $1$, and $W_2$ be an $18$-dimensional module that appears with multiplicity $2$. If we are working in $L(A_{55})$ then there are 16 parameters governing the location of $W_1$, otherwise there are ten. In both cases, setting the last of these to $1$ immediately yields the $\mathfrak e_7$-subalgebra up to the action of the centralizer in the ambient algebraic group (either $A_{55}$ or $C_{28}$). Setting this equal to $0$ eventually leads to $W_1\oplus W_2$ being abelian, but here we require the ambiguity of $W_2$ to allow us to set one of the parameters labelling a generic subspace $W_2$ to be $0$. The main step here to help us reduce things is to note that $\Lambda^2(W_1)$ has two copies of $W_1$ in it, and so $[W_1,W_1]$ might have $W_1$ in its image. This is, in fact inconsistent with the last parameter governing $W_1$ being set to $0$, and so $[W_1,W_1]$ has no image on $W_1$. This show actually that $W_1$ is abelian. We bring in $W_2$ and it is not difficult to prove that $[W_1,W_2]=0$, and then $[W_2,W_2]=0$.

\paragraph{$\PGL_2(19)$, $p\neq 5$} Let $W_1$ be the $18$-dimensional submodule that appears with multiplicity $1$, and let $W_2$ denote the $19$-dimensional submodule. We let $a,b,c$ be distinct points from $W_1$ and $d,e,f$ be distinct points from $W_2$. We construct Lie products $[[a,b],c]$, $[[d,e],f]$, $[a,d]$, $[[a,b],d]$ and $[[a,d],e]$, and use these to obtain a number of quadratics and cubic in the 22 variables (20 if $p=3$). Assuming $r_5$ and $r_{10}$ are both non-zero, we scale them to be $1$ using the centralizer and quickly obtain the unique solution.

Alternatively, $r_5r_{10}=0$, and we show that this implies that $W_1\oplus W_2$, or a different subspace of dimension greater than $27$, is abelian. Note that $\Lambda^2(W_1)$ has an image $W_1$ and $\Lambda^2(W_2)$ has an image $W_2$. Thus we can force $[a,b]$ to either project onto $W_1$ (which yields a contradiction to $r_5r_{10}=0$) or not, and the same with $[d,e]$ and $W_2$. This yields exactly two options for the subspace $W_1$, both of which yield a $36$- or $37$-dimensional abelian subalgebra of $L(E_7)$. Thus $H$ is unique up to conjugacy in $\mbG$.

\paragraph{$\PGL_2(19)$, $p=5$} Let $W_1$ be an $18$-dimensional submodule not in the radical of the module, and $W_2$ be $18$-dimensional submodule that is in the radical. We let $a,b,c$ be distinct points from $W_1$ and $d,e,f$ be distinct points from $W_2$. We construct Lie products $[[a,b],c]$, $[[d,e],f]$, $[a,d]$, $[[a,b],d]$ and $[[a,d],e]$, and use these to obtain a number of quadratics and cubic in the $23$ variables $r_1,\dots,r_{23}$. Just imposing the conditions that these Lie products lie inside \texttt{U} is enough to determine a unique solution.

%\paragraph{$\PSL_2(13)$, $p=2$} Let $W_1$, $W_2$ and $W_3$ denote the three $12$-dimensional simple submodules. Let $a,b,c$ lie in $W_1$, $d,e,f$ lie in $W_2$ and $g,h$ lie in $W_3$. If the Lie product $[a,b]$ hits $W_1$ then one quickly finds an $E_7$-Lie subalgebra, so we assume it does not hit $W_1$. Some case distinctions and reduction, and eventually we can prove that there are no solutions for the 16 variables for the $W_1$ subalgebra and the others. Note that here the bilinear form is not unique so there is also a parameter in the form.

\paragraph{$\PSL_2(13)$, $p\neq 7,13$} For most of the possibilities, $L(E_7){}\downarrow_H$ contains a sum of three distinct $12$-dimensional modules. The exception is $p=3$ and the situation from Proposition \ref{prop:sl213primitive}. We consider these first.

Let $W_1$, $W_2$ and $W_3$ denote the three $12$-dimensional simple submodules. Let $a,b,c$ lie in $W_1$. From the structure of $\Lambda^2(W_1)$, the Lie product $[a,b]$ can hit $W_1$ and $W_2$ (or $W_3$, depending on labelling). We can derive a contradiction in the case where $[a,b]$ misses $W_1$ entirely. Usually this is fairly easy, but it just involves deductions using polynomials in the standard way.

If $[a,b]$ meets $W_1$ non-trivially then it is fairly easy to find an $E_7$ Lie subalgebra in this case, and standard case distinctions with $r_i=0$ or $r_i\neq 0$ for various variables is enough to prove a unique class of Lie subalgebras under the action of the normalizer of $H$.

\paragraph{$\PSL_2(13)$, $p=7$} The case $p=7$ deviates substantially. This is because the $12_i$ are all isomorphic when reducing modulo $7$, and we actually have a summand $P(12)$ in the Lie algebra $L(E_7){}\downarrow_H$. Here there is a map $\Lambda^2(12)\to 12/12$ which we use. We construct a point in $\Lambda^2(12)$ whose image under this homomorphism lies in the submodule $12$. Let $W_1$ be this $12$ and let $a,b$ lie in $W_1$.

As with the case $p\neq 7,13$, if $[a,b]$ misses $P(12)$ entirely then we can derive a contradiction by solving the system of quadratics and cubics, as usual. If $[a,b]$ meets $P(12)$ then there are two cases: $[W_1,W_1]\cap P(W_1)=W_1$ and $[W_1,W_1]\cap P(W_1)$ has dimension $24$, and is $12/12$. The first case again yields a contradiction as before. If the intersection has dimension $24$ then we use the point in $\Lambda^2(12)$ mentioned above to find a commutator $[u,v]$ for $u,v\in W_1$ whose image lies in $W_1$ itself (and some point outside of $P(W_1)$, which can be ignored). This gives us an equation because the coefficients $r_1,\dots,r_{16}$ that determine the copy of $W_1$ appear in both $u$ and $[u,v]$. Thus we obtain equations of the form $r_i=f(r_1,\dots,r_{16})$, where $f$ is homogeneous of degree $2$. These can finally be solved to obtain the unique orbit of $E_7$ Lie subalgebras.

\paragraph{$\PSL_2(11)$, $p=3$} This is by far the most long-winded of the proofs we give here. We start with $W_1=5$, $W_2=5^*$, $W_3=12_1$, $W_4=12_2$ and $W_5=10$, and let $a,b,c\in W_1$, and $d,e,f\in W_2$. The tensor product $5\otimes 5^*$ decomposes as $12_1\oplus 12_2\oplus 1$, so the Lie product $[W_1,W_2]$ has a (possibly zero) projection onto $W_3$ and $W_4$. This allows us to split up $[a,d]$ into its $W_3$- and $W_4$-components, and take the product of each of these with $b$ and $e$. These have maps onto $5$ and $5^*$, which are either $0$ or the exact subspaces $W_1$ and $W_2$ (as they appear with multiplicity $1$). Many, many calculations are needed with the resulting equations to reach a conclusion.

\section{An alternative proof of Korhonen's result on the exotic local subgroup}
\label{sec:altkorhonenproof}

In this short appendix we sketch an alternative proof of the main result of \cite{korhonen2025}, namely that if $q\equiv \pm 1\bmod 8$ then the intersection of the subgroup $(D_4 \times 2^2).\Sym(3)$ with the simple group $E_7(q)$ contains the full $\Sym(3)$ subgroup, and if $q\equiv \pm 3\bmod 8$ then only the $3$ is contained in the simple group. This section, being logically independent of the rest of the paper, will use $G$ for a different group. As usual though, let $E_7(q)$ denote the simple group.

\medskip

First let $q\equiv 1\bmod 4$. Let $H$ be the group $2\cdot E_7(q)$ and $G$ be the group $2\cdot E_7(q)\cdot 2\leq 2\cdot E_7(q^2)$. Let $Z$ be the centre of $H$, with non-trivial element $z$.

Let $T$ denote a split torus of $H$, so of order $(q-1)^7$, with normalizer $N$. There is an element $e$ of order $4$ in $T$ whose centralizer is of type $A_7.2$, squaring to $z$ (so it has order $2$ in the adjoint group). Let $f$ be the longest element of the Weyl group in the adjoint group, viewed as an element of order $4$ in $N$ (so $f^2=z$). The group $V=\langle e,f\rangle$ is a Klein four group modulo $Z$, so a $Q_8$ subgroup in $H$, with normalizer $X=(2^2\times (D_4)_{\mathrm{ad}}).\Sym(3)$ in the algebraic group, so of the type we are looking for.

The largest soluble quotient of $X\cap G=N_G(V)$ is $2^4\rtimes \Sym(3)$, thought of as a subdirect product of two copies of $\Sym(4)$. It is easy to see, either with the presentation or a computer, that this has derived subgroup of index $2$. Hence $N_G(V)$ has a unique subgroup of index $2$. Thus either $N_H(V)=N_G(V)$ or $N_H(V)$ is this particular group. Notice that in this subgroup of index $2$, which has shape $(\langle e,f\rangle)\times \SO_8^+(q).2^2).3$, the elements $f$ and $ef$ are not conjugate modulo $Z$, whereas in $N_G(V)$ they are.

Note that if $g\in N_G(V)$ conjugates $ef$ to $f$ then $g\in C_G(e)$, and if $g\in C_G(e)$ conjugates $f$ to $ef$ then $(ef)^g=e^gf^g=e^2g\in V$. Thus $N_G(V)=N_H(V)$ if and only if $f$ and $ef$ are conjugate in $C_H(e)$. If we find some element $g$ of $C_G(e)$ that conjugates $f$ to $ef$, then $N_G(V)=N_H(V)$ if and only if $g\in H$.

We now find such an element. Since $e$ lies in a split torus of the algebraic group, there exists some element $g$ in the algebraic torus that squares to $e$, so has order $8$ in the algebraic simply connected group and $4$ in the adjoint group. Notice that $f$, as the longest element of the Weyl group, acts like inversion on the algebraic torus, so $g^f=g^{-1}$. Thus $f^{(g^{-1})}=gfg^{-1}=g^2f=ef$. Also, as $g^2=e$, certainly $g$ centralizes $e$. So this element works.

If $q\equiv 1\bmod 8$ then $T$ contains every element of the algebraic torus of order $8$, so $g\in C_H(e)$. If $q\equiv 5\bmod 8$ then $T$ contains no elements of order $8$, so $g\notin H$. However, $g^4=z$, so $g$ has order $4$ in the adjoint group, and so it \emph{does} lie in the adjoint version of $E_7(q)$. Hence $g\in C_G(e)$. Thus $N_G(V)=N_H(V)$ if $q\equiv 1\bmod 8$ and $|N_G(V):N_H(V)|=2$ if $q\equiv 5\bmod 8$.

Now suppose that $q\equiv 3\bmod 4$, let $H$ and $G$ be as before, let $X$ denote a copy of $E_7(q^2)$, and let $\sigma$ denote the standard Frobenius. As usual in groups of Lie type, there exists a maximal split torus $T_1$ of $X$ such that $T_1\cap H$ is a `$\Phi_2$-torus', i.e., there is a subgroup $(q+1)^7.W(E_7)$, analogous to the standard subgroup $(q-1)^7.W(E_7)$. The corresponding subgroup $V_1\leq T_1$ lies in $H$, and $N_G(V_1)\cong N_G(V)$ (since all such subgroups are conjugate in $G$). But now the analogous element $g_1$ still has order $8$, so $g_1\in C_H(e_1)$ if and only if $q\equiv -1\bmod 8$, but still lies in $C_G(e_1)$ in all cases.

This $N_G(V)=N_H(V)$ if and only if $q\equiv \pm1 \bmod 8$.

\bibliographystyle{amsplain}

\begin{thebibliography}{10}

\bibitem{andietrich2016}
Jianbei An and Heiko Dietrich, \emph{Maximal $2$-local subgroups of
  ${E}_7(q)$}, J.\ Algebra \textbf{445} (2016), 503--536.

\bibitem{bbr2015}
John Ballantyne, Chris Bates, and Peter Rowley, \emph{The maximal subgroups of
  {$E_7(2)$}}, LMS J. Comput. Math. \textbf{18} (2015), 323--371.

\bibitem{borovik1989}
Alexandre Borovik, \emph{Structure of finite subgroups of simple algebraic
  groups}, Algebra i Logika \textbf{28} (1989), 249--279 (Russian), English
  transl., Algebra and Logic \textbf{28} (1989), no. 3, 163--182 (1990).

\bibitem{bourbakilie2}
Nicolas Bourbaki, \emph{Lie groups and {L}ie algebras. {C}hapters 4--6},
  Elements of Mathematics, Springer-Verlag, Berlin, 2002, translated from the
  1968 French original by Andrew Pressley.

\bibitem{bhrd}
John Bray, Derek Holt, and Colva Roney-Dougal, \emph{The maximal subgroups of
  low-dimensional finite classical groups}, London Mathematical Society Lecture
  Note Series, no. 407, Cambridge University Press, Cambridge, 2013.

\bibitem{clss1992}
Arjeh Cohen, Martin Liebeck, Jan Saxl, and Gary Seitz, \emph{The local maximal
  subgroups of exceptional groups of {L}ie type, finite and algebraic}, Proc.
  London Math. Soc. (3) \textbf{64} (1992), 21--48.

\bibitem{cohenwales1997}
Arjeh Cohen and David Wales, \emph{Finite subgroups of {$F_4(\mathbb{C})$} and
  {$E_6(\mathbb{C})$}}, Proc.\ London Math. Soc. (3) \textbf{74} (1997),
  105--150.

\bibitem{atlas}
John Conway, Robert Curtis, Simon Norton, Richard Parker, and Robert Wilson,
  \emph{Atlas of finite groups}, Oxford University Press, Eynsham, 1985.

\bibitem{cooperstein1981}
Bruce Cooperstein, \emph{Maximal subgroups of {$G_2(2^n)$}}, J. Algebra
  \textbf{70} (1981), 23--36.

\bibitem{craven2017}
David~A.\ Craven, \emph{Alternating subgroups of exceptional groups of {L}ie type},
  Proc. Lond. Math. Soc. \textbf{115} (2017), 449--501.

\bibitem{craven2015un}
\bysame, \emph{Maximal ${\PSL_2}$ subgroups of exceptional groups of Lie type}, Mem. Amer. Math. Soc. \textbf{276} (2022),
no. 1355, vi+155pp.

\bibitem{craven2019un}
\bysame, \emph{On medium-rank {L}ie primitive and maximal subgroups of
  exceptional groups of {L}ie type}, Mem. Amer. Math. Soc., \text{288} (2023), no. 1434, v+213pp.

\bibitem{craven2020un}
\bysame, \emph{The maximal subgroups of the exceptional groups $F_4(q)$,  $E_6(q)$ and ${}^2\!E_6(q)$ and related almost simple groups}, Invent. Math. \textbf{234} (2023), 637--719.

\bibitem{frey1998}
Darrin Frey, \emph{Conjugacy of {$Alt_5$} and {$SL(2,5)$} subgroups of
  {$E_6(\mathbb{C})$}, {$F_4(\mathbb{C})$}, and a subgroup of
  {$E_8(\mathbb{C})$} of type {$A_2E_6$}}, J. Algebra \textbf{202} (1998),
  414--454.

\bibitem{frey2016}
\bysame, \emph{Embeddings of ${A}lt_n$ and its perfect covers for $n\geq 6$ in
  exceptional complex {L}ie groups}, J. Algebra \textbf{451} (2016), 1--45.

\bibitem{griessryba1994}
Robert Griess and Alexander Ryba, \emph{Embeddings of $U_3(8)$, $Sz(8)$ and the Rudvalis group in algebraic groups of type $E_7$}, Invent.\ Math.\ \textbf{116} (1994), 215--241.

\bibitem{griessryba1998}
\bysame, \emph{Embeddings of {$PGL_2(31)$} and
  {$SL_2(32)$} in {$E_8(\mathbb{C})$}}, Duke. Math. J. \textbf{94} (1998),
  181--211.

\bibitem{griessryba2002}
\bysame, \emph{Embeddings of {$\mathrm{SL}(2,27)$} in complex exceptional
  algebraic groups}, Michigan Math. J. \textbf{50} (2002), 89--99.

\bibitem{griessryba2002a}
\bysame, \emph{Classification of finite quasisimple groups which embed in exceptional algebraic groups}, J.\ Group Theory \textbf{5} (2002), 1--39.

\bibitem{hassanabadi1978} A.\ Mohammadi Hassanabadi, \emph{Automorphisms of permutational wreath products}, J.\ Aust.\ Math.\ Soc.\ \textbf{26} (1978), 198–208.

\bibitem{houghton1975} Chris Houghton, \emph{Wreath products of groupoids}, J.\ London Math.\ Soc.\ \textbf{10} (1975), 179–188.
\bibitem{abc}
Christoph Jansen, Klaus Lux, Richard Parker, and Robert Wilson, \emph{An atlas
  of {B}rauer characters}, Oxford University Press, New York, 1995.

\bibitem{kleidman1988}
Peter Kleidman, \emph{The maximal subgroups of the {C}hevalley groups
  {$G_2(q)$} with {$q$} odd, the {R}ee groups {$^2G_2(q)$}, and their
  automorphism groups}, J.\ Algebra \textbf{117} (1988), 30--71.

\bibitem{kmr1999}
Peter Kleidman, Ulrich Meierfrankenfeld and Alexander Ryba, $HS<E_7(5)$. J.\ London Math.\ Soc.\ (2) \textbf{60} (1999), 95--107.

\bibitem{kmr2000}
\bysame, $Ru<E_7(5)$. Comm.\ Algebra \textbf{28} (2000), 3555--3583.

\bibitem{kleidmanryba1993}
Peter Kleidman and Alexander Ryba, \emph{Kostant's conjecture holds for {$E_7$}: {$L_2(37)<E_7(\mathbb{C})$}}, J.\ Algebra \textbf{161} (1993), 535--540.

\bibitem{korhonen2025}
Mikko Korhonen, \emph{Structure of an exotic $2$-local subgroup in ${E}_7(q)$}, J.\ Group Theory, to appear.

\bibitem{lawther1995}
Ross Lawther, \emph{Jordan block sizes of unipotent elements in exceptional
  algebraic groups}, Comm.\ Algebra \textbf{23} (1995), 4125--4156.

\bibitem{liebeckmartinshalev2005}
Martin Liebeck, Benjamin Martin, and Aner Shalev, \emph{On conjugacy classes of
  maximal subgroups of finite simple groups, and a related zeta function}, Duke
  Math.\ J.\ \textbf{128} (2005), 541--557.

\bibitem{liebecksaxl1987}
Martin Liebeck and Jan Saxl, \emph{On the orders of maximal subgroups of the
  finite exceptional groups of {L}ie type}, Proc.\ London Math.\ Soc.\ (3)
  \textbf{55} (1987), 299--330.

\bibitem{liebecksaxlseitz1992}
Martin Liebeck, Jan Saxl, and Gary Seitz, \emph{Subgroups of maximal rank in
  finite exceptional groups of {L}ie type}, Proc.\ London Math.\ Soc.\ (3)
  \textbf{65} (1992), 297--325.

\bibitem{liebeckseitz1990}
Martin Liebeck and Gary Seitz, \emph{Maximal subgroups of exceptional groups of
  {L}ie type, finite and algebraic}, Geom. Dedicata \textbf{35} (1990),
  353--387.

\bibitem{liebeckseitz1998}
\bysame, \emph{On the subgroup structure of exceptional groups of {L}ie type},
  Trans.\ Amer.\ Math.\ Soc.\ \textbf{350} (1998), 3409--3482.

\bibitem{liebeckseitz1999}
\bysame, \emph{On finite subgroups of exceptional algebraic groups}, J.\ reine
  angew. Math. \textbf{515} (1999), 25--72.

\bibitem{liebeckseitz2004}
\bysame, \emph{The maximal subgroups of positive dimension in exceptional
  algebraic groups}, Mem.\ Amer.\ Math.\ Soc. \textbf{169} (2004), no.~802,
  vi+227.

\bibitem{liebeckseitzbook}
\bysame, \emph{Unipotent and nilpotent classes in simple algebraic groups and
  {L}ie algebras}, Mathematical Surveys and Monographs, vol. 180, American
  Mathematical Society, Providence, RI, 2012.

\bibitem{litterickmemoir}
Alastair Litterick, \emph{On non-generic finite subgroups of exceptional
  algebraic groups}, Mem.\ Amer.\ Math.\ Soc. \textbf{253} (2018), no.~1207,
  vi+156.

\bibitem{malcev1945}
Anatoly Mal'cev, \emph{Commutative subalgebras of semi-simple {L}ie algebras},
  Bull. Acad. Soc. URSS Ser. Math. \textbf{9} (1945), 291--300.

\bibitem{malletesterman}
Gunter Malle and Donna Testerman, \emph{Linear algebraic groups and finite
  groups of {L}ie type}, Cambridge University Press, 2011.

\bibitem{pacherathesis}
Andrea Pachera, \emph{Embeddings of $\PSL_2(q)$ in exceptional groups of Lie type over a field of characteristic $\neq 2,3$}, PhD thesis, University of Birmingham, 2021.

\bibitem{premetstewart2019}
Alexander Premet and David Stewart, \emph{Classification of the maximal
  subalgebras of exceptional {L}ie algebras over fields of good
  characteristic}, J. Amer. Math. Soc. \textbf{32} (2019), 965--1008.

\bibitem{seitz1991}
Gary Seitz, \emph{Maximal subgroups of exceptional algebraic groups}, Mem.\
  Amer.\ Math.\ Soc. \textbf{90} (1991), no.~441, iv+197.

\bibitem{serre1996}
Jean-Pierre Serre, \emph{Exemples de plongements des groupes
  {${\mathrm{PSL}}_2({\mathrm{F}}_p)$} dans des groupes de {L}ie simples},
  Invent. Math. \textbf{124} (1996), 525--562.

\bibitem{stewart2013}
David Stewart, \emph{The reductive subgroups of {$F_4$}}, Mem.\ Amer.\ Math.\
  Soc. \textbf{223} (2013), no.~1049, vi+88.

\bibitem{stewart2016}
\bysame, \emph{On the minimal modules for exceptional {L}ie algebras: {J}ordan
  blocks and stabilizers}, LMS J. Comput. Math. \textbf{19} (2016), 235--258.

\bibitem{thomas2016}
Adam Thomas, \emph{Irreducible ${A}_1$ subgroups of exceptional algebraic
  groups}, J. Algebra \textbf{447} (2016), 240--296.

\bibitem{thomas2020}
\bysame, \emph{The irreducible subgroups of exceptional algebraic groups}, Mem.
  Amer. Math. Soc. \textbf{268} (2020), no.~1307, vi+191.

\bibitem{wilsonrob}
Robert Wilson, \emph{The finite simple groups}, Graduate Texts in Mathematics,
  vol. 251, Springer-Verlag, London, 2009.

\end{thebibliography}
\providecommand{\bysame}{\leavevmode\hbox to3em{\hrulefill}\thinspace}
\providecommand{\MR}{\relax\ifhmode\unskip\space\fi MR }
\providecommand{\MRhref}[2]{%
  \href{http://www.ams.org/mathscinet-getitem?mr=#1}{#2}
}
\providecommand{\href}[2]{#2}

\end{document}